\documentclass[10pt]{amsart}

\usepackage{amsmath}
\usepackage{amssymb, bm}
\usepackage{amscd}
\usepackage{amsthm}
\usepackage{latexsym} 
\usepackage[latin1]{inputenc}

\newtheorem{theorem}{Theorem}[subsection]
\newtheorem{corollary}[theorem]{Corollary}
\newtheorem{lemma}[theorem]{Lemma}
\newtheorem{proposition}[theorem]{Proposition}

\numberwithin{equation}{section}

\newcommand{\Hom}{\operatorname{Hom}}

\def\N{\mathbb{N}}
\def\Z{\mathbb{Z}}

\def\Q{\mathbb{Q}}
\def\R{\mathbb{R}}
\def\Rp{\mathbb{R}_+}
\def\Rpa{\mathbb{R}_+^\ast}
\def\C{\mathbb{C}}

\newcommand{\cB}{{\mathcal  B}}

\newcommand{\cE}{{\mathcal  E}}

\newcommand{\cF}{{\mathcal  F}}

\newcommand{\cH}{{\mathcal  H}}

\newcommand{\cL}{{\mathcal L}}

\newcommand{\cN}{{\mathcal N}}

\newcommand{\cO}{{\mathcal O}}
\newcommand{\cP}{{\mathcal P}}

\newcommand{\cS}{{\mathcal S}}
\newcommand{\cT}{{\mathcal T}}

\newcommand{\cX}{{\mathcal X}}

\newcommand{\OK}{{{\mathcal O}_K}}

\newcommand{\Spec}{\operatorname{Spec}}

\newcommand{\Res}{{\rm Res}}

\newcommand{\covol}{\operatorname{covol}}

\newcommand{\cover}{R_{\rm cov}}

\newcommand{\Eb}{{\overline E}}
\newcommand{\Fb}{{\overline F}}

\newcommand{\Lb}{{\overline L}}

\newcommand{\cOb}{\overline{\mathcal{O}}}

\newcommand{\dega}{\,\widehat{\rm deg }\,}

\newcommand{\mua}{\widehat{\mu }}

\newcommand{\rk}{{\rm rk \,}}

\newcommand{\supp}{{\rm supp\,}}

\newcommand{\hot}{{h^0_{\theta}}}

\newcommand{\hon}{{h^0_{\rm{Ar}}}} %{{h^0_{\rm{naive}}}}
\newcommand{\honm}{{h^0_{\rm{Ar}-}}} %{{h^0_{\rm{naive}-}}}
\newcommand{\hont}{{\tilde{h}^0_{\rm{Ar}}}}
\newcommand{\hut}{{h^1_{\theta}}}

    \renewcommand{\Re}{\,{\rm Re}\,}

\newcommand{\ra}{\rightarrow}
\newcommand{\lrasim}{\stackrel{\sim}{\longrightarrow}}
\newcommand{\lra}{\longrightarrow}
\newcommand{\hra}{\hookrightarrow}

\newcommand{\hlra}{{\lhook\joinrel\longrightarrow}}

\renewcommand{\phi}{\varphi}

\renewcommand{\epsilon}{\varepsilon}

\newcommand{\Hm}{{H_{\rm min}}}
\newcommand{\Hmeta}{{H_{{\rm min},\eta}}}
\newcommand{\cEm}{{{\mathcal E}_{\rm min}}}

\newcommand{\CTnb}{{\overline{\rm CT}_n}}

\date{\today}

\title[Euclidean lattices, Theta invariants, and thermodynamic formalism]{Euclidean lattices, Theta invariants,\\ and thermodynamic formalism} %\titlerunning{Diophantine Approximation and Algebraization} %for an abbreviated version of
\author{Jean-Beno\^{\i}t Bost}
\address{Jean-Beno\^{\i}t  Bost, D{é}partement de Math{é}matique, Universit{é}
Paris-Sud,
B{â}timent 307, 91405 Orsay cedex, France}
\email{jean-benoit.bost@math.u-psud.fr}

\begin{document}

\begin{abstract} These are the notes of lectures delivered at Grenoble's summer school on \emph{Arakelov Geo\-me\-try and Diophantine Applications}, in June 2017. They constitute an introduction to the study of Euclidean lattices and of their invariants defined in terms of theta series.

Recall that Euclidean lattice is defined as a pair  
$\overline{E}:= (E, \Vert .\Vert)$
where $E$ is  some free $\mathbb{Z}$-module of finite rank $E$ and $\Vert. \Vert$ is some Euclidean norm on the real vector space $E_\mathbb{R} := E \otimes \mathbb{R}$. 
The most basic of these invariants is the non-negative real number:
$$h^0_\theta(\Eb) := \log \sum_{v \in E} e^{- \pi \Vert v \Vert^2}.$$

In these notes, we explain how such invariants naturally arise when one investigates basic questions concerning classical invariants of Euclidean lattices, such as their successive minima, their covering radius, or the number of lattice points in balls of a given radius.

We notably discuss their significance from the perspective of Arakelov geometry and of the analogy between number fields and function fields, their role (discovered by Banaszczyk) in the derivation of optimal transference estimates, and their interpretation in terms of the formalism of statistical thermodynamics.

These notes have been  primarily written for an audience of arithmetic geometers, but should also be suited  to a wider circle of mathematicians and theoretical physicists with some interest in Euclidean lattices or in the mathematical foundations of statistical physics.

\end{abstract}

\maketitle

\tableofcontents

\setcounter{section}{-1}

\section{Introduction}

{0.1.} My talks during the summer school \emph{Arakelov Geoemetry and Diophantine Applications}  were devoted to the formalism of infinite dimensional vector bundles over arithmetic curves and to the properties of their theta invariants studied in the monograph 
\cite{BostTheta}, and to some Diophantine applications of this formalism.

In these notes, I will focus on the content of the first of these lectures, where I discussed various motivations for considering the theta invariants of (finite dimensional) hermitian vector bundles over arithmetic curves, notably of Euclidean lattices.

Recall that a Euclidean lattice is defined as a pair  
$$\Eb := (E, \Vert .\Vert),$$
where $E$ is  some free $\Z$-module of finite rank $E$ and $\Vert. \Vert$ is some Euclidean norm on the real vector space $E_\R := E \otimes \R$. The theta invariants of $\Eb$ are  invariants defined by means of the theta series 
\begin{equation}\label{firstthetadef}
\sum_{v \in E} e^{- \pi t \Vert x - v\Vert^2},
\end{equation}
where $(t,x)$ belongs to $\Rpa \times E,$ and of  its special values. The most basic of these is the non-negative real number: 
$$\hot(\Eb) := \log \sum_{v \in E} e^{- \pi \Vert v \Vert^2}.$$

My purpose in these notes is to explain how they naturally arise when one investigates diverse basic questions concerning classical invariants of Euclidean lattices, such as their successive minima, their covering radius, or the number of lattice points in balls of a given radius.

\medskip

{0.2.} The first part of these notes consists in a self-contained introduction to the study of Euclidean lattices.

In Section \ref{Euclatt}, we recall some basic definitions concerning Euclidean lattices and their basic invariants. We also introduce some less classical, although elementary, notions concerning Euclidean lattices, such as the admissible short exact sequences of Euclidean lattices. These notions naturally arise from the perspective of Arakelov geometry, but  do not appear in classical introductions to Euclidean lattices. However this formalism should be appealing to geometrically minded readers, as it is specifically devised to emphasize the formal similarities between Euclidean lattices and vector bundles over varieties.

In Section \ref{RedEuc}, we discuss, in a simple guise, a central topic of the classical theory of Euclidean lattices, the so-called \emph{reduction theory}. This will demonstrate the flexibility of the ``geometric formalism" of Euclidean lattices previously introduced, and also  exemplify one of the main features of the classical theory of Euclidean lattices: the occurence, in diverse inequalities relating their classical invariants, of constants depending of the rank $n$ of the Euclidean lattices under study.

\medskip

{0.3.} The precise dependence on $n$ of these constants is a formidable problem --- already determining their asymptotic behavior when $n$ grows to infinity is often delicate --- and their occurence is a nuisance, both from a formal or aesthetic perspective and in applications, notably to Diophantine geometry. The use of more sophisticated invariants attached to Euclidean lattices, such as their  slopes \emph{\`a la} Stuhler-Grayson or their theta invariants, appears as a natural remedy to these difficulties.

In these notes, we focus on the theta invariants, and the reader is refered to the survey article \cite{Bost2018} for a discussion of these non-classical invariants with more emphasis on the role of slopes. Our aim in the second part of these notes will be  to convince the reader of the significance of the theta invariants when investigating Euclidean lattices, by giving accessible presentations of diverse results involving their classical invariants, in the derivation or in the statement of which theta invariants play a key role.\footnote{There is some overlap between Sections 1--3 and 5 of \cite{Bost2018} and Sections 1-4 of these lecture notes. The remaining sections of  \emph{loc. cit.}, devoted to some remarkable recent  results  of Regev, Dadush, and Stephens--Davi\-dowitz  (\cite{DadushRegev2016}, \cite{RSD17b}), provide some further illustrations of the relevance of theta invariants in the proofs of estimates relating invariants of Euclidean lattices.}

In Section \ref{ThetaBan}, after discussing some basic properties of the theta series (\ref{firstthetadef}), we give an introductory account of their use in the seminal article of  Banaszczyk \cite{Banaszczyk93}  for deriving \emph{transference estimates} --- namely, estimates comparing some classical invariants of some Euclidean lattice $\Eb$ and of its dual $\Eb^\vee$ --- where the involved constants depending of $n := \rk \Eb$ are basically optimal.

In Section \ref{AnaVect}, we discuss the occurrence of  theta invariants of Euclidean lattices from a completely different perspective, namely when developing the classical analogy between number fields and function fields. In this analogy,  the theta invariant
$\hot(\Eb)$ % := \log \sum_{v \in E} e^{- \pi \Vert v \Vert^2}$$
attached to some Euclidean lattice $\Eb$ %, defined by some free $\Z$-module of finite rank $E$ and some Euclidean norm an the real vector space $E_\R := E \otimes \R$,  
appears as an arithmetic counterpart of the dimension
$$h^0(C, E) := \dim_k \Gamma (C, E)$$
of the $k$-vector space of sections $\Gamma(C, E)$ of some vector bundle $E$ over a smooth projective geometrically irreducible
curve $C$ over some field $k$.
The similarities between $\hot(\Eb)$ and $h^0(C, E)$ may actually be pursued to a striking level of precision, and we survey several of them at the end of Section \ref{AnaVect}.

It turns out that, when dealing with the analogy between number fields and function fields, besides the invariant $\hot(\Eb)$ attached to some Euclidean lattice $\Eb := (E, \Vert .\Vert)$, 
   one also classically considers the non-negative real number
$$\hon(\Eb) := \log \vert \{ v \in E \mid \Vert v \Vert \leq 1 \} \vert$$
--- simply defined in terms of the number of lattice points in the unit ball of $(E_\R, \Vert.\Vert)$ ---
 as an arithmetic counterpart of $h^0(C,E).$

The coexistence of two distinct invariants playing the role of an arithmetic counterpart of the basic geometric invariant $h^0(C,E)$ is intriguing. This puzzle has been solved in \cite{BostTheta}, Chapter 3, in two ways. Firstly, by establishing some comparison estimate, bounding the difference $\hot(\Eb)-\hon(\Eb)$ in terms of the rank of $E$, by means of  Banaszczyk's methods  discussed in Section \ref{ThetaBan}. And secondly, by showing that the theta invariant $\hot(\Eb)$ are related, by Fenchel-Legendre transform, to some ``stable variant" $\hont(\Eb)$ of the invariant $\hon(\Eb)$  defined in  terms of lattice point counting in the direct sums
$$\Eb^{\oplus n} := \Eb \oplus \ldots \oplus \Eb  \quad \mbox{ ($n$-times)} $$
of copies of the Euclidean lattice $\Eb$, when the integer $n$ goes to $+\infty$.

We present these relations between $\hot(\Eb)$, $\hon(\Eb),$ and $\hont(\Eb)$ with some details in Subsection~\ref{Reconcile}.  

 \medskip
{0.4.} The ``Legendre duality" between  $\hont(\Eb)$ and $\hot(\Eb)$ provides another striking motivation for considering the  theta invariant $\hot(\Eb)$. Somewhat surprisingly, this duality holds in a much more general context. It is indeed  a  special case of some general measure theoretic results, concerning a measure space $\cE$ equipped with some measurable function $H$ with values in $\Rp$, that describes the asymptotic behavior of the measure of the subsets
\begin{equation}\label{HnE}
\{(x_1, \ldots, x_n) \in \cE^n \mid H(x_1) + \ldots + H(x_n) \leq n E \}
\end{equation}
of $\cE^n$ when $n$ goes to $+\infty$, for a given value of $E \in \Rp$. These measure theoretic results are actually closely related to the formalism of statistical thermodynamics.

The proof of these general measure theoretic results is arguably clearer than its specialization to the invariants $\hont(\Eb)$ and $\hot(\Eb)$ associated to some Euclidean lattice $\Eb:= (E, \Vert.\Vert)$\footnote{This specialization arises from taking the measure space $\cE$ to be the set $E$ of lattice points of $\Eb$ equipped with the counting measure, and the function $H$ to be a multiple of $\Vert. \Vert^2$.}.  In \cite{BostTheta}, these results were established by reduction to 
 some classical theorems of the theory of large deviation. Moreover their relations with the thermodynamic formalism  was only alluded to. In the third  part of these notes, we provide  a  self-contained presentation of these results, accessible  with some basic knowledge of measure theory and of the theory of analytic functions only (say, at the level of Rudin's classical textbook \cite{Rudin87}). Our presentation also includes a discussion of the physical signification of these results and of their relations with some classical techniques to derive estimates in probability and analytic number theory. 
 
 In Section \ref{Thermod}, we state our general measure theoretic theorem (Theorem \ref{ThMain}) and we discuss its interpretation in statistical physics and its application to the  invariants $\hont(\Eb)$ and $\hot(\Eb)$ of Euclidean lattices. 
 
 In Section \ref{PfMain}, we give a proof of Theorem \ref{ThMain} that uses a few basic notions of measure theory only. The key point of this proof is a variation on a classical proof of Cram\'er's theorem, the starting point of the theory of large deviations. 
 
 Section \ref{Compl} is devoted to  some complements to Theorem \ref{ThMain} and its proof. Notably, we present Lanford's approach to the study of the asymptotic behavior of the measure of the sets (\ref{HnE}) when $n$ grows to infinity. We also discuss a mathematical interpretation of the second law of thermodynamics in our formalism and its application to Euclidean lattices. 
 
 Finally, in Section \ref{PDF}, we give an alternative derivation of the main assertion  of Theorem \ref{ThMain}, which originates in the works of Poincar\'e (\cite{Poincare12}) and of Darwin and Fowler (\cite{DarwinFowler1922},
\cite{DarwinFowler1922II},
\cite{DarwinFowler1923}). Instead of arguments from measure and probability theory,  it relies on the theory of analytic functions and  on the use of the saddle-point method. 

We hope that this presentation will be suited to the arithmetically minded mathematicians for which the summer school was devised, and also to to a wider circle of mathematicians and theoretical physicists with some interest in Euclidean lattices or in the mathematical foundations of statistical physics. 

\medskip{0.5} During the preparation of these notes, I benefited from the support of the ERC project AlgTateGro, supervised by Fran\c{c}ois Charles (Horizon 2020 Research and Innovation Programme, grant agreement No 715747).

\section{Euclidean lattices}\label{Euclatt}

\subsection{Un peu d'histoire}

Let $V$ be a finite dimensional vector space. A \emph{lattice} $\Lambda$ in $V$ is a discrete subgroup of $V$ such that the quotient topological group $V/\Lambda$ is compact, or equivalently, such that there exists some $\R$-basis $(e_i)_{1 \leq i \leq n}$ of $V$ such that $\Lambda = \bigoplus_{i=1}^n \Z e_i.$ The $\R$ vector space $V$ is then canonically isomorphic to $\Lambda_\R := \Lambda \otimes \R.$

A \emph{Euclidean lattice} is the data $(V, \Lambda, \Vert .\Vert)$ of some finite dimensional $\R$-vector space $V$, equipped with some Euclidean norm $\Vert.\Vert$, and of some lattice $\Lambda$ in $V$.

Equivalently, it is the data 
$$\Eb := (E, \Vert.\Vert)$$
of some free $\Z$-module of finite rank $E$, and of some Euclidean norm $\Vert.\Vert$ on the $\R$-vector space $E_\R:= E \otimes \R.$ (The $\Z$-module $E$ will always be identified to its image by the injective morphism $(E\hra E_\R, v \mapsto v \otimes 1).$ This image is a lattice in $E_\R$.)

Three-dimensional Euclidean lattices constitute a mathematical model for the spatial organization  of atoms or molecules in a crystal and for this reason have been investigated  since the seventeenth century (notably by Huyghens in his \emph{Trait\'e de la lumi\`ere}, published in 1690). At the end of the eighteenth  century, the development of number theory led to the study of Euclidean lattices in a purely mathematical perspective: Lagrange, in his work on integral quadratic forms in two variables, considered two-dimensional Euclidean lattices and their reduction properties; the investigation of integral quadratic forms in an arbitrary number of indeterminates led Gauss and then Hermite to study Euclidean lattices of rank three, and then of arbitrary rank.

At the beginning of the twentieth century, the study of Euclidean lattices had become a full fledged domain of pure mathematics, after major contributions of Korkin, Zolotarev, Minkowski (who introduced the terminology of \emph{geometry of numbers} for the study of triples  $(V, \Lambda, \Vert .\Vert)$ as above, with the norm $\Vert.\Vert$ non necessarily Euclidean), and Voronoi. We refer the reader to the books and surveys articles  
\cite{Cassels71}, \cite{RyshkovBaranovskii79}, \cite{Lagarias95},  and \cite{Martinet2003} for presentations of the classical results of this theory.

\subsection{The classical invariants of Euclidean lattices}\label{ClassInv} 

There is an obvious notion of \emph{isomorphism} between Euclidean lattices: an isomorphism between $\Eb_1:= (E_1, \Vert.\Vert_1$ and $\Eb_2:= (E_2, \Vert.\Vert_2)$ is an isomorphism $\phi: E_1 \lrasim E_2$ of $\Z$-modules such that the attached isomorphism of $\R$-vector spaces $\phi_\R: E_{1,\R} \lrasim E_{2,\R}$ is an isometry between %the Euclidean $\R$ vector spaces
  $(E_{1,\R}, \Vert.\Vert_1)$ and $(E_{2,\R}, \Vert.\Vert_2)$.

To some Euclidean lattice $\Eb := (E, \Vert.\Vert)$ are classically attached the following invariants, which depend only of its isomorphism class:
\begin{itemize}
\item its  \emph{rank}:
$$\rk E  = \dim_\R E_\R \in \N;$$

\item its \emph{covolume}: if $m_{\Eb}$ denotes the Lebesgue measure\footnote{It is defined as the unique translation invariant Borel measure on  $E_\R$ such that $m_\Eb (\sum_{i=1}^n [0,1[ v_i) = 1$ for any orthonormal basis $(v_i)_{1 \leq i \leq n}$ of the Euclidean vector space $(E_\R, \Vert. \Vert)$. An equivalent normalization condition is the following one:  $\int_{E_\R} e^{- \pi \Vert x \Vert^2} \, dm_{\Eb}(x) = 1.$} on the Euclidean vector space $(E_\R, \Vert. \Vert)$ and if  $\Delta$ is a fundamental domain fondamental\footnote{Namely, a Borel subset of $E_\R$ such that $(\Delta + e)_{e \in E}$ is a partition of $E_\R$. One easily establishes that such a fundamental domain $\Delta$ exists and that the measure $m_{\Eb}(\Delta)$ does  not depend of the choice of $\Delta$.} for $E$ acting by translation on  $E_\R,$ the covolume of  $\Eb$ is defined as
$$\covol (\Eb) := m_{\Eb}(\Delta) \in \Rpa.$$
Observe that $\covol(\Eb) =1$ when $\rk E = 0.$ 

\item its \emph{first minimum}, when $\rk E > 0$ :
$$\lambda_1(\Eb) := \min_{e \in E \setminus\{0\}} \Vert e \Vert \in \Rpa.$$
More generally, one defines the \emph{successive minima} $(\lambda_i(\Eb))_{1 \leq i \leq \rk E}$ of $\Eb$ by:% pour tout $i \in \{ 1, \ldots, \rk E\},$ on définit le \emph{$i$-ième minimum de $\Eb$}:
$$\lambda_{i}(\Eb) := \min \left\{r \in \R_+ \mid E \cap \overline{B}_{\Vert. \Vert}(0, r) \mbox{ contains $i$ $\R$-linearly independant elements} \right\},$$
where $\overline{B}_{\Vert. \Vert}(0, r)$ denotes the closed ball of center $0$ and radius $r$ in the Euclidean vector space $(E_\R, \Vert. \Vert)$.

\item its \emph{covering radius}, when $\rk E > 0$:
$$\cover (\Eb) := \max_{x \in E_\R} \min_{e \in E} \Vert x - e \Vert = \min \{ r \in \R_+ \mid E + \overline{B}_{\Vert. \Vert}(0, r) = E_\R\}.$$  
\end{itemize}

Many  results of the theory of  Euclidean lattices may be stated as inequalities relating these divers invariants.

For instance, a classical results, which goes back to Hermite and plays a central role in algebraic theory of numbers, is the following estimate for the first minimum of some Euclidean lattice in terms of its covolume:

\begin{theorem}[Hermite, Minkowski] \label{ThHermite}
For any integer $n >0,$ there exists $C(n)$ in  $\R_+^\ast$ such that, for any Euclidean lattice $\Eb$ of rank $n,$
\begin{equation}\label{Hermiteineq}
\lambda_1(\Eb) \leq C(n) (\covol (\Eb))^{1/n}.
\end{equation}

If we denote the Lebesgue measure of the unit ball in    $\R^n$ by $v_n$, this holds with:
\begin{equation}\label{Minkowskiineq}
C(n) = 2 v_n^{-1/n}.
\end{equation}
\end{theorem}

Since $v_n = \pi^{n/2}/\Gamma(n/2+1),$ it follows from Stirling's formula that, when $n$ goes to $+\infty$, this value of $C(n)$ admits the following asymptotics:
\begin{equation}\label{vnas}
2 v_n^{-1/n}  \sim \sqrt{2n /e \pi}.
\end{equation}

Hermite  has proved this theorem by induction on the rank $n$, by developing what is known as \emph{reduction theory} for Euclidean lattices of arbitrary rank. We present a modernized version of Hermite's arguments in Section \ref{RedEuc} below. These arguments allowed him to establish the estimate  (\ref{Hermiteineq}) with
$$C(n) = (4/3)^{(n-1)/2}.$$
(see Subsection  \ref{sussecred}, \emph{infra}).

 In his  \emph{Geometrie der Zahlen} (\cite{Minkowski1896}, p. 73-76), Minkowski has given  a new elegant proof of Hermite's estimate which leads to the value (\ref{Minkowskiineq}) for  $C(n)$ and admits a simple physical interpretation. Let us think of the Euclidean lattice as a model for a crystal in the $n$-dimensional Euclidean space $(E_\R, \Vert.\Vert)$: the molecules in this crystal are represented by the points of the lattice $E$. As the open balls $\mathring{B}_{\Vert.\Vert}(v, \lambda_1(\Eb)/2)$ of radius $\lambda_1(\Eb)/2$ centered at these points are pairwise disjoint, the density of the crystal --- defined as the number of its molecules per unit of volume --- is at most the inverse of the volume of any of these balls,  which is
$$v_n (\lambda_1(\Eb)/2)^n.$$
This density is nothing but the inverse of the  covolume of $\Eb$.  Therefore:
$$\covol (\Eb)^{-1} \leq [v_n (\lambda_1(\Eb)/2)^n]^{-1}.$$
This estimate is precisely (\ref{Hermiteineq}) with $C(n)$ given by (\ref{Minkowskiineq}).

Similarly, by observing that the ball  $\overline{B}_{\Vert.\Vert}(0, \cover(\Eb))$ contains some fundamental domain for the action of  $E$ over $E_\R$, we obtain:
$$v_n \cover(\Eb)^n \geq \covol (\Eb),$$
or equivalently:
\begin{equation}\label{Rcovmin}
\cover(\Eb) \geq v_n^{-1/n} \, \covol (\Eb)^{1/n}.
\end{equation}

The square $\gamma_n = C(n)^2$ of the best constant in Hermite's inequality (\ref{Hermiteineq}) is classically known as the \emph{Hermite's constant}. Its exact value is known for small values of $n$ only (see \cite{ConwaySloane1999}, \cite{CohnKumar2007}). However Minkowski has proved that the asymptotic estimate $\gamma_n = O({n}),$ which follows  (\ref{vnas}), is essentially optimal --- namely, when $n$ goes to $+\infty,$, we have: $$ \log \gamma_n = \log n + O(1).$$
By comparison, Hermite's arguments based on reduction theory lead to the weaker estimate:%Par comparaison, l'argument de `` théorie de la réduction "\, d'Hermite prouvait seulement la majoration :
$$\log \gamma_n \leq  (n-1) \log (4/3).$$

The previous discussion exemplifies a major theme of the theory of Euclidean lattices, since the investigation by Hermite and his followers   Korkin and  Zolotarev of Euclidean lattices of arbitrary rank: the investigation of the ``best constants" appearing in the estimates relating invariants of Euclidean lattices, and notably the determination of their asymptotic behavior when this rank goes to infinity.

\subsection{Euclidean lattices as Hermitian vector bundles over $\Spec \Z$}

In this paragraph, we introduce a few additional definitions concerning Euclidean lattices, which are less classical than the ones discussed in \ref{ClassInv} above, although they still are quite elementary. These definitions naturally arise from the perspective of Arakelov geometry, where Euclidean lattices occur as an instance of the so-called \emph{Hermitian vector bundles} over some regular $\Z$-scheme of finite type $\cX$, in the special case $\cX = \Spec \Z.$  

\subsubsection{Short exact sequences and duality}

Let  us consider some Euclidean lattice $\Eb:= (E, \Vert.\Vert)$.

For any $\Z$-submodule $F$ of  $E$, the inclusion morphism $F \hra E$ defines, by extension of scalars, a canonical injection $F_\R \hra E_\R.$ Equipped with the restriction to  $F_\R$ of the norm $\Vert.\Vert,$ the submodule $F$ (which is also a free  $\Z$-module of finite rank) defines some euclidean lattice:
$$\Fb := (F, \Vert.\Vert_{\mid F_\R}).$$

If moreover $F$ is  \emph{saturated} in $E$ --- namely, if the $\Z$-module $E/F$ is torsion-free, or equivalently, if  $F = F_\R \cap E$ --- then  $E/F$ is a free $\Z$-module of finite rank. Moreover the exact sequence
$$0 \lra F \stackrel{i}{\lra} E \stackrel{p}{\lra} E/F \lra 0$$
(where we denote by  $i$ and $p$ the inclusion and quotient morphisms) becomes, by extension of scalars, a short exact sequence of $\R$-vector spaces:
 $$0 \lra F_\R \stackrel{i_\R}{\lra} E_\R \stackrel{p_\R}{\lra} (E/F)_\R \lra 0.$$
 Accordingly the $\R$-vector space  $(E/F)_\R$ may be identified with the quotient of $E_\R$ by $F_\R$. In particular it may be equipped with the quotient Euclidean norm $\Vert.\Vert_{\rm quot}$ induced by the Euclidean norm $\Vert.\Vert$ on $E_\R$. This defines the Euclidean lattice 
 $$\overline{E/F} := (E/F, \Vert. \Vert_{\rm quot}).$$
 
With the previous notation, we shall say that the diagram%his construction iOn résumera souvent cette construction en disant que le diagramme
\begin{equation}\label{adm}
0 \lra \Fb \stackrel{i}{\lra} \Eb \stackrel{p}{\lra} \overline{E/F} \lra 0
\end{equation}  
is  an \emph{admissible short exact sequence} of Euclidean lattices.%de réseaux euclidiens. 

Let us observe that any saturated $\Z$-submodule  $F$ in  $E$ is determined by the $\R$-vector subspace $F_\R$ in $E_\R$, and also by the $\Q$-vector subspace $F_\Q:= F \otimes \Q$ of $E_\Q:= E \otimes \Q,$ since $$F = F_\R \cap E = F_\Q \cap E.$$
The map $(F \mapsto F_\Q)$ indeed establishes a bijection between the sets of saturated $\Z$-submodules of $E$ and of $\Q$-vector subspaces of $E_\Q$.  %Si $F$ est un sous-$\Z$-module (non nécessairement saturé) de $E$, on pose :
%$$F^{\rm sat} := F_\Q \cap E.$$ C'est le sous-module saturé de $E$ associé à $F_\Q$ par la bijection précédente. On a $F \subset F^{\rm sat}$ et $F^{\rm sat}/F$ est fini.

Besides, to any Euclidean lattice $\Eb := (E, \Vert.\Vert)$ is attached its \emph{dual Euclidean lattice} $$\Eb^\vee:= (E^\vee, \Vert.\Vert^\vee)$$ defined as follows. 

Its underlying $\Z$-module $E^\vee$ is the dual $\Z$-module dual of $E$,
$$E^\vee := \Hom_\Z(E,\Z),$$
which a free $\Z$-module of the same rank as $E$. The $\R$-vector space  $(E^\vee)_\R := E^\vee \otimes \R$ may be identified with $(E_\R)^\vee :=\Hom_\R(E_\R, \R);$ we shall denote it by $E_\R^\vee$. The Euclidean norm $\Vert.\Vert^\vee$ is defined as the norm dual of the norm $\Vert. \Vert$ on $E_\R$. In other words, for any $\xi \in E_\R^\vee,$ $$\Vert \xi \Vert^\vee := \max\{ \vert \xi(x) \vert; x \in \overline{B}_{\Vert.\Vert}(0,1) \}.$$ 
  
There is a canonical biduality isomorphism: $$\Eb \lrasim \Eb^{\vee \vee}.$$ Moreover an admissible short exact sequence \ref{adm}) of Euclidean lattices defines, by duality, a diagram
$$
0 \lra\overline{E/F}^\vee \stackrel{{}^t i}{\lra} \Eb^\vee \stackrel{{}^tp}{\lra} \Fb^\vee \lra 0
$$
which may be identified with the admissible short exact sequence  
$$
0 \lra \Fb^\perp {\lra} \Eb {\lra} \overline{E^\vee/F^\perp} \lra 0
$$
attached to the saturated $\Z$-submodule $$F^\perp := \{ \xi \in E^\vee \mid \xi_{\mid F} = 0 \}$$
in $E^\vee.$ Actually the map $(F \mapsto F^\perp)$ establishes a bijection between the set of saturated submodules of  $E$ and of $E^\vee$.

\subsubsection{Arakelov degree and slope}\label{Degap}

Instead of its covolume, it is often more convenient to use the \emph{Arakelov degree} of some Euclidean lattice $\Eb,$ defined as the logarithm of its  ``density"  $\covol(\Eb)^{-1}$:
\begin{equation} \label{degadef}
\dega \Eb := -\log \covol(\Eb),
\end{equation}
and, when $\rk E >0,$ its  \emph{slope}
\begin{equation} \label{muadef}
\mua(\Eb) := \frac{\dega \Eb}{\rk E} = \log (\covol(\Eb)^{-1/\rk E}).
\end{equation}

For instance, one easily sees that, for any admissible short exact sequence (\ref{adm}) of Euclidean lattices, the covolumes  of $\Eb,$ $\Fb$ and $\overline{E/F}$ satisfy :
\begin{equation}\label{multcovol}
\covol (\Eb) = \covol (\Fb) . \covol(\overline{E/F}).
\end{equation}
Consequently their Arakelov degrees satisfy the additivity property: 
\begin{equation}\label{degadd}
\dega \Eb = \dega \Fb + \dega \overline{E/F},
\end{equation}
similar to the one satisfied by their rank:$$\rk E = \rk F + \rk E/F.$$

In the same vein, the covolumes of some Euclidean lattice $\Eb$ and of its dual $\Eb^\vee$ satisfy the relation
$$\covol (\Eb^\vee) = \covol (\Eb)^{-1},$$
which may also be written as
$$\dega \Eb^\vee = - \dega \Eb.$$

\subsubsection{Operations on Euclidean lattices} 

The operations of direct sum and of tensor product on  $\Z$-modules on Euclidean $\R$-vector spaces allow one to define similar operations on Euclidean lattices.

For instance, if $\Eb_1:= (E_1, \Vert.\Vert_1)$ and  $\Eb_2:= (E_2, \Vert.\Vert_2)$ are two Euclidean lattices, we let: 
$$\Eb_1 \oplus \Eb_2 := (E_1 \oplus E_2, \Vert . \Vert_\oplus) \quad \mbox{et} \quad \Eb_1 \otimes \Eb_2 := (E_1 \otimes E_2, \Vert . \Vert_\otimes),$$
where the Euclidean norm $\Vert . \Vert_\oplus$ on  $(E_1 \oplus E_2)_\R \simeq E_{1, \R} \oplus E_{2,\R}$ is defined by 
$$ \Vert x_1 \oplus x_2 \Vert^2_\oplus := \Vert x_1 \Vert^2_1 + \Vert x_2 \Vert^2_2,$$ and where the norm $\Vert . \Vert_\otimes$ sur $(E_1 \otimes E_2)_\R \simeq E_{1, \R} \otimes_\R E_{2,\R}$ is characterized by the following property: for any orthonormal basis $(e_{1\alpha})_{1\leq \alpha \leq n_1}$ (resp. $(e_{2\beta})_{1\leq \beta \leq n_2})$  of the Euclidean space $(E_{1,\R}, \Vert.\Vert_1)$ (resp. of $(E_{2,\R}, \Vert.\Vert_2)$),  $(e_{1\alpha} \otimes e_{2\beta})_{1\leq \alpha, \beta \leq n_1, n_2}$  
is an orthonormal basis of  $(E_{1,\R}\otimes_\R E_{2,\R}, \Vert.\Vert_\otimes)$.

The canonical inclusion  $i: E_1 \lra E_1 \oplus E_2$ and projection $p: E_1 \oplus E_2 \lra E_2$ make the diagram
\begin{equation}\label{sommedir} 0 \lra \Eb_1 \stackrel{i}{\lra} \Eb_1 \oplus \Eb_2 \stackrel{p}{\lra} \Eb_2 \lra 0
\end{equation}
an admissible short exact sequence of Euclidean lattices
\footnote{One should 
beware that, in general,  an admissible short exact sequence of Euclidean lattice \emph{is not}
 isomorphic to an exact sequence of the form (\ref{sommedir}): the obstruction for the admissible short exact sequence 
(\ref{adm}) to be \emph{split}, that is isomorphic to an admissible short exact sequence of the form (\ref{sommedir}), is an element of some extension group attached to the Euclidean lattices $\Eb$ et $\overline{E/F}$, the properties of which are closely related to reduction theory; 
see \cite{BK10}.}.

 In particular, as a special case of (\ref{degadd}), we have:
$$\dega (\Eb_1 \oplus \Eb_2) = \dega \Eb_1 + \dega \Eb_2.$$

For any $t\in \R,$ we define the rank 1 Euclidean lattice 
$$\cOb(t) := (\Z, \Vert.\Vert_t),$$
where $\Vert.\Vert_t$ denotes the norm over $\Z_\R = \R$ defined by  
$$\Vert x \Vert_t := e^{-t} \vert x \vert.$$

It is straightforward that
$$\dega \cOb(t) = t$$
and that any Euclidean lattice $\Lb$ of rank $1$ is isomorphic to  $\cOb(t)$ où $t := \dega \Lb.$ 
Moreover, for any Euclidean lattice $\Eb:= (E, \Vert. \Vert)$, the tensor product $\Eb \otimes \cOb(t)$ may be identified to the Euclidean lattice  $(E, e^{-t} \Vert.\Vert)$, deduced from $\Eb$ by  ``scaling" its norm by  $e^{-t}$.

\subsubsection{Example: direct sums of Euclidean lattices of rank 1}\label{ExRk1}

The invariants of Euclidean lattices direct sums of rank 1 Euclidean lattices are easily computed.  Let us indeed consider the Euclidean lattice
$$\Eb := \bigoplus_{i= 1}^n \cOb(t_i),$$
for some positive integer $n$ est un entier $>0$, and a non-increasing sequence $t_1 \geq \cdots \geq t_n$ on $n$ real numbers. One easily computes:  
$$\dega \Eb = t_1+ \ldots + t_n \quad \mbox{and} \quad \mua (\Eb) = \frac{t_1+ \ldots + t_n}{n},$$
\begin{equation}\label{lambdasumi}\lambda_i(\Eb) = e^{-t_i} \quad \mbox{for any  $i \in \{1, \ldots,n\}$,}\end{equation}
and  
$$\cover (\Eb) = (1/2) (\sum_{i=1}^n e^{-2t_i})^{1/2}.$$
This notably implies:
\begin{equation}\label{Rtn}
\cover (\Eb) \in [(1/2) e^{-t_n}, (\sqrt{n}/2) e^{-t_n}].
\end{equation}

Besides,
$$\Eb^\vee \simeq \bigoplus_{i= 1}^n \cOb(-t_i),$$
and accordingly:
$$\lambda_i(\Eb^\vee) = e^{t_{n+1-i}} \quad \mbox{for every $i \in \{1, \ldots,n\}$.}$$

The relation (\ref{Rtn}) may therefore be written:
\begin{equation}\label{RlambdaE} \cover (\Eb)\,  \lambda_1(\Eb^\vee) \in [1/2, \sqrt{n}/2].
\end{equation}

\section{Reduction theory for  Euclidean lattices}\label{RedEuc}

In this section, using the geometric language introduced in the previous one, we present a basic result of reduction theory. Namely we show that that any Euclidean lattice $\Eb$ of rank $n$ may be ``approximated " by some Euclidean lattice that is the direct sum $\Lb_1 \oplus \ldots \oplus \Lb_n$ of  Euclidean lattices $\Lb_1, \cdots, \Lb_n$ of rank one, with an ``error" controlled in terms of $n$, and accordingly is approximately determined by the $n$ real numbers $\mu_i := \dega \Lb_i,$ $1 \leq i \leq n$ (see Theorem \ref{theored} \emph{infra} for a precise statement). 

The  derivation of this result in Paragraph \ref{sussecred} below is nothing but a reformulation of some classical arguments that go back to Hermite, Korkin and Zolotarev. But we believe that the geometric point of view used here --- notably the notion of admissible short exact sequences of Euclidean lattices --- makes these proof more transparent and demonstrates the conceptual interest of a more geometric approach. 

As the successive minima or the covering radius of the direct sum $\Lb_1\oplus \cdots \oplus \Lb_n$ and the dual Euclidean lattice $\Lb_1^\vee\oplus \cdots \oplus \Lb_n^\vee$ are simple functions of $(\mu_1, \cdots, \mu_n),$ our reduction theorem easily implies some ``transference inequalities" that relates the above invariants of some Euclidean lattice $\Eb$ and of its dual $\Eb^\vee$.

The fact that the properties of some Euclidean lattice $\Eb$ of rank $n$ are (approximately) controlled by the $n$ real numbers $(\mu_1, \cdots, \mu_n)$, already demonstrated by this basic discussion of reduction theory, is a forerunner of the role of the so-called \emph{slopes} $(\mua_1(\Eb), \ldots, \mua_n(\Eb))$, a non-increasing sequence of $n$ real numbers associated by Stuhler to the Euclidean lattice $\Eb$ (see \cite{Stuhler76} and \cite{Grayson84}).   We refer the reader to \cite{Bost2018} for a discussion and references concerning slopes of Euclidean lattices and recent advances on their properties.

\subsection{A theorem of Hermite, Korkin and Zolotarev}\label{sussecred}

In substance, the following theorem appears in some letters of Hermite to Jacobi (see \cite{Hermite1850}). A streamlined version of Hermite's arguments appears in the work of Korkin and Zolotarev (\cite{KorkinZolotarev1873}, p. 370-373), and we gave below a geometric rendering of their proof, using the formalism introduced in the previous section.

\begin{theorem}\label{theored}
For any positive integer $n$, there exists  $D(n) \in \Rpa$ such that, for any Euclidean lattice  $\Eb := (E, \Vert.\Vert)$ of rank $n,$ the $\Z$-module $E$ admits some $\Z$-base $(v_1, \ldots, v_n)$ such that
\begin{equation}\label{basered}
\prod_{i=1}^n \Vert v_i \Vert \leq D(n) \covol ({\Eb}).
\end{equation}

This indeed holds with :
\begin{equation}\label{Dn}
D(n) = (4/3)^{n(n-1)/2}.
\end{equation}
\end{theorem}

Observe also that, with the notation of Theorem  \ref{theored}, we immediately obtain:
$$\lambda_1(\Eb) \leq (\prod_{i=1}^n \Vert v_i \Vert)^{1/n} \leq D(n)^{1/n} \covol({\Eb})^{1/n}.$$
In this way, we recover Hermite's inequality  (\ref{Hermiteineq}), with 
$$C(n) = D(n)^{1/n} = (4/3)^{(n-1)/2}.$$
(Compare with \cite{Hermite1850},  pages  263--265 and 279--283) 

\begin{proof} The theorem is established by induction on the integer $n$.

Let  $\Eb$ be a Euclidean lattice of rank $n>0.$ Let us choose some element  $s \in E$ such that $\Vert s \Vert = \lambda_1(\Eb).$ The submodule  $\Z s$ is the saturated in $E$.

If $n=1,$ then $E = \Z s$. In this case, $$ \covol (\Eb) = \lambda_1(\Eb)$$ and the estimate (\ref{basered}) is satisfied by $v_1 :=s$ and $D(1) =1.$

When $n>1,$ we may consider the quotient Euclidean lattice  $$\overline{E/\Z s} := (E/\Z s, \Vert .\Vert_{\rm quot}),$$ of rank $n-1$. By induction, there exists some basis $(w_1, \ldots, w_{n-1})$ of $E/\Z s$ such that
\begin{equation}\label{baseredmoins}\prod_{i = 1}^{n-1} \Vert w_i \Vert_{\rm quot} \leq D(n-1) \covol(\overline{E/\Z s}). 
\end{equation}

If, for any  $i \in \{0,\dots, n-1\},$ we choose some element $v_i$ in the inverse image $p^{-1}(w_i)$ of $w_i$ by the quotient map $$p: E \lra E/\Z s$$ and if we let $v_n := s,$ then $(v_1, \ldots, v_n)$ is a $\Z$-basis of $E$. Moreover, for any  $i \in \{0,\ldots, n-1\},$ we may choose for  $v_i$ an element of  $p^{-1}(w_i)$ of minimal norm. Then we have:\begin{equation}\label{redvi}
\Vert v_i \Vert \leq \Vert v_i - s \Vert \quad \mbox{et} \quad \Vert v_i \Vert \leq \Vert v_i + s \Vert.
 \end{equation}
Besides, by the very definition of $\lambda_1(\Eb),$ we also have:
\begin{equation}\label{redvibis}
\Vert v_i \Vert \geq \lambda_1(\Eb) = \Vert s \Vert.
\end{equation}

Let us consider the element $v_i^\perp$ in $p_\R^{-1}(w_i)$ orthogonal to $s$. By definition of $\Vert .\Vert_{\rm quot},$ we have:
\begin{equation}\label{vperpw}\Vert v_i^{\perp} \Vert = \Vert w_i \Vert_{\rm quot}.
\end{equation}
Moreover we may write:
$$v_i = v_i^\perp + \eta_i s$$
for some $\eta_i \in \R$. Then we have:
$$ \Vert v_i \Vert^2 = \Vert v_i^{\perp} \Vert^2 + \eta_i^2 \Vert s \Vert^2,$$
and similarly:
$$ \Vert v_i - s \Vert^2 = \Vert v_i^{\perp} \Vert^2 + (\eta_i-1)^2 \Vert s \Vert^2$$
and
$$ \Vert v_i + s \Vert^2 = \Vert v_i^{\perp} \Vert^2 + (\eta_i+1)^2 \Vert s \Vert^2$$
The conditions  (\ref{redvi}) may therefore be rephrased as
$$\eta_i^2  \leq \min ((\eta_i -1)^2, (\eta_i + 1)^2),$$ or equivalently as
$$\vert \eta_i \vert \leq 1/2.$$

This implies:
$$ \Vert v_i \Vert^2 \leq  \Vert v_i^{\perp} \Vert^2 + (1/4) \Vert s \Vert^2,$$
and finally, by taking (\ref{redvibis}) and (\ref{vperpw}) into account:
 \begin{equation}\label{redviter}
\Vert v_i \Vert^2 \leq (4/3)  \Vert w_i \Vert_{\rm quot} ^2.
\end{equation}

The estimates (\ref{redviter}) and (\ref{baseredmoins}), together with the mutiplicativity (\ref{multcovol}) of the covolume, show that:
\begin{equation*}
\begin{split}
\prod_{i=1}^n \Vert v_i \Vert &\leq (4/3)^{(n-1)/2} \prod_{i = 1}^{n-1} \Vert w_i \Vert_{\rm quot} . \Vert s \Vert  \\
& \leq (4/3)^{(n-1)/2} D(n-1) \covol(\overline{E/ \Z s}). \covol (\overline{\Z s}) \\
&=  (4/3)^{(n-1)/2} D(n-1) \covol(\Eb).
\end{split}
\end{equation*}

This establishes the existence of some $\Z$-basis  $(v_1, \dots, v_n)$ of $E$ that satisfies the inequality (\ref{basered}) with $$D(n) = (4/3)^{(n-1)/2} D(n-1),$$ and finally with $D(n)$ given by (\ref{Dn}). \end{proof}

The previous proof actually provide an algorithm\footnote{provided algorithms for  finding a vector of shortest positive norm in some Euclidean lattice, etc., are known.} for constructing the basis $(v_1,\ldots, v_n)$. Bases obtained by this algorithm are called \emph{Korkin-Zolotarev reduced} (see for instance \cite{LagariasLenstraSchnorr90}).  

%\subsection{An application to transference estimates}\label{RedInv}

\subsection{Complements}

In applications, it is convenient to combine Theorem \ref{theored} with the following observations.

\subsubsection{Non-isometric isomorphisms and invariant of Euclidean lattices}\label{invisom} Let $\Eb:= (E, \Vert.\Vert)$ and $\Eb':= (E', \Vert.\Vert')$ be two Euclidean lattices of the same rank $n$  and let 
$$\phi: E \lrasim E'$$ an isomorphism between the underlying $\Z$-modules.

The map $$\phi_\R := \phi \otimes Id_\R : E_\R \lra E'_\R$$ is the an isomorphism of $\R$-vector spaces, but is not necessary an isometry between the Euclidean vector spaces $(E_\R, \Vert.\Vert)$ and $(E'_\R, \Vert.\Vert')$. The ``lack of isometry" of   $\phi_\R$ is controled by the operator norms  $\Vert \phi_\R \Vert$ and $\Vert \phi_\R^{-1} \Vert$ defined by means of the norms $\Vert. \Vert$ and  $\Vert. \Vert'$ on $E_\R$ and $E'_\R$, and one easily sees, by unwinding the definitions, that the covolume, the successive minima, or the coverin radius of  $\Eb$ and $\Eb'$ may be compared, with some error terms controlled by these operator norms:
$$\Vert \phi_\R^{-1} \Vert^{-n} \leq \frac{\covol (\Eb')}{\covol(\Eb)} = \Vert \Lambda^n \phi_\R \Vert \leq \Vert \phi_\R \Vert^n,$$

$$\Vert \phi_\R^{-1} \Vert^{-1} \leq \frac{\lambda_i(\Eb')}{\lambda_i(\Eb)} \leq \Vert \phi_\R \Vert \quad \mbox{for any $i \in \{1, \ldots,n\}$,}$$

$$\Vert \phi_\R^{-1} \Vert^{-1} \leq \frac{\cover (\Eb')}{\cover(\Eb)} \leq \Vert \phi_\R \Vert.$$

These estimates may be reformulated as follows:

\begin{proposition}\label{psilambda}
If by $\psi$ we denote any of the invariants $\mua,$ $\log \lambda_i^{-1}$, or $\log \cover^{-1},$ we have:
\begin{equation}\label{ineqinvisom}
-\log \Vert \phi_\R \Vert \leq \psi(\Eb') - \psi(\Eb) \leq \log \Vert \phi_\R^{-1} \Vert.
\end{equation}
Notably, for any  $\lambda \in \R,$
\begin{equation}\label{psihom}
\psi (\Eb \otimes \cOb(\lambda)) = \psi(\Eb) + \lambda. 
\end{equation}\qed
\end{proposition}

\subsubsection{Reduction theory and norms of sum maps}

Let us consider some Euclidean lattice $\Eb$ of rank $n >0$, and let $L_1, \ldots, L_n$ be some $\Z$-submodules of rank 1 in $E$ such that the $\Z$-module $E$ is the direct sum of  $L_1, \ldots, L_n$. 

We may  introduce the ``sum map"
$$\Sigma : L_1 \oplus \ldots \oplus L_n \lrasim E$$
and the Euclidean lattice  $\Lb_1 \oplus \ldots \oplus \Lb_n$, and consider the operator norms $\Vert \Sigma_\R \Vert$, $\Vert \Lambda^n \Sigma_\R \Vert$ and $\Vert \Sigma_\R^{-1} \Vert$ defined by means of the Euclidean structures on $\Lb_1 \oplus \ldots \oplus \Lb_n$ and on  $\Eb$. 

Finally, we may define:
\begin{align}
\delta(\Eb; L_1, \ldots,L_n)  &:= \mua (\Eb) - \frac{1}{n} \sum_{i = 1}^n \dega \Lb_i  
\\ & = \mua(\Eb) - \mua(\Lb_1 \oplus \ldots \oplus \Lb_n).
\end{align}

\begin{proposition}\label{RedAp} With the previous notation, we have:
\begin{equation}\label{RQ1}
\delta(\Eb; L_1, \ldots,L_n) = -\frac{1}{n} \log \Vert \Lambda^n \Sigma_\R \Vert \geq 0,
\end{equation}

\begin{equation}\label{RQ2}
\log \Vert \Sigma_\R \Vert \leq (1/2) \log n,
\end{equation}
and 
\begin{equation}\label{RQ3}
\log \Vert \Sigma_\R^{-1} \Vert \leq \frac{n-1}{2} \log n + n \delta(\Eb; L_1, \ldots,L_n).
\end{equation}
\end{proposition}

\begin{proof} The estimates (\ref{RQ1}) and (\ref{RQ2}) easily follow from the definitions. They imply  (\ref{RQ3}) thanks to ``Cramer's formula" applied to $\Sigma^{-1}$. Indeed it identifies  $\Sigma^{-1}$ and $\Lambda^{n-1} \Sigma \otimes (\Lambda^n \Sigma)^{-1}$ and shows that:
\begin{align*}
\log \Vert \Sigma_\R^{-1} \Vert & = \log \Vert \Lambda^{n-1} \Sigma_\R \Vert - \log \Vert \Lambda^n \Sigma_\R \Vert \\
& \leq (n-1) \log \Vert\Sigma_\R \Vert + n \delta(\Eb; L_1, \ldots,L_n).
\end{align*} \end{proof}

In the situation of Theorem \ref{theored}, we may apply Proposition \ref{RedAp} with $L_i := \Z v_i$ for $1 \leq i \leq n.$ Then we have:
$$n \delta(\Eb; L_1, \ldots, L_n) = - \log \covol \Eb + \sum_{i=1}^n \log \Vert v_i \Vert \leq \log D(n),$$
and therefore:
\begin{equation*}
\log \Vert \Sigma_\R^{-1}\Vert \leq \frac{n-1}{2} \log n + \log D(n).
\end{equation*}

\subsection{An application to transference inequalities}

Let us keep the previous notation. By applying Proposition  \ref{RedAp} to  $\phi = \Sigma$, we obtain that, if   $\psi$ denotes  any of the invariants $\log \lambda_i^{-1}$ ou $\log \cover^{-1}$, then the following estimate holds:
\begin{equation}\label{RG1}
-\frac{n-1}{2} \log n - \log D(n) \leq \psi (\bigoplus_{i =1}^n \overline{\Z v_i}) - \psi (\Eb) \leq (1/2) \log n.
\end{equation}

We may also apply  Proposition \ref{RedAp} to the isomorphism
$${}^t\Sigma: E^\vee \lrasim \bigoplus L_i^\vee,$$
and thus we obtain:
\begin{equation}\label{RG2}
- (1/2) \log n  \leq \psi (\bigoplus_{i =1}^n \overline{\Z v_i}^\vee) - \psi (\Eb^\vee) \leq  \frac{n-1}{2} \log n +  \log D(n).
\end{equation}

The computations of Paragraph \ref{ExRk1} allow us to compute the invariants of the Euclidean lattices $\bigoplus_{i =1}^n \overline{\Z v_i}$ and  $\bigoplus_{i =1}^n \overline{\Z v_i}^\vee$ in terms of the sequence  $(t_i)_{1\leq i \leq n}:= (\log \Vert v_i \Vert^{-1})_{i \leq i \leq n}$, where the $\Vert v_i \Vert$ are ordered increasingly. Together with the estimates (\ref{RG1}) and (\ref{RG2}) above, these expressions allows one to relate suitable invariants of the Euclidean lattice $\Eb$ and of its dual $\Eb^\vee$. 

For instance, in this way, we may derive the following comparison estimate  between the covering radius of $\Eb$ and the first minimum of $\Eb^\vee$: 

\begin{corollary}\label{weaktransfercor} For any Euclidean lattice $\Eb$ of positive rank $n$, we have:
\begin{equation}\label{weaktransfer}
\left\vert \log \cover (\Eb) + \log \lambda_1(\Eb^\vee) \right\vert \leq E(n),
\end{equation}
where
\begin{equation}\label{weaktransferE}
E(n) = \frac{n+1}{2} \log n + \log D(n).% = O(n^2).
\end{equation}
\end{corollary}

We leave the details of the proof as an exercise.

Statements like \ref{weaktransfercor}, which relates the invariants of geometry of numbers attached to some Euclidean lattice and to its dual are classically known as \emph{transference theorems}\footnote{Originally, \emph{\"Ubertragungss\"atze}; see for instance \cite{Cassels71}, Chapter XI.}. As demonstrated in the above proof of Corollary \ref{weaktransfercor}, reduction theory allows one to give simple proofs of such estimates, by reducing to the easy case of Euclidean lattices  direct sums of Euclidean lattices of rank 1. 

However the constants depending on the rank $n$ of the Euclidean lattices under study --- such as the constant $E(n)$ in (\ref{weaktransfer}) --- that occur in transference estimates derived in this way
turn out  to be ``very large". For instance, as we shall see in the next section, the optimal constant $E(n)$ in (\ref{weaktransfer}) is actually $\log n + O(1)$, while its upper bound (\ref{weaktransferE}) derived from Theorem \ref{theored} is of the order of $n^2$.

%, où toutefois les constantes dépendant de $n$ sont excessivement grandes.
%
%For instance, if we apply (\ref{RG1}) with $\psi = \log \cover^{-1}$ et (\ref{RG2}) avec $\psi = \log \lambda_1^{-1},$ on obtient l'inégalité:
%$$\vert \log \cover (\Eb) + \log \lambda_1(\Eb^\vee) \vert \leq E(n),$$
%avec $$E(n) = \frac{n+1}{2} \log n + D(n) = O(n^2).$$

\section{Theta series and Banaszczyk's transference estimates}\label{ThetaBan}

In this section, we discuss the basic properties of the theta series associated to Euclidean lattices and their remarkable applications, due to Banaszczyk (\cite{Banaszczyk93}), to the study of their classical invariants. 

\subsection{Poisson formula and theta series of Euclidean lattices}

The notion of dual lattice plays a central role in crystallography, since the development of the investigation of crystalline structures by X-ray diffraction: the diffraction pattern obtained from a crystal modeled by some three dimensional Euclidean lattice $\Eb$ produces a picture of the dual lattice  $\Eb^\vee$ (Ewald, von Laue, Bragg, 1912).  This is a physical expression of the \emph{Poisson formula} attached to the Euclidean lattice $\Eb$. Let us recall its formulation, for some Euclidean lattice $\Eb := (E, \Vert . \Vert)$ of arbitrary rank $n$.

The Fourier transform provides an isomorphism of topological vector spaces 
$$\cF: \cS( E_\R) \lrasim \cS(E_\R^\vee)$$
between the Schwartz spaces of $E_\R$ and its dual $\R$-vector space $E_\R^\vee,$ defined by the following formula, for any$f \in \cS(E_\R)$ and any $\xi \in E^\vee_\R$:
$$\cF(f)(\xi):= \int_{E_\R} f(x) e^{-2\pi i \xi(x)} dm_{\Eb}(x).$$
It extends to an isomorphism of topological vector spaces between  spaces of tempered distributions: $$\cF: \cS'( E_\R) \lrasim \cS'(E_\R^\vee).$$

Poisson formula asserts that the counting measures $\sum_{v\in E} \delta_v$ and $\sum_{\xi\in E^\vee} \delta_\xi$ --- which are tempered distributions tempérées on $E_\R$ and $E_\R^\vee$ --- may be deduced from each other by Fourier transform:\begin{equation}\label{Poisson1}
\cF(\sum_{v\in E} \delta_v) = (\covol(\Eb))^{-1} \sum_{\xi\in E^\vee} \delta_\xi.
\end{equation}
Equivalently it asserts that, for any $f \in \cS(E_\R)$ and any $x \in E_\R$, the following equality holds:
\begin{equation}\label{Poisson2}
\sum_{v \in E} f(x-v) = (\covol(\Eb))^{-1} \sum_{\xi\in E^\vee} \cF(f)(\xi) e^{2\pi i \langle \xi, x \rangle}.
\end{equation}
This equality is nothing but the Fourier series expansion of the function  $\sum_{v \in E} f(.-v)$, which is $E$-periodic on $E_\R$.

For any $t \in \Rpa,$ we may apply (\ref{Poisson2}) to the function  $f_t \in \cS(E_\R)$ defined as
$$f_t(x) := e^{- \pi t \Vert x\Vert^2};$$
its Fourier transform is:
$$(\cF f_t)(\xi) = t^{-n/2} e^{- \pi t^{-1}\Vert \xi \Vert^2}.$$
We thus obtain the following equality, for any $x \in E_\R$ :
\begin{equation}\label{Poisson3}
\sum_{v \in E} e^{- \pi t \Vert x -v\Vert^2} = (\covol(\Eb))^{-1} t^{-n/2} \sum_{\xi\in E^\vee} e^{- \pi t^{-1} \Vert \xi \Vert^2 + 2 \pi i \langle \xi, x \rangle}. 
\end{equation}

In particular, when $x=0$, the Poisson formula (\ref{Poisson3}) becomes:
\begin{equation}\label{Poisson4}
\theta_{\Eb} (t) = (\covol (\Eb))^{-1} \, t^{-n/2} \, \theta_{\Eb^\vee}(t^{-1}),
\end{equation}
where the \emph{theta function} $\theta_{\Eb}$ associated to the Euclidean lattice is defined, for any
 $t\in \Rpa$, by the series:
\begin{equation}\label{thetadef}
\theta_{\Eb} (t) := \sum_{v \in E} e^{- \pi t \Vert v \Vert^2}. 
\end{equation}

\subsection{Banaszczyk's transference estimates}

%Les énoncés reliant les invariants de géométrie des nombres attachés à un réseau et à son réseau dual sont classiquement connus sous le nom de \emph{théorèmes de transférence}\footnote{Originellement, \emph{\"Ubertragungss\"atze}; voir par exemple \cite{Cassels71}, Chapter XI.}.

In 1993, in his article \cite{Banaszczyk93}, Banaszczyk has established some remarkable transference estimates, concerning the successive minima and the covering radius: 

\begin{theorem}[Banaszczyk]\label{TrB}
For any Euclidean lattice $\Eb$ of positive rank $n$ and for  any integer $i$ in $\{1, \ldots, n\},$ the following estimate holds :
\begin{equation}\label{Blambda}
\lambda_i(\Eb). \lambda_{n+1 -i}(\Eb^\vee) \leq n.
\end{equation}
Moreover,
\begin{equation}\label{BR}
\cover(\Eb) . \lambda_1(\Eb^\vee)  \leq n/2.
\end{equation}
\end{theorem}

As observed by Banaszczyk, these estimates are optimal, up to some multiplicative error term, uniformly bounded when $n$ varies. This follows from the existence, establishes by Conway and Thompson, of a sequence of Euclidean lattices $\CTnb$ such that 
$$\rk \CTnb = n,$$
\begin{equation}\label{CT}\CTnb^\vee \lrasim \CTnb
\end{equation}
and:
$$\lambda_1(\CTnb) \geq \sqrt{n/2 \pi e} \, (1 + o(n)) \quad \mbox{when $n \lra + \infty$}.$$
(See\cite{MilnorHusemoller1973}, Chapter II, Theorem 9.5. The lattices $\CTnb$ are actually integral unimodular lattices, the existence of which follows from Smith--Minkowski--Siegel mass formula.)
The lattices $\CTnb$ satisfy:
$$\lambda_1(\CTnb). \lambda_n(\CTnb^\vee) \geq \lambda_1(\CTnb)^2 \geq ({n/2 \pi e}) (1 + o(n)) \quad \mbox{when $n \lra + \infty$}.$$
Moreover, according to (\ref{CT}), we have :
$$\covol(\CTnb) =1$$
and therefore, according to (\ref{Rcovmin}) :
$$\cover (\CTnb) \geq v_n^{-1/n} = \sqrt{n/2 \pi e} \,(1 + o(n)) \quad \mbox{when  $n \lra + \infty$}.$$
Consequently, 
$$\lambda_1(\CTnb). \cover(\CTnb^\vee) \geq \lambda_1(\CTnb)^2 \geq ({n/2 \pi e}) (1 + o(n)) \quad \mbox{when $ 
n \lra + \infty$}.$$

To prove Theorem \ref{TrB}, Banaszczyk introduces an original method, which relies on the analytic properties of the theta series  (\ref{thetadef}) associated to Euclidean lattices and on the Poisson formula (\ref{Poisson3}). Previous approaches to transference inequalities,  such as the ones in (\ref{Blambda}) and (\ref{BR}), did rely on reduction theory and, in their best version, led to estimates which constant of the order of  $n^{3/2}$ instead of $n$ (see for instance\cite{LagariasLenstraSchnorr90}). 

The role of the theta series  $\theta_\Eb$ associated to  \emph{integral} Euclidean lattices --- namely, the Euclidean lattices  $\Eb$ defined by some Euclidean scalar product that is $\Z$-valued on  $E \times E$ --- does not need to emphasized: for such lattices, the functions $\theta_\Eb$ define modular form and, through this construction, the theory of modular forms plays a key role in the study and in the classification of integral lattices (see for instance \cite{Ebeling2013} for a modern presentation  of this circle of ideas and for references).

Banaszczyk's method highlights the significance of the theta functions  $\theta_\Eb$ when investigating the fine properties of \emph{general} Euclidean lattices. We present it with some details in the next two sections. For simplicity, we will focus on the second transference inequality (\ref{BR}) in Theorem \ref{TrB}; the proof of (\ref{Blambda}) relies on similar arguments, and we refer the reader to \cite{Banaszczyk93}, p. 631--632 for details. Let us also point out that Banaszczyk has applied analogous techniques to related problems in \cite{Banaszczyk95} and \cite{Banaszczyk96}.

\subsection{The key inequalities}\label{KeyIneq}

Let us consider some Euclidean lattice $\Eb := (E, \Vert.\Vert)$ of positive rank~$n$.

Its theta function $\theta_\Eb$ clearly is a decreasing function. The same holds for $\theta_{\Eb^\vee}$ and the functional equation  (\ref{Poisson4}) relating  $\theta_\Eb$ and  $\theta_{\Eb^\vee}$ therefore show that $t^{n/2} \theta_{\Eb}(t)$ is some increasing function of  $t \in \Rpa$.

Besides, Poisson formula  (\ref{Poisson3}) shows that, for any $x \in E_\R$ and any $t \in \R_+,$ we have:
\begin{equation}\label{Poisson5}
\sum_{v \in E} e^{-\pi t \Vert x-v\Vert^2} \leq \sum_{v \in E} e^{- \pi t \Vert v \Vert^2},
\end{equation}
and that the equality holds in (\ref{Poisson5}) if and only if $x \in E.$

The starting point of Banaszczyk's technique is the following inequality, which easily follows from the previous observations: 

\begin{lemma}\label{lem:Ban} For any $x \in E_\R,$ any $r \in \R_+$ and any $t \in ]0, 1],$ we have:
 \begin{equation}\label{Banana}
\sum_{v \in E, \Vert v -x \Vert \geq r} e^{-\pi \Vert v -x \Vert^2} \leq t^{-n/2} e^{-\pi (1-t) r^2} \sum_{v \in E} e^{-\pi \Vert v \Vert^2}.\end{equation}
\end{lemma}

\begin{proof}
This follows from the following chain of inequalities:
\begin{align}
 \sum_{v \in E, \Vert v -x \Vert \geq r} e^{-\pi \Vert v -x \Vert^2} & = \sum_{v \in E, \Vert v -x \Vert \geq r} e^{-\pi (1-t) \Vert v -x \Vert^2} e^{-\pi t \Vert v -x \Vert^2} \notag \\
 & \leq e^{-\pi (1-t) r^2} \sum_{v \in E, \Vert v -x \Vert \geq r}  e^{-\pi t \Vert v -x \Vert^2} \notag \\
 & \leq e^{-\pi (1-t) r^2} \sum_{v \in E}  e^{-\pi t \Vert v \Vert^2} \label{expl1}\\
 & \leq e^{-\pi (1-t) r^2} t^{-n/2} \sum_{v \in E}  e^{-\pi  \Vert v \Vert^2}. \label{expl2}
\end{align}
Indeed, the estimate (\ref{expl1}) is a consequence of  (\ref{Poisson5}),  and  (\ref{expl2}) of the estimate $t^{n/2} \theta_{\Eb}(t) \leq \theta_{\Eb}(1).$
\end{proof}

The upper bound (\ref{Poisson5}) shows that the estimate (\ref{Banana}) is relevant only when   $r$ is such that
$$\inf_{t \in ]0,1]} t^{-n/2} e^{-\pi (1-t) r^2} <  1.$$
An elementary computation, that we shall left as an exercise, establishes that this inequality is satisfied precisely when  $r > \sqrt{{n}/{2\pi}}$, and that, if this holds and if we define $\tilde{r} \in ]1, +\infty[$ by the relation
$$r =  \sqrt{{n}/{2\pi}}\, \tilde{r},$$
then the minimum of $t^{-n/2} e^{-\pi (1-t) r^2}$ on $]0,1]$ is achieved at
$$t= t_{\rm min}:= \tilde{r}^{-2}$$and assumes the value:
$$ t_{\rm min}^{-n/2} e^{-\pi (1-t_{\rm min}) r^2} = \beta(\tilde{r})^n.$$
where
\begin{equation}\label{betadef}
\beta(\tilde{r}) := \tilde{r} e^{-(1/2) (\tilde{r}^2-1)}.
\end{equation}

These remarks  show that  lemma \ref{lem:Ban} may be reformulated as the following proposition, better suited to applications:

\begin{proposition}\label{prop:Banamna}  Let $\Eb:= (E, \Vert.\Vert)$ be a Euclidean lattice of positive rank $n$, and let  $x$ be some element de  $E_\R.$ For any $\tilde{r} \in[1, +\infty[$, if we let $$r := \sqrt{\frac{n}{2\pi}}\tilde{r},
$$ 
then the following upper bound holds:
\begin{equation}\label{Banamna}
\sum_{v \in E, \Vert v -x \Vert \geq r} e^{-\pi \Vert v -x \Vert^2} \leq
\beta(\tilde{r})^n
 \sum_{v \in E} e^{-\pi \Vert v \Vert^2}. 
\end{equation}
\qed
\end{proposition}

Observe that formula  (\ref{betadef}) defines some decreasing homeomorphism:
$$\beta : [1, +\infty)  \lrasim (0,1].$$

Besides,  Poisson formula (\ref{Poisson3}) implies the following equalities:
\begin{align*}\sum_{v \in E} e^{-\pi \Vert x-v \Vert^2} + \sum_{v \in E} e^{-\pi \Vert v \Vert^2} & =
(\covol \Eb)^{-1}  \sum_{\xi \in E^\vee} e^{-\pi \Vert \xi \Vert^2} [1+ \cos (2\pi \xi(x))] \\
& = 2 (\covol \Eb)^{-1}  \sum_{\xi \in E^\vee} e^{-\pi \Vert \xi \Vert^2} \cos^2 (\pi \xi (x)),
\end{align*}
This implies:
\begin{proposition}\label{propPoissonGaussineqmoins} For any Euclidean lattice $\Eb$ and for any $x \in E_\R,$ we have:
\begin{equation}\label{PoissonGaussineqmoins}
\sum_{v \in E} e^{-\pi \Vert x-v \Vert^2} + \sum_{v \in E} e^{-\pi \Vert v \Vert^2} \geq  
2 (\covol \Eb)^{-1}.
\end{equation}
\qed
\end{proposition}

\subsection{Proof of the transference inequality (\ref{BR})}\label{preuveBR} 

Let us first state two corollaries of Propositions  \ref{prop:Banamna} and \ref{propPoissonGaussineqmoins}. 

By applying Proposition \ref{prop:Banamna} to $x=0$ et $r = \lambda_1(\Eb),$ we get:

\begin{corollary}\label{corolambda} Let $\Eb$ be some Euclidean lattice of positive rank $n$ and of first minimum 
$\lambda_1(\Eb) > \sqrt{n/2\pi}$, and let 
 $\tilde{\lambda} \in ]1,+\infty[$ be defined by 
$$\lambda_1(\Eb)= \sqrt{n/2\pi}\tilde{\lambda}.$$

Then the following upper bound on $\theta_\Eb(1)$ holds:
\begin{equation}\label{hotlambdaBan}
%\theta_{\Eb}(1) 
\theta_\Eb(1) :=\sum_{v \in E} e^{-\pi \Vert v \Vert^2} \leq (1- \beta(\tilde{\lambda})^n)^{-1}.
\end{equation} \qed 
%and
%\begin{equation}\label{hotlambdaBanbis}
%e^{\hot(\Eb)} - 1 \leq \frac{[\tilde{\lambda} e^{-(1/2) (\tilde{\lambda}^2-1)}]^n}{1-[\tilde{\lambda} e^{-(1/2) (\tilde{\lambda}^2-1)}]^n}.
%\end{equation}
\end{corollary}

Besides, by the very definition of the covering radius $\cover(\Eb)$ of some Euclidean lattice $\Eb$, there exists  $x \in E_\R$ such that $\Vert v -x \Vert \geq \rho(\Eb)$ for any~$v$ in  $E.$ If we apply Proposition \ref{prop:Banamna} to such a point $x$ and to $r = \cover(\Eb),$ we obtain the first assertion in the following Corollary:

\begin{corollary}\label{corR} Let $\Eb$ be some Euclidean lattice of positive rank $n$ and of covering radius $$\cover(\Eb) \geq  \sqrt{n/2\pi},$$ and let $\tilde{R} \in [1,+\infty[$ be defined by
$\cover(\Eb) =  \sqrt{n/{2\pi}}\tilde{R}.$

Then there exists $x \in E_\R$ tel que 
\begin{equation}\label{boundR}
\frac{\sum_{v \in E} e^{-\pi \Vert v -x \Vert^2}}{\sum_{v \in E} e^{-\pi \Vert v \Vert^2}} \leq \beta(\tilde{R})^n,
\end{equation}
and consequently:
\begin{equation}\label{betaminor}
\beta(\tilde{R})^n \geq 2 \theta_{\Eb^\vee}(1)^{-1} -1.
\end{equation}
\end{corollary}

\begin{proof} We are left to prove (\ref{betaminor}). To achieve this, observe that, according to Proposition \ref{propPoissonGaussineqmoins}, 
$$\frac{\sum_{v \in E} e^{-\pi \Vert v -x \Vert^2}}{\sum_{v \in E} e^{-\pi \Vert v \Vert^2}} \geq 2 \covol(\Eb)^{-1}\theta_{\Eb}(1)^{-1} -1, $$ and use the functional equation 
(\ref{Poisson4}) relating $\theta_\Eb$ et $\theta_{\Eb^\vee}$ for $t=1$, which takes the form:
$$\theta_{\Eb}(1) = (\covol(\Eb))^{-1} \, \theta_{\Eb^\vee}(1).$$
\end{proof}

We are now in position to establish the transference inequality (\ref{BR}), namely:
$$ \cover(\Eb). \lambda_1(\Eb^\vee) \leq n/2.$$

Let us consider a Euclidean lattice $\Eb$ of positive rank $n$ and let us define $\tilde{R}$ and $\tilde{\lambda}^\vee$ by the equalities 
$$\cover(\Eb) = \sqrt{n/2\pi}\, \tilde{R} \quad \mbox{and} \quad  \lambda_1(\Eb^\vee) = \sqrt{n/2\pi} \,\tilde{\lambda}^\vee.$$

\begin{lemma}\label{lembetabeta}
If  $\min(\tilde{\lambda}^\vee, \tilde{R}) > 1,$ then :
\begin{equation}\label{betabeta}
\beta(\tilde{R})^n + 2 \beta(\tilde{\lambda}^\vee)^n \geq 1.
\end{equation}
\end{lemma}

\begin{proof}% Le corollaire \ref{corR} assure l'existence de $x \in E_\R$ satisfaisant à l'inégalité (\ref{boundR}).  En minorant le membre de gauche de cette inégalité au moyen de la
%proposition \ref{propPoissonGaussineqmoins}, on en déduit :
%\begin{equation}\label{betatheta}
%2 \covol(\Eb)^{-1}\theta_{\Eb}(1)^{-1} -1 \leq \beta(\tilde{R})^n.
%\end{equation}	 

%Par ailleurs, l
Corollary \ref{corolambda}, applied to $\Eb^\vee,$ shows that:
\begin{equation}\label{thetalambda}
 1 -\beta(\tilde{\lambda}^\vee)^n \leq \theta_{\Eb^\vee}(1)^{-1}.
\end{equation}
The inequality (\ref{betabeta}) follows from  (\ref{betaminor}) and (\ref{thetalambda}).
\end{proof}

For any $n>0,$ we let:
$$t_n := \beta^{-1}( 3^{-1/n}) \in ]1, +\infty[.$$

\begin{lemma}\label{tn}
When $n$ goes to infinity,
\begin{equation}\label{tnasymp}
t_n = 1 + \sqrt{(\log 3)/n} + O(1/n).
\end{equation}
Moreover, 
\begin{equation}\label{tnmin} t_n \leq  \sqrt{\pi} \quad \mbox{for any $n \geq 3.$}
\end{equation}
\end{lemma}

\begin{proof} An elementary computation shows that,  when $x \in \R_+^\ast$ goes to 0, 
$$\beta(1+ x) = 1-x^2 + O(x^3).$$
This implies that, when $y \in (0, 1)$ goes to zero,
$$\psi^{-1}(1-y) = 1 + \sqrt{y} + O(y).$$ 
Since $$t_n = 1 - (\log 3)/n + O(1/n^2),$$
this  proves (\ref{tn}).

Observe also that:
\begin{equation*}
\begin{split}
t_n \leq \sqrt{\pi} &\,\Longleftrightarrow \, \beta(t_n) \geq \beta(\sqrt{\pi}) \\
&\,\Longleftrightarrow \, 3^{-1/n} \geq \sqrt{\pi} \exp(-(\pi -1)/2) \\
&\,\Longleftrightarrow \, -( \log 3) /n \geq - (\pi - 1) /2 + (1/2) \log \pi.
\end{split}
\end{equation*}
As $$\log 3 = 1.0986...$$ and 
$$(\pi - 1) /2 - (1/2) \log \pi =0.4984...,$$
the above inequalities hold for any integer $n \geq 3$.
\end{proof}

Using Lemma \ref{lembetabeta}, we may derive a slightly stronger version of Banaszczyk's inequality (\ref{BR}) for $n\geq 3$:

\begin{proposition}\label{propBRraf} For any Euclidean lattice of positive rank $n$, the following inequality holds:
\begin{equation}\label{BRraf}
\cover(\Eb) . \lambda_1(\Eb^\vee) \leq t_n^2 n/2\pi.
\end{equation}
\end{proposition}

Actually (\ref{BR}) is trivial when $n=1$ and follows from elementary considerations, involving reduced bases of two dimensional Euclidean lattices, when $n=2$.
\begin{proof}[Proof of Proposition \ref{propBRraf}] Let us first assume that\begin{equation}\label{hypegal}\cover (\Eb) = \lambda_1(\Eb^\vee) =: t.
\end{equation}

According to Lemma \ref{lembetabeta},  if $t > 1$, then  $\beta(t)\geq 3^{-1/n}$ and therefore  $t \leq t_n$.
Since $t_n > 1,$ this inequality still holds when $t \leq 1$. The estimate  (\ref{BRraf}) immediately follows.

The general validity of  (\ref{BRraf}) follows from its validity under the additional assumption  (\ref{hypegal}). Indeed,  replacing the Euclidean lattice $\Eb$ by $\Eb \otimes \cOb(\delta)$ for some $\delta \in \R$ --- that is, scaling the Euclidean norm of $\Eb$ by the positive $e^{-\delta}$ ---  does not change the product $\cover(\Eb) . \lambda_1(\Eb^\vee)$; moreover, by a suitable choice of $\delta$, the condition
$$\rho(\Eb\otimes \cOb(\delta)) = \lambda_1((\Eb\otimes \cOb(\delta))^\vee)$$ may be achieved. 
Indeed, from the very definitions of the covering radius and of the first minimum, we obtain:
$$\rho(\Eb\otimes \cOb(\delta)) = e^{-\delta}  \rho(\Eb)\quad \mbox{and} \quad \lambda_1((\Eb\otimes \cOb(\delta))^\vee)= e^\delta \lambda_1(\Eb^\vee).$$
\end{proof}

\section{Vector bundles on curves and the analogy with Euclidean lattices}\label{AnaVect}

\subsection{Vector bundles on smooth projective curves and their invariants}\label{fibresvect}

In this section, we recall some   basic facts  concerning vector bundles on algebraic curves that   play a key role in the analogy  between vector bundles and Euclidean lattices 

%These results paly a central role 
%Dans cette section, nous rappelons divers résultats classiques concernant les fibrés vectoriels sur les courbes algébriques. Ces résultats jouent un rôle central dans l'étude des espaces de modules classifiant ces fibrés vectoriels, et l'on pourra se reporter à \cite{VerdierLePotier83} pour une présentation synthétique et des références sur ce sujet. Nous présentons ici ces résultats en tant que "modèles", pour l'étude des réseaux euclidiens %font l'objet
%dans la suite de cet exposé. Notamment les propositions  \ref{szero} et \ref{sun} ont été formulées explicitement parce qu'elles apparaîtront comme les analogues géométriques du théorème\ref{mainth}.  

Let  $C$ be a smooth, projective and geometrically connected curve over some field  $k$. We shall denote the field of rational functions over $C$ by
$$K:= k(C).$$ 

\subsubsection{} A \emph{vector bundle} $E$ over $C$ is locally free coherent sheaf over $C$.  Any coherent subsheaf $F$ of $E$ is again a vector bundle over $C$. We shall say that $F$ is a \emph{vector subbundle} of $E$ when the coherent sheaf $E/F$ is torsion-free, and therefore also defines a vector bundle over $C$.

The fiber $E_K$ of $E$ at the generic point of  $C$ --- namely, the space of rational sections of  $E$ over $C$ --- is a  finite dimensional $K$-vector space. When $F$ is a coherent subsheaf of $E$,  $F_K$ is a $K$-vector subspace of  $E_K$, and this construction establishes a bijection between vector subbundles of  $E$ and $K$-vector subspaces of $E_K$.

We may define tensor operations on vector bundles: to any vector bundle $E$ over  $C$, we may attach its dual vector bundle $E^\vee$ and, for any $n \in \N$, its tensor power  $E^{\otimes n}$ and its exterior power  $\bigwedge^n E$; to any two vector bundles $E$ and $F$ over $C$, we may attach their tensor product  $E\otimes F$ and the vector bundle $$\Hom (E, F) \simeq E^\vee \otimes F.$$

%\subsection{Invariants des fibrés vectoriels}

To any vector bundle  $E$ over  $C$ are associated the following invariants:

\begin{itemize}
\item its rank $$\rk E := \dim_K E_K \in \N;$$
\item its degree $$\deg E \in \Z.$$
\item when $\rk E >0,$ its slope:
$$\mu(E) := \frac{\deg E}{\rk E} \in \Q.$$
\end{itemize}

\subsubsection{} A reminder on the various definitions of $\deg E$ in the present setting may be in order.

When  $E$ has rank  $1$ --- that when $E$ is a \emph{line bundle} or \emph{invertible sheaf} ---  hence isomorphic to the sheaf $\cO_C(D)$ associated to the divisor  $$D = \sum_{i\in I} n_i P_i$$ of some non-zero rational section of $E$ (defined by some family  $(P_i)_{i \in I}$ of closed points of  $C$ and multiplicities  $(n_i)_{i\in I} \in \Z^I$), it is defined as:
$$\deg E = \deg \cO_C(D) = \deg D := \sum_{i \in I} n_i [\kappa(P_i): k].$$
To define the degree of some vector bundle $E$ of arbitrary rank, one reduces to the case of line bundles by considering its maximal exterior power:
$$\deg E := \deg \bigwedge^{\rk E} E.$$

An alternative definition of the degree a vector bundle $E$ involves it so-called Hilbert polynomial. Let us assume, for simplicity, that the curve $C$ admits some divisor $D$ of degree 1\footnote{Such a divisor exists when the base field $k$ is algebraically closed (then the divisor $D$ defined by any point in $C(k)$ will do), or when $k$ is finite.}.  Then, when the integer $n$ is large enough, we have:
\begin{equation}\label{HilPol}
\dim_k H^0(C, E \otimes \cO_C(nD)) = n\,  \rk E + (1-g)\,  \rk E + \deg E.
\end{equation}
where we denote by $g$ the genus of $C$.  This is a straightforward consequence of the Riemann-Roch formula for the vector bundle $E \otimes \cO_C(nD)$, combined to the vanishing of $H^0(C, E \otimes \cO_C(nD))$ when $n$ is large enough, itself a consequence of the ampleness of $\cO_C(D)$. 

The right-hand side of (\ref{HilPol}), as a function of $n$, defines the Hilbert polynomial of $E$. In particular, when $g=1$, its constant term is the degree of $E$. 
\subsubsection{} The invariants defined above satisfy the following properties.

(i) For any vector bundle  $E$ over $C$ and any vector subbundle $F$ of $E$,
\begin{equation}\label{degaddgeom}
\deg E = \deg F + \deg E/F;
\end{equation}
%L'existence d'un morphisme entre fibrés vectoriels $\phi : E_1 \lra E_2$ qui est un isomorphisme au point générique $$\phi_K: E_{1,K} \lrasim E_{2,K}$$ implique l'inégalité
%$$\deg E_1 \leq \deg E_2$$
%entre leurs degrés.

(ii) For any vector bundle of positive rank $E$ and any line bundle $L$ over  $C$,
$$\mu( E \otimes L) = \mu(E) + \deg L.$$
More generally, for any two vector bundles of positive rank $E$ and $F$ over $C,$ we have:
$$\mu (E \otimes F) = \mu(E) + \mu(F).$$

(iii) If $\phi: E \lra E'$ is a morphism of sheaves of  $\cO_C$-modules between two vector bundles which is an isomorphism at the generic point:
$$\phi_K : E_K \lrasim E'_K,$$
then
$$\deg E \leq \deg E',$$
and equality holds if and only if $\phi$ is an isomorphism.

%In particular, for any coherent subsheaf  $F$ of $E$, we have:
%$$\deg F \leq \deg F^{\rm sat},$$
%avec égalité si et seulement si $F$ est un sous-fibré vectoriel de $E$.

(iv) For any vector bundle $E$ over $C,$ of dual  $E^\vee := \Hom (E, \cO_C)$, we have:
$$\deg E^\vee = - \deg E.$$

\subsection{Euclidean lattices as analogues of vector bundles over projective curves}

The analogy between number fields and function fields has played a central role in the development of algebraic geometry and number theory since the second half of the nineteenth century, starting with the works of Dedekind and Weber and of Kronecker.

Here we will be concerned with the version of this analogy which constitutes the framework of Arakelov geometry\footnote{Rather different versions of this analogy have played a key role in some other areas of arithmetic geometry, for instance in Iwasawa theory.}, and which originates in Hensel's idea that ``all places of a number field $K$ are on the same footing" and that, accordingly, besides the places of $K$ defined by closed points of $\Spec \OK$, its archimedean places, associated to field extensions $\sigma: K \hlra \C$ (up to complex conjugation) play an equally important role.

An elementary but significant manifestation of the analogy between number fields and function fields is the analogy between Euclidean lattices and vector bundles over a smooth projective and geometrically irreducible curve $C$ over some base field $k$.

In this analogy, the field $\Q$ takes the place of the field $K:= k(C)$ of rational functions over $k$, and the set of places of $\Q$ (which may be identified to the disjoint union of the closed points of $\Spec \Z$ --- in other words, the set of prime numbers --- and of the archimedean place of $\Q$, defined by the usual absolute value) takes the palce of the closed points of $C$. 

Moreover the $\Q$-vector space $E_\Q$ associated to some Euclidean $\Eb$ is the counterpart of the fibre $E_K$ of some vector bundle  $E$ at the generic point of $C$;  the Euclidean lattices  $\Fb$ associated to some $\Z$-submodules $F$ of $E$ (resp. to saturated $\Z$-submodules) play the role of coherent subsheaves (resp. of sub-vector bundles) of  $E$, and the admissible short exact sequences of euclidean lattices (\ref{adm}) the one of short exact sequences of vector bundles over $C$. 

These analogies, in their crudest form, have been pointed out for a long time (see notably \cite{Weil39} and \cite{Eichler66}, Chapter I). It turns out that the invariants $\hot(\Eb)$ of Euclidean lattices allow one to pursue these classical analogy in diverse directions, with an unexpected level of precision. Notably diverse recent progresses in the study of Euclidean lattices that arose in the last decade in relation with their application to cryptography offer striking illustrations of this general philosophy. We refer the reader to 
\cite{Bost2018} for a discussion of these developments, due notably to Micciancio, Regev, Dadush and Stephens-Davidowitz, which involve comparison estimates relating the slopes of Euclidean lattices and suitable invariants defined in terms of their theta series. 

In this section, we discuss a few simple instances of this analogy only.

\subsection{The invariants $\hon(\Eb)$, $\hot(\Eb)$ and $\hut(\Eb)$}
 
 In the literature devoted to the analogy between number fields and function fields and to Arakelov geometry are described \emph{several} invariants of Euclidean lattices which play the role of the  dimension $$h^0(C,E):= \dim_ k H^0(C, E)$$ of the space of sections of some vector bundle $E$ over some curve  $C$ over some base field $k$, or of the dimension $$h^1(C,E):= \dim_k H^1(C, E)$$ of its first cohomology group.
  
 \subsubsection{The invariant $\hon(\Eb)$}
 
With the notation of  Section \ref{fibresvect}, the  $k$-vector space $H^0(C,E)$ may be identified to the $k$-vector space $\Hom_{\cO_C}( \cO_C, E)$ of morphisms of sheaves of  $\cO_C$-modules  from $\cO_C$ to $E$. When the base field $k$ is finite of cardinality $q$, it is a finite set and we have:
$$h^0(C,E) = \dim_k \Hom_{\cO_C}( \cO_C, E) = \frac{\log \vert \Hom_{\cO_C}( \cO_C, E) \vert}{\log q}.$$

This leads one to consider the set of morphisms from $\cOb(0) = (\Z, \vert .\vert)$ to some Euclidean lattice $\Eb := (E, \Vert.\Vert)$ --- by mapping such a morphism $\phi$ to $\phi(1)$, it may identified with the finite set
 $$E \cap \overline{B}_{\Vert.\Vert}(0,1)$$ of the lattice points in the unit ball of  $(E_\R, \Vert.\Vert)$ --- and then to consider the logarithm of its cardinality:
\begin{equation}\label{Defhon}\hon(\Eb) := \log \vert E \cap \overline{B}_{\Vert.\Vert}(0,1) \vert.\end{equation}

This definition  appears implicitly in the works of  Weil (\cite{Weil39}) and  Arakelov (\cite{Arakelov75}), and more explicitly in the presentations of Arakelov geometry in \cite{Szpiro85} and \cite{Manin85}. See also \cite{GilletMazurSoule1991} for some variation on this definition, and some definition in the same vein of an analogue of $h^1(C,E)$.

\subsubsection{The invariants $\hot(\Eb)$ and $\hut(\Eb)$}

One may also introduce the theta series $\theta_\Eb$ associated to some Euclidean lattice  $\Eb := (E, \Vert.\Vert)$, defined as:
$$\theta_\Eb (t) := \sum_{v \in E} e^{- \pi t \Vert v\Vert^2} \quad \mbox{for any $t \in \Rpa$}$$
(see (\ref{thetadef}) \emph{supra}),  and then define :
\begin{equation}\label{Defhot}
\hot(\Eb) := \log \theta_{\Eb}(1) =\log \sum_{v \in E} e^{-\pi \Vert v \Vert^2} \in \R_+.
\end{equation}

The fact that the so-defined invariant $\hot(\Eb)$ of the Euclidean lattice $\Eb$ is an analogue of the invariant  $h^0(C,E)$ attached to some vector bundle $E$ over some projective curve $C$ is a remarkable discovery of the German school of number theory, and goes back to F. K. Schmidt (at least).  Indeed, if one compare the proofs, respectively by Hecke (\cite{Hecke17}) and  Schmidt (\cite{Schmidt31}) of the analytic continuation and of the functional equation of the zeta functions associated to a number field and to a function field  $K:= k(C)$ attached to some curve $C$ (projective, smooth, and geometrically connected) over some finite field $k$ of cardinality, one sees that the sum  
$$\sum_{v \in E} e^{- \pi \Vert v \Vert^2}$$
associated to some Euclidean lattices $\Eb := (E, \Vert .\Vert)$ play the same role as the expressions $$q^{h^0(C,E)}.$$

A key feature of Schmidt's proof is actually that the Riemann-Roch formula for a (rank 1) vector bundle over a curve plays a role similar to the one of the Poisson formula (\ref{Poisson4}) which relates  $\theta_\Eb$ and $\theta_{\Eb^\vee}$. Indeed, at the point  $t=1$,  this formula becomes: %devient en effet:
$$\theta_\Eb(1) = (\covol (\Eb))^{-1} \theta_{\Eb^\vee}(1)$$
and, by taking logarithms, may also be written:
\begin{equation}\label{prePRR}
 \hot(\Eb) - \hot(\Eb^\vee) = \dega \Eb.
 \end{equation}
 
 This equality is formally similar to the Riemann-Roch formula over a smooth projective curve $C$ of genus $1$ (hence of trivial canonical bundle), and leads one to define:
\begin{equation}\label{hundef} \hut(\Eb) := \hot(\Eb^\vee),\end{equation}
 so that (\ref{prePRR}) becomes the ``Poisson-Riemann-Roch" formula:
 \begin{equation}\label{PRR}
 \hot(\Eb) - \hut(\Eb) = \dega \Eb.
 \end{equation}
 
 During the last decades, the above definitions (\ref{Defhot}) and (\ref{hundef})  have notably appeared in Quillen's mathematical diary \cite{QuillenNotebooks} (see the entries on  12/24/1971, 04/26/1973 and  04/01/1983), in \cite{Roessler93}, \cite{Morishita95},  and more recently in the articles by van der Geer and Schoof \cite{vanderGeerSchoof2000} and Groenewegen \cite{Groenewegen2001}.
 
\subsection{How to reconcile the invariants $\hon(\Eb)$ and $\hot(\Eb)$}\label{Reconcile}  

It is comforting that, as shown in \cite{Bost2018}, the  definitions     (\ref{Defhon}) et (\ref{Defhot}) of the invariants $\hon(\Eb)$ and $\hot(\Eb)$, both candidate for playing the role of $h^0(C,E)$ for Euclidean lattices, may be reconciled.

\subsubsection{Comparing $\hon(\Eb)$ both $\hot(\Eb)$ by Banaszczyk's method}

\begin{proposition}[\cite{BostTheta}, Theorem 3.1.1]\label{honhot} For any Euclidean lattice of positive rank $n$, the following inequalities hold:
\begin{equation}\label{ineqhonhot}
-\pi \leq \hot(\Eb) - \hon(\Eb) \leq (n/2) \log n + \log (1-1/2\pi)^{-1}.
\end{equation}
\end{proposition}

\begin{proof} To prove the first inequality in (\ref{ineqhonhot}), we simply observe that:
\begin{equation*}
\begin{split}
\hot(\Eb) = \log \sum_{v \in E} e^{-\pi \Vert v \Vert^2} & \geq \log \sum_{\stackrel{v \in E}{\vert v\vert \leq 1}} e^{-\pi \Vert v \Vert^2} \\ & \geq \log (e^{-\pi} \vert \{ v \in E \mid \Vert v \Vert \leq 1 \} \vert) = -\pi + \hon(\Eb).
\end{split}
\end{equation*}

The proof of second inequality in (\ref{ineqhonhot}) will rely on the following assertions, which are variants of results  in \cite{Banaszczyk93}, Section 1. 

\begin{lemma}\label{VarBanas}
1) The expression $\log \theta_{\Eb}(t)$ defines a decreasing function of $t$ in $\R_{+}^\ast,$
and the expression 
\begin{equation}\label{incr}\log \theta_{\Eb}(t)  + \frac{1}{2} \rk E. \log t\end{equation}
 an increasing function of $t$ in $\R_{+}^\ast.$

2) We have:
\begin{equation}\label{thetav2}
\sum_{v \in E} \Vert v \Vert^2 e^{-\pi t \Vert v \Vert^2} \leq \frac{\rk E}{2\pi t} \sum_{v\in E}e^{-\pi t \Vert v\Vert^2}.
\end{equation}

3) For any $t$ and $r$ in $\R^\ast_{+},$ we have:
\begin{equation}\label{thetar}
\sum_{v \in E, \Vert v \Vert  <  r}  e^{-\pi t \Vert v \Vert^2} \geq \left( 1 - \frac{\rk E}{2\pi t r^2}\right) \sum_{v\in E}e^{-\pi t \Vert v\Vert^2}.
\end{equation}
\end{lemma}

Assertion 1) was already used in  Subsection \ref{KeyIneq}, as the starting point of the proof of Banaszczyk's transference estimates.

\proof The first assertion in 1) is clear. According to the functional equation (\ref{Poisson4}) which relates $\theta_\Eb$ and $\theta_{\Eb^\vee}$, the expression (\ref{incr}) may also be written
$$ \dega \Eb + \log \theta_{\Eb^\vee}(t^{-1}),$$
and consequently defines an increasing function of $t.$ 

The inequality (\ref{thetav2}) may also be written
$$-  \frac{1}{\pi} \frac{d \theta_{\Eb}(t)}{dt} \leq \frac{\rk E}{2 \pi t} \theta_{\Eb}(t),$$
and simply expresses that the derivative of (\ref{incr}) is non-negative. 

To establish the inequality (\ref{thetar}), we combine 
the straightforward estimate
$$\sum_{v \in E, \Vert v \Vert \geq  r}  e^{-\pi t \Vert v \Vert^2} \leq \frac{1}{r^2}  \sum_{v\in E}\Vert v \Vert^2 e^{-\pi t \Vert v\Vert^2}$$
with (\ref{thetav2}). This yields:
$$ \sum_{v \in E, \Vert v \Vert \geq  r}  e^{-\pi t \Vert v \Vert^2} \leq   \frac{\rk E}{2 \pi t r^2}  \sum_{v \in E}  e^{-\pi t \Vert v \Vert^2},$$
or equivalently, 
$$\sum_{v \in E, \Vert v \Vert <  r}  e^{-\pi t \Vert v \Vert^2} \geq \left( 1 - \frac{\rk E}{2\pi t r^2}\right) \sum_{v\in E}e^{-\pi t \Vert v\Vert^2}.$$ 
\qed 

From (\ref{thetar}) with $r=1$, we obtain
that, for any $t > \rk E /2\pi,$ we have:
$$\hon(\Eb) \geq \log (1 - \rk E /(2 \pi t)) + \log \theta_{\Eb} (t).$$
Using also that, for any $t\geq 1,$
$$\log \theta_{\Eb} (t) \geq \log \theta_{\Eb}(1)  - \frac{1}{2} \rk E. \log t,$$
we finally obtain that, for any $t \geq \min (1, \rk E/2 \pi),$ the following inequality holds:
$$\honm(\Eb) \geq \log (1 - \rk E /(2 \pi t))  - \frac{1}{2} \rk E. \log t + \hot(\Eb).$$

Notably we may choose $t= \rk E,$ and then we obtain\footnote{The ``optimal"   choice of $t$ in terms of $n := \rk E$ would be $t= (n+2)/2\pi.$ This choice leads to the slightly stronger  estimate:
$\hon(\Eb) \geq - \frac{n+2}{2} \log \frac{n+2}{2\pi} -\log \pi + \hot(\Eb).$}:
$$\hon(\Eb) \geq \log (1 - {1}/{2 \pi}) - \frac{1}{2} \rk E. \log \rk E+ \hot(\Eb).$$
This completes the proof of the second inequality in   (\ref{ineqhonhot}). 
\end{proof}

%La première inégalité est immédiate, et la seconde se démontre au moyen des techniques de Banasz\-czyk présentées en section \ref{secBan} (voir \cite{BostTheta}, section 3).

\subsubsection{The stable invariant $\hont(\Eb, x)$ and the Legendre transform of theta invariants}

As before, we denote by $\Eb$ some Euclidean lattice of positive rank. As mentioned in the Introduction, it is also possible to relate $\hot(\Eb)$ to some ``stable version" of the invariant $\hon (\Eb)$.  

For every $t \in \Rpa,$ we shall define: 
\begin{equation*}
\hon(\Eb, x) :=  \hon(\Eb \otimes \cOb ((\log x)/2) =  \log \vert \{ v \in E \mid \Vert v \Vert^2 \leq x \}\vert.
\end{equation*}

\begin{theorem}[\cite{BostTheta}, Theorem 3.4.5] \label{Thhonthot} For any $x \in \Rpa,$ the following limit exists in  $\R_+$:
$$\hont(\Eb, x)  := \lim_{k \rightarrow + \infty} \frac{1}{k} \; \hon(\Eb^{\oplus k}, kx).$$

The function  $\log \theta_\Eb(\beta) (= \hot(\Eb \otimes \cOb((\log \beta^{-1})/2)$ and  $\hont(\Eb, x)$ of $\beta$ and $x$ in $\Rpa$ are  real analytic, and respectively decreasing and strictly convex, and increasing and strictly concave. 

Moreover, they  may be deduced from each other  by Fenchel-Legendre duality. Namely, for any  $x \in \Rpa$, we have:
\begin{equation*}\label{Leg-hontIntro}
 \hont(\Eb, x) = \inf_{\beta > 0} ( \log \theta_\Eb(\beta) + \pi \beta x),
\end{equation*}
and, for any $\beta \in \Rpa$:
\begin{equation*}\label{Leg-logthetaIntro}
\log\theta_\Eb(\beta) = \sup_{x > 0} (\hont(\Eb, x) -\pi \beta x).
\end{equation*}
\end{theorem} 

When $\Eb$ is the ``trivial'' Euclidean lattice of rank one $\cOb(0):= (\Z, \vert. \vert)$, Theorem \ref{Thhonthot} may be deduced   from  results of Mazo and Odlyzko (\cite{MazoOdlyzko90}, Theorem 1). 

As mentioned in the introduction, the next sections of these notes are devoted to the proof of Theorem \ref{Thhonthot}. This proof will emphasize the relation between this theorem and the thermodynamic formalism. Notably, the function $\hont(\Eb, x)$ will appear as some kind of ``entropy function" associated to the Euclidean lattice $\Eb$, and will satisfy the following additivity property, which may be seen as an avatar of the second principle of thermodynamics for Euclidean lattices:

\begin{corollary}\label{SecondLawLattice} For any two Euclidean lattices $\Eb_1$ and $\Eb_2$ of positive rank, and any $x \in \Rpa,$
\begin{equation}\label{equ:secondlaawlattice}
\hont(\Eb_1 \oplus \Eb_2, x) = \max_{\stackrel{x_1, x_2 >0}{x_1 +x_2 = x}} \left(\hont(\Eb_1, x_1) + \hont(\Eb_2, x_2)\right).
\end{equation}
\end{corollary}

This will be proved in paragraph \ref{SPLatt}, in a more precise form.

\subsubsection{Complements}
By elaborating on the proof of Proposition \ref{honhot} above, it is possible to establish the following additional comparison estimates relating $\hon, \hont,$ and $\hot.$

For any integer $n\geq 1,$ we let:
\begin{equation*}
C(n) := - \sup_{t >1}\;  [\log (1-t^{-1}) - (n/2) \log t]. 
\end{equation*}
One easily shows that
$$C(n) = \log (n/2) + (1+n/2) \log(1+2/n)$$
and that
$$1 \leq C(n) - \log (n/2) \leq (3/2) \log 3.$$

\begin{proposition}\label{honhonthot} For every Euclidean lattice $\Eb$ of positive rank $\Eb,$ we have:
\begin{equation*}%\label{compthetamoinsnaive1}
-C(\rk E) \leq \hon(\Eb, \rk E/ 2 \pi) - \hot(\Eb)  \leq \rk E/2
\end{equation*}
and 
\begin{equation*}%\label{compthetamoinsnaive2}
0 \leq   \hont(\Eb, \rk E/ 2 \pi) - \hot(\Eb)\leq \rk E/2.
\end{equation*}

\end{proposition}

See \cite{Bost2018}, paragraph 3.4.4, for the details of the proof. 

It may also be shown that, when $\rk E$ goes to infinity, the order of growth of the constants in the comparison estimates in Propositions \ref{honhot} and \ref{honhonthot} is basically optimal; see \cite{Bost2018}, Section 3.5.
 
 \subsection{Some further analogies between $\hot(\Eb)$ and $h^0(C, E)$}
 
 A major difference between the invariants  $\hot(\Eb)$ et $\hut(\Eb)$ attached to Euclidean lattices and the dimensions  $h^0(C, E)$ and $h^1(C, E)$ of the cohomology groups of vector bundles is that, while the latter are integers, the former are real, and that, when  $\Eb$ has positive rank, the former never vanish. 
  
 This being said, the analogies between the properties both sets of invariant are especially striking. In the next paragraphs, we describe three of them, by order of increasing difficulty. 
  
\subsubsection{Asymptotic behavior of  $\log \theta_\Eb$} Starting from the equality
$$\int_{E_\R} e^{- \pi \Vert x \Vert^2} \, dm_{\Eb}(x) =1,$$
by approximating this Gaussian integral by Riemann sums over the lattice $\sqrt{t} E$, where $t\in \Rpa$ goes to zero, we obtain:
$$\lim_{t \lra 0_+} \sqrt{t}^{\rk E}  \covol(\Eb) \sum_{v\in E} e^{- \pi t \Vert v\Vert^2} = 1,$$
or equivalently:
\begin{equation}\label{asymptheta}
\log \theta_\Eb(t) = - (\rk E)/2 \,   \log t + \dega \Eb + o(1) \quad \mbox{when $t\lra 0_+$}.
\end{equation}
If we let  $\lambda = -(1/2) \log t,$ we get:
$$\hot(\Eb \otimes \cOb(\lambda)) = \rk E \,  \lambda + \dega \Eb + \epsilon (\lambda)  \quad \mbox{where $\lim_{\lambda \ra + \infty} \epsilon(\lambda) = 0.$}$$

In this formulation, the asymptotic expression  (\ref{asymptheta}) for $\theta_\Eb(t)$ when $t$ goes to $0_+$ becomes the analogue of the expression (\ref{HilPol}) for the Hilbert polynomial of a vector bundle over some curve of genus  $g=1$. 
 
 The expression (\ref{asymptheta}) is also a consequence of Poisson formula  (\ref{Poisson4}),  which actually shows that the error term 
$\epsilon(\lambda)$ decreases extremely fast at infinity. Namely, there exists $c \in \Rpa$ such that:% $\lambda \lra + \infty$,
 $$\epsilon(\lambda) = O(e^{- c e^{\lambda^2}}) \quad \mbox{when $\lambda \lra + \infty$}. $$

\subsubsection{Admissible short exact sequences and theta invariants}
A further analogy between the properties of  $\hot(\Eb)$ and of  $h^0(C,E)$ concerns their compatibility with direct sums and, more generally, with short exact sequences: 
 
 \begin{proposition}\label{thetasubadd} 1) For any two Euclidean lattices $\Eb_1$ and $\Eb_2$, we have: 
  \begin{equation}\label{hotsum}
\hot(\Eb_1 \oplus \Eb_2) = \hot(\Eb_1) + \hot(\Eb_2).
\end{equation}

2) For any admissible short exact sequences of Euclidean lattices
$$0 \lra \Fb \stackrel{i}{\lra} \Eb \stackrel{p}{\lra} \overline{E/F} \lra 0,$$ 
we have :
\begin{equation}\label{ineqsumdir}
\hot(\Eb) \leq \hot(\Fb) + \hot(\overline{E/F}).
\end{equation}
\end{proposition}

 The subadditivity inequality   (\ref{ineqsumdir}) has been observed by Quillen (\cite{QuillenNotebooks}, entry of  04/26/1973) and Groenewegen (\cite{Groenewegen2001}, Lemma 5.3).
 
 \begin{proof}
1)  The equality  (\ref{hotsum})  follows from the relation:
 $$\sum_{(v_1, v_2) \in E_1 \times E_2} e^{- \pi (\Vert v_1\Vert^2_{\Eb_1} + \Vert v_2\Vert^2_{\Eb_2})}=
 \sum_{v_1 \in E_1} e^{- \pi \Vert v_1\Vert^2_{\Eb_1}} . \sum_{v_2 \in E_2} e^{- \pi \Vert v_2\Vert^2_{\Eb_2}}.$$
  
2) Observe that  the Poisson formula in the form (\ref{Poisson5}) shows that, for any $\alpha \in E/F$, 
 $$\sum_{ v \in p^{-1}(\alpha)} e^{- \pi \Vert v \Vert^2_{\Eb}} 
 \leq e^{- \pi \Vert \alpha \Vert^2_{\overline{E/F}}} \, \sum_{f \in F} e^{- \pi \Vert f \Vert_\Fb^2}.$$
By summing over $\alpha$ and taking the logarithms, we obtain (\ref{ineqsumdir}). 
\end{proof}

\subsubsection{A theorem of  Regev and Stephens-Davidowitz} 

Let $C$ be some smooth projective curve over some base field $k$, as in paragraph \ref{fibresvect}. Let $E$ be some vector bundle over  $C$, and  $F_1$ and  $F_2$ two coherent subsheaves of $E$. We may define the following short exact sequence of vector bundles over $C$:
$$0 \lra F_1 \cap F_2 \lra F_1 \oplus F_2 \lra F_1 +F_2 \lra 0,$$
where the morphism from $ F_1 \oplus F_2$ to $F_1 +F_2$ is the sum map, and the one from $F_1 \cap F_2$ to $F_1 \oplus F_2$ maps a section $s$ to $(s,-s)$. It induces an exact sequence of finite dimensional $k$-vector spaces:
$$0 \lra H^0(C, F_1 \cap F_2) \lra H^0(C, F_1) \oplus H^0(C,F_2) \lra H^0(C, F_1 + F_2),$$
which yields the following inequality concerning their dimensions: 
$$h^0(C, F_1) +h^0(C,F_2) \leq h^0(C, F_1 \cap F_2) + h^0(C, F_1 + F_2).$$

Answering a question by  McMurray Price (see \cite{McMurrayPrice2017}), Regev and Stephens-Davidowitz have shown that this estimate holds  \emph{ne varietur} for Euclidean lattices and their invariants~$\hot$:

\begin{theorem}[\cite{RSD17a}] Let $\Eb:= (E, \Vert. \Vert)$ be a Euclidean lattice and let  $F_1$ and $F_2$ two $\Z$-submodules of $E$. Then the following estimate holds:
$$\hot(\Fb_1) + \hot(\Fb_2) \leq \hot(\overline{F_1 \cap F_2}) + \hot(\overline{F_1 + F_2}).$$
\end{theorem}

Here we denote by $\Fb_1$, $\Fb_2,$ $\overline{F_1 \cap F_2},$ and $\overline{F_1 + F_2}$ the Euclidean lattices defines by the free $\Z$-modules $F_1$, $F_2,$ ${F_1 \cap F_2},$ and $F_1 + F_2$ equipped with the restrictions of the Euclidean norm $\Vert. \Vert$.

We refer the reader to  \cite{RSD17a} for the proof of this theorem, which is elementary but extremely clever.

\subsection{Varia} As explained in \cite{Bost2018}, it is possible to understand the remarkable recent results of Dadush, Regev and Stephens-Davidowitz on the Kannan-Lov\'asz conjecture (\cite{DadushRegev2016}, \cite{RSD17b}) as a further illustrations of the analogy between the invariants $h^i_\theta(\Eb)$, $i\in \{0, 1\},$ associated to Euclidean lattices and the dimensions of cohomology groups  $h^i(C, E)$ associated to vector bundles on curves (see notably \emph{loc. cit.}, section 5.4).

Let us also recall that Euclidean lattices are nothing but a special case, associated to the field  $K = \Q$, of Hermitian vector bundles over the ``arithmetic curve"  $\Spec \OK$, attached to some number field $K$ of ring of integers $\OK$. The analogy between vector bundles over a curve and Euclidean lattices extends to an analogy between vector bundle over a curve and Hermitian vector bundles over  $\Spec \OK$, where  $K$ is now an arbitrary number field. 

This actually constitutes the natural framework for this analogy: considering arbitrary number fields is akin to considering curves  $C$ of arbitrary genus  $g \geq 1$.  The definitions of the invariants $h^i_\theta(\Eb)$ extend to this setting, that already was, in substance, the one of  \cite{Hecke17}. We refer to  \cite{BostTheta}, Chapter 2, for their study in this more general framework.

Let us finally emphasize that the theta series  (\ref{thetadef}) associated to Euclidean lattices appear in various areas of mathematics and mathematical physics and have led to multiple developments, from very diverse perspectives.
We may notably mention the investigations of extremal values of theta functions, motivated by the classical theory of modular and automorphic forms (see for instance 
\cite{SarnakStr2006} and its references), the works on the  ``Gaussian core model", inspired by the study of sphere packing and statistical physics (\cite{CohndeCourcy2016}), and various developments in crystallography and solid state physics  (see for instance \cite{BeterminPetrache2017}).

\section{A  mathematical model of the thermodynamic formalism}\label{Thermod}

In this section, logically independent of the previous ones, we introduce a simple mathematical model of classical statistical physics and we establish some of its basic properties.

In this model, the central object of study is a pair $((\cE, \cT, \mu), H)$, consisting of some measure space  $(\cE, \cT, \mu)$ equipped with some non-negative measurable function $H:\cE \lra \R_+$.
The measure space $(\cE, \cT, \mu)$ should be thought as the configuration space (or phase space) of some elementary physical system, and the function $H$ for the energy function on this space. 

For instance, when dealing with some physical system in the realm of classical mechanics, described by the Hamiltonian formalism,  $(\cE, \cT, \mu)$ will be the measure space underlying a symplectic manifold $(M, \omega)$, of dimension $2n$, equipped with the Liouville measure defined by the top degree form 
$$\mu := \frac{1}{n !}\,  \omega^n,$$
and $H$ the function in $C^\infty(M, \R)$ such that the associated Hamiltonian vector field $X_H$ on $M$, defined
by 
$$i_{X_H} \omega = dH,$$ describes the evolution of the system (see for instance \cite{AbrahamMarsden78} or \cite{Arnold89}).

 Besides such examples related to classical mechanics for the pair $((\cE, \cT, \mu), H)$,  other examples, of  physical and number theoretical origin, will turn out to be interesting (see \ref{quantum} and \ref{AppLatt}, \emph{infra}).  However we shall still refer to the function $H$ as the Hamiltonian or as the energy  of the system under study. 
 
\subsection{Measure spaces with a Hamiltonian: basic definitions}\label{MesH}

Let us consider a measure space $(\cE, \cT,\mu)$ defined by a set $\cE$, a $\sigma$-algebra $\cT$ of subsets of  $\cE$, and   a non-zero $\sigma$-finite  measure $$\mu : \cT \lra [0,+\infty].$$

Besides, let us consider some $\cT$-measurable function $$H : \cE \lra \R_+.$$
We shall denote by $\Hm$ its essential infimum with respect to $\mu$ and introduce the $\cT$-measurable subset of $\cE$ :
$$\cEm := H^{-1}(\Hm).$$

We may introduce the following two conditions:

${\mathbf T_1}:  $  $\mu(\cE) = + \infty,$

\noindent and:

${\mathbf T_2}:  $ \emph{For every $E \in \R_+,$ the measure 
$$N(E) := \mu( H^{-1} ([0, E]))$$
is finite and, when $E$ goes to $+\infty$, 
}
$$\log N(E) = o(E).$$

The condition ${\mathbf T_2}$ on the finiteness and the subexponential growth of $N$ is easily seen to be equivalent to following condition:

${\mathbf T'_2}:  $ \emph{For every $\beta \in \Rpa$, the function $e^{- \beta H}$ is $\mu$-integrable.}

When ${\mathbf T_2}$ and ${\mathbf T'_2}$ are satisfied, we may introduce the \emph{partition function} (in German, \emph{Zustandsumme})
$$Z: \Rpa \lra \Rpa$$
and \emph{Planck's characteristic function} 
$$\Psi: \Rpa \lra \R$$ defined by the relations, for any $\beta \in \Rpa$:
\begin{equation}\label{defZ}
Z(\beta) := \int_\cE e^{-\beta H} \, d\mu
\end{equation}
and 
\begin{equation}\label{defPsi}
\Psi(\beta) := \log Z(\beta)
\end{equation}

When the measure $\mu$ is a probability measure, the function $\Psi(-\beta)$ also appears in the literature as the \emph{logarithmic moment generating function} (see for instance \cite{Stroock2011}, Section 3.1.1).

\subsection{Main theorem}\label{subsMainTh}

 For every positive integer $n,$ we  also consider the product $\mu^{\otimes n}$ of $n$ copies of the measure $\mu$ on $\cE^n$ equipped with the $\sigma$-algebra $\cT^{\otimes n}$, and we define the $\cT^{\otimes n}$-measurable function
 $$H_n : \cE^n \lra \R_+$$
 defined by:
 $$H_n(x_1, \ldots, n) := H(x_1) + \cdots + H(x_n)$$
 for every $(x_1, \cdots, x_n) \in \cE^n.$
 
 For every $E \in \R$ and every integer $n \geq 1,$ we may consider 
 \begin{equation}\label{Andef}
 A_n(E) := \mu^{\otimes n} (\{x \in \cE^n \mid H_n(x) \leq nE\}).
 \end{equation}
 
 Clearly, $A_n(E)$ is a non-decreasing function of $E$; moreover:
 $$A_n(E) = 0 \quad \mbox{if} \quad E< \Hm$$
 and 
 $$A_n(\Hm) = \mu^{\otimes n}(\cEm^n) = \mu(\cEm)^n.$$
(Observe that the measure $\mu(\cEm)$ is finite when ${\mathbf T_2}$ holds.) 
Besides, when $E > \Hm,$ 
$$A_n(E) \geq \mu^{\otimes n}((H^{-1}([0, E]))^n) = N(E)^n >0.$$

The following theorem, which constitutes the main result in these notes, describes the asymptotic behavior  of $A_n(E)$ when $n$ goes to infinity. It shows that
$$A_n(E) = e^{(n+ o(n)) S(E)} \quad \mbox{when $n \ra + \infty$},$$
for some real valued function $S$ on $(\Hm, +\infty)$ that is deduced from $\Psi$ by Legendre-Fenchel transform.

\begin{theorem}\label{ThMain} Let us assume that Conditions ${\mathbf T_1}$ and ${\mathbf T_2}$ are satisfied.

1) For any $E \in (\Hm, +\infty)$ and any integer $n \geq 1$, $A_n(E)$ belongs to $\Rpa$ and the limit 
\begin{equation}\label{Slim}
S(E) := \lim_{ n \ra + \infty} \frac{1}{n} \log A_n(E)
\end{equation}
exists in $\R,$ and actually coincides with $\sup_{n \geq 1} (1/n) \log A_n(E).$

The function 
$$S: (\Hm, +\infty) \lra \R$$
is real analytic, increasing and strictly concave\footnote{In other words, for any $E \in (\Hm, +\infty),$ $S'(E) >0$ and $S''(E) <0.$}, and satisfies:
\begin{equation}\label{limSmin} 
\lim_{E \ra (\Hm)_+} S(E) = \log \mu (\cEm)
\end{equation}
and 
\begin{equation}\label{limSinfty}
\lim_{E \ra + \infty} S(E) = + \infty. 
\end{equation}
Moreover its derivatives establishes a real analytic decreasing diffeomorphism:
\begin{equation}\label{Sprdiff}
S' := (\Hm, +\infty) \lrasim \Rpa.
\end{equation}

2) The function $\Psi: \Rpa \lra \R$ is real analytic, decreasing and strictly convex\footnote{In other words, for any $\beta \in \Rpa,$ $\Psi'(\beta) < 0$ and $\Psi''(\beta) >0$.}. 
Its derivative up to a sign
$$U:= -\Psi'$$
defines a real analytic decreasing diffeomorphism
\begin{equation}\label{Udiff}
U : \Rpa \lrasim
(\Hm, +\infty)
\end{equation}
and satisfies, for every $\beta \in \Rpa$:
\begin{equation}\label{UH}
U(\beta) = \frac{\int_\cE H \, e^{-\beta H} \, d\mu}{\int_\cE e^{-\beta H} \, d\mu}.
\end{equation}

3) The functions $-S(- .)$ and $\Psi$ are Legendre-Fenchel transforms of each other. 

Namely, for every $E \in (\Hm, +\infty),$
\begin{equation}\label{LegendreS}
S(E) = \inf_{\beta \in \R^\ast_+} (\Psi(\beta)  + \beta E ),
\end{equation}
and, for every $\beta \in \R^\ast_+,$
\begin{equation}\label{LegendrePsi}
\Psi(\beta) = \sup_{E \in ({\inf}_\mu H, +\infty)} (S(E) -\beta E).
\end{equation}
Moreover the diffeomorphisms $S'$ and $U$ \emph{(see (\ref{Sprdiff}) and (\ref{Udiff}))} are inverse of each other. For any $E \in (\Hm, +\infty)$ and any $\beta \in \Rpa,$ the following inequality holds:
\begin{equation}\label{ineqSPsi}
S(E) \leq \Psi(\beta) + \beta E,
\end{equation}
and (\ref{ineqSPsi}) becomes an equality precisely when 
\begin{equation}\label{betaE}
\beta = S'(E),  \quad \mbox{or equivalently} \quad E= U(\beta).
\end{equation}

%For any $E \in (\Hm, +\infty)$, the infimum in the right hand side of (\ref{LegendreS}) is attained at a unique point $\beta$, namely $\beta = S'(E).$ Dually, for any $\beta \in \Rpa,$ the supremum in (\ref{LegendrePsi}) is attained at a unique point $E$, namely $E = U(\beta)$.
 \end{theorem}
 
 In \cite{BostTheta}, Appendix A, a more general form of this result is presented, with a strong emphasis on its relation with Cram\'er's fundamental theorem on large deviations (which deals with the situation where $\mu$ is a probability measure). In particular, Theorem  \ref{ThMain} is established  in \cite{BostTheta} by some reduction to Cram\'er's theorem. (Theorem  \ref{ThMain} appears as Theorem A.5.1 in \emph{loc. cit.}; it is a special case of Theorem A.4.4, which extends Cram\'er's theorem by a means of a reduction trick discussed in Section A.4.1.)  
 
The main purpose of the final part of these notes is to present some self contained derivation of Theorem \ref{ThMain}, that hopefully will make clear the basic simplicity of the underlying arguments. These arguments will also show that there is a considerable flexibility in the definition (\ref{Andef}) of $A_n(E)$ that ensures the validity of Theorem \ref{ThMain}: diverse variants, where the conditions $H_n(x) \leq n E$ is replaced by  stronger conditions --- for instance  $H_n(x) < n E$, or $(1-\eta) n E \leq H_n(x) \leq n E$ for some fixed $\eta \in (0,1)$ --- would still lead to to the convergence of $(1/n) \log A_n(E)$ towards the same limit $S(E)$.

For instance, in paragraph \ref{PrSvar}, we will derive the following result, which notably covers the above variants of the condition  $H_n(x) \leq n E$:

\begin{proposition}\label{Svar}
Let us keep the notation of Theorem \ref{ThMain}.

For any $E \in (\Hm, +\infty)$ and any sequence $(I_n)_{n \geq n_0}$ of intervals in $\R$ such that
$$\lim_{n \ra +\infty} \, \sup I_n = E$$
and such that their lengths $l_n$ \emph{(that is, their Lebesgue measure)} satisfy
$$\liminf_{n \ra +\infty} (\sqrt{n}\, l_n )> 2 \sqrt{\Psi''(\beta)}  \quad \mbox{where $\beta := S'(E) = U^{-1}(E)$},$$
we have:
\begin{equation*}
\lim_{n \ra +\infty} \frac{1}{n} \log \mu^{\otimes n}(\{ x \in \cE^n \mid H_n(x) \in n I_n \}) = S(E).
\end{equation*}
\end{proposition}

\subsection{Relation with statistical physics}\label{Phys} We want now to explain briefly how the framework of measure spaces with Hamiltonian and the main theorem presented this section are related to %may be understood as a mathematical model of 
the formalism of %statistical
thermodynamics.

These relations go back to the classical works of Boltzmann and Gibbs (see for instance
\cite{Boltzmann72}, 
\cite{Boltzmann77}, and
\cite{Gibbs02}). However the present discussion is more specifically related to the approach to statistical thermodynamics presented in Schr\"odinger's seminar notes \cite{Schroedinger52}, and also, to  a lesser extent, to Khinchin's exposition in \cite{Khinchin49}. We  refer to \cite{Ellis85} for another perspective on the relations between statistical mechanics and large deviations.

\subsubsection{} Let us consider a physical system consisting of a large number $n$ of identical ``elementary" systems, each of them described by a measure space equipped with some Hamiltonian function $H$. One assumes that these systems are ``loosely" coupled: their coupling is assumed to allow the exchange of energy between these elementary systems; however, we suppose that these exchanges are negligible with respect to the internal dynamics of the elementary systems. 

One is interested in the ``average properties" of these elementary systems when the total energy of our composite system belongs to some ``small interval" $[n(E-\delta E), nE]$, or in other words, when the average energy per elementary system is about $E$. Theorem \ref{ThMain} and its variant Proposition \ref{Svar} provide the following answer to this type of question: they show that the measure  
\begin{equation}\label{W}
W(E):= \mu^{\otimes n} \left( \{ x \in \cE^n \mid E-\delta E \leq H_n(x)/n \leq E \} \right)
\end{equation}  
of the points in the phase space $\cE^n$ describing states of our composite systems which satisfy the above energy condition grows like 
$$\exp(n (S(E) + o(1))$$
when $n$ goes to infinity. 

In other words, when $n$ is large, 
$$n^{-1} \log W(E) = S(E) + o(1),$$
and $S(E)$ may be understood as the logarithmic volume of the phase space available ``per elementary system" when there average energy is about $E$: it is the \emph{Boltzmann entropy}, at the energy $E$, attached to the elementary system described by the measure space with Hamiltonian $((\cE, \cT, \mu), H)$.

Its expression (\ref{Udiff})  shows that $U(\beta)$ may be interpreted as the average value of the Hamiltonian of our elementary system computed by using Gibbs canonical distribution\footnote{namely, the probability measure $\nu_\beta := Z(\beta)^{-1} e^{-\beta H} \mu$; see also paragraph \ref{nubeta}, \emph{infra}.} at temperature $\beta^{-1}$ --- in brief, to its energy at temperature $\beta^{-1}$. 

Part 3) of Theorem \ref{ThMain} relates the Boltzmann entropy of our system, defined as the limit (\ref{Slim}), to its energy $U(\beta)$ as a function of $\beta$.  As stated in (\ref{betaE}), to every value of $\beta$ in $\Rpa$ is attached bijectively a value $E= U(\beta)$ in $(\Hm, +\infty)$ of this energy, and we then have:
$$\beta = S'(E).$$
If we let  
$\beta = 1/T,$ this last relation takes the familiar form:
$$dS = \frac{dE}{T}.$$
The function $\Psi(\beta)$ satisfies
$$\Psi(\beta) = S(E) -\beta E,$$
where, as before, $E = U(\beta).$ It coincides with the function initially introduced by Planck\footnote{Planck initially denoted this function by $\Phi$.  The notation $\Psi$  seems to have been introduced in the English translation \cite{Planck03} of Planck's classical treatise on thermodynamics, and is also used by Schr\"odinger in  \cite{Schroedinger52}.} as
$$\Psi := \frac{TS -U}{T} = -\frac{F}{T},$$
where $F:= U -TS$ is the so-called Helmoltz free energy.  

\subsubsection{}\label{quantum} The general framework introduced in \ref{MesH} also covers the thermodynamics of some quantum systems, namely of systems composed of some large number $n$ of copies of some  ``elementary" quantum system  described by a (non-negative selfadjoint) Hamiltonian operator $\cH$ acting with a discrete spectrum on some Hilbert space, say $L^2(X)$. 

To the data of such an elementary system are indeed associated the spectrum $\cE$ of $\cH$ (a discrete subset of $\R_+$), the spectral measure $\mu$ of $\cH$ --- defined by the relation
$${\rm Tr}_{L^2(X)}\,  f(\cH) = \int_\cE f \, d\mu$$
for any finitely supported function $f$ on $\cE$ --- and the ``tautological" function $$H: \cE \,\hlra\,  \R_+,$$ defined by the inclusion  of the spectrum $\cE$ in $\R_+$.  Then the partition function $Z(\beta)$ associated to the so-defined measure space with Hamiltonian is 
$$Z(\beta) := \int_\cE e^{- \beta H} \, d\mu = {\rm Tr}_{L^2(X)} e^{-\beta \cH},$$ and the previous discussion still holds \emph{mutatis mutandis}: $A_n(E)$ now represents the number of ``quantum states" of our composite system of total energy at most $nE$, etc.  

A remarkable instance of this situation is provided by the quantum harmonic oscillator, say of frequency $\nu,$ described by the Hamiltonian operator 
$$\cH := - (h^2/2) \frac{d^2}{dx^2} + (\nu^2/2) x^2$$
acting on $L^2(\R).$

Then $\cE = (\N + 1/2) h\nu,$ and $\mu$ is the counting measure $\sum_{e \in \cE} \delta_e$. The associated partition function is 
$$Z(\beta) = \sum_{k \in \N} e^{- \beta (k + 1/2) h \nu} = \frac{e^{-\beta h \nu/2}}{1 - e^{- \beta h \nu}}$$
and accordingly:
$$\Psi(\beta) := \log Z(\beta) = - (1/2)\, \beta h \nu- \log(1- e^{- \beta h \nu}).$$
Consequently, we then have:
$$U(\beta) = \frac{h\nu}{2} + \frac{h\nu \,e^{-\beta h \nu}}{1- e^{-\beta h \nu}}.$$
We recover Planck's formula for the energy of a quantum oscillator of frequency $\nu$ at temperature $\beta^{-1}$, which constitutes the historical starting point of quantum physics.  

\subsection{Gaussian integrals and Maxwell's kinetic gas model}\label{GaussMaxwell}

In this subsection, we discuss an instance of the formalism introduced in \ref{MesH} and of Theorem \ref{ThMain} that may be seen as a mathematical counterpart of Maxwell's statistical approach to the theory of ideal gases. It is  included for comparison with the application in Subsection \ref{AppLatt} of the above formalism to Euclidean lattices --- the present example appears as a ``classical limit" of the discussion of Section \ref{AppLatt}.
 
\subsubsection{Euclidean spaces and Gaussian integrals}\label{EucGauss}

We begin by a purely mathematical discussion.

 Let $V$ be a finite dimensional real vector space equipped with some Euclidean norm $\Vert. \Vert.$
 
 We shall denote by $\lambda$ the Lebesgue measure on $V$ attached to this Euclidean norm. It may be defined as the unique translation invariant Radon measure on $V$ which satisfies the following normalization condition: for any orthonormal base $(e_1, \cdots, e_N)$ of the Euclidean space $(V, \Vert.\Vert)$,
 $$\lambda \left(\sum_{i = 1}^N [0,1) e_i\right) = 1.$$ This normalization condition may be equivalently expressed in terms of a Gaussian integral:
 \begin{equation*}
\int_V  e^{-\pi \Vert x \Vert^2} \, d\lambda(x) = 1.
\end{equation*}

We may apply the formalism of this appendix to the measure space $(V, \cB, \lambda)$, defined by $V$ equipped with the Borel $\sigma$-algebra $\cB$ and with the Lebesgue measure $\lambda$, and to the function
$$H := (1/2m) \Vert. \Vert^2$$
where $m$ denotes some positive real number.

Then, for every $\beta$ in $\R^\ast_+,$ we have:
\begin{equation}\label{Gaussp}
\int_V e^{- \beta \Vert p \Vert^2/2m} \, d\lambda(p) = (2 \pi m/\beta)^{\dim V / 2}.
\end{equation}
Therefore
\begin{equation}\label{PsiMaxw}\Psi(\beta) =  (\dim V/2)  \, \log (2 \pi m/\beta)
\end{equation}
and 
\begin{equation}\label{UMaxw}
U(\beta) = -\Psi'(\beta) = \dim V/(2\beta).
\end{equation}
 
The relation (\ref{betaE}) between the  ``energy"
$E$ and the ``inverse temperature" $\beta$, takes the following form, for any $E\in (\Hm, +\infty) = \R^\ast_+$ and any $\beta \in  \R^\ast_+$:
\begin{equation}\label{betaxMaxw}
\beta E = (\dim V)/2.
\end{equation}

The function $S(E)$ may be computed directly from its definition. 

Indeed, for any $E \in \R^\ast_+$ and any positive integer $n$, we have:
\begin{equation}\label{S1}
\lambda^{\otimes n} \left(\{(x_1, \ldots, x_n) \in V^n \mid (1/2m)(\Vert x_1\Vert^2+\ldots+\Vert x_n\Vert^2) \leq n E \}\right)
= v_{n\dim V} (2mnE)^{n (\dim V)/2}.
\end{equation} 
Here $v_{n\dim V}$ denotes the volume of the unit ball in the Euclidean space of dimension $n\dim V$. It is given by:
\begin{equation}\label{S2} v_{n\dim V} = \frac{\pi^{n (\dim V)/2}}{\Gamma(1 + {n (\dim V)/2})}.
\end{equation} 

From (\ref{S1}) and (\ref{S2}), by a simple application of Stirling's formula, we get:
$$S(E) = \lim_{n \ra + \infty} \frac{1}{n} \log\left[ v_{n\dim V} (2mnE)^{n (\dim V)/2} \right] = (\dim V/2) [1 + \log(4 \pi m E /\dim V)].$$
In particular%, $$S'(E) = (\dim V)/(2E),$$
$$S'(E) = \frac{\dim V}{2E},$$
 and we recover (\ref{betaxMaxw}).

Conversely,  combined with the expression (\ref{PsiMaxw}) for the function $\Psi,$ Part 3) of Theorem \ref{ThMain} allows one to recover the asymptotic behaviour of the volume $v_n$ of the $n$-dimensional unit ball, in the form:
$$v_n^{1/n} \sim \sqrt{2e\pi/n} \;\; \mbox{ when $n \ra + \infty.$} $$ 

Finally, observe that when $m = (2\pi)^{-1}$ --- the  case relevant for the comparison with the application to Euclidean lattices in Section  \ref{AppLatt} --- the expressions for $\Psi$ and $S$ take the following simpler forms: 
$$\Psi(\beta) = (\dim V/2)  \, \log (1/\beta)$$
and 
$$S(E) = (\dim V/2) [1 + \log(2E /\dim V)].$$

\subsubsection{Hamiltonian dynamics on a compact Riemannian manifold}\label{HRM}

Let $X$ be a compact $C^\infty$ manifold of (pure positive) dimension $d$, and let $g$ be a $C^\infty$ Riemannian metric on $X$. 

We shall denote the tangent (resp. cotangent) vector of $X$ by $T_X$ (resp. by $T_X^\vee$). The Riemannian metric $g$ defines an isomorphism of $C^\infty$ vector bundles
\begin{equation}\label{IsoR}
T_X \lrasim T_X^\vee,
\end{equation}
by means of which the Euclidean metric $\Vert . \Vert$ on the fibers of $T_X$ defined by $g$ may be transported into some Euclidean metric $\Vert. \Vert^\vee$ on the fibers of $T_X^\vee.$ 

The $2d$-dimensional $C^\infty$ manifold $T_X^\vee$ is endowed with a canonical symplectic form $\omega$, defined as the exterior differential $d\alpha$ of the tautological $1$-form $\alpha$ on $T_X^\vee$, which is characterized by the following property  (see for instance \cite{AbrahamMarsden78} or \cite{Arnold89}): for any $C^\infty$ function $f$ on some open subset $U$ of $X$, the differential of $f$ defines a section $Df$ over $U$ of the structural morphism
%\begin{equation}\label{defpi}
$$\pi: T_X^\vee \lra X$$
%\end{equation}
of the cotangent bundle $T_X^\vee$, and the $1$-form $Df^\ast \alpha$ over $U$, defined as the pull-back of $\alpha$ by this section, coincides with the differential $df$ of $f$.

%\footnote{See for instance \cite{AbrahamMarsden78} or \cite{Arnold89}): for any $C^\infty$ function $f$ on some open subset $U$ of $X$, the differential of $f$ defines a section $Df$ over $U$ of the structural morphism
%\begin{equation}\label{defpi}
%\pi: T_X^\vee \lra X
%\end{equation}
%of the cotangent bundle $T_X^\vee$, and the $1$-form $Df^\ast \alpha$ over $U$, defined as the pull-back of $\alpha$ by this section coincides with the differential $df$ of $f$.}.

Let $m$ be a positive real number. The Hamiltonian flow associated to the function
\begin{equation}\label{Hamfree}
H:= \frac{1}{2m} \Vert .\Vert^{\vee 2} : T_X^\vee \lra \R
\end{equation}
on the symplectic manifold $(T_X^\vee, \omega)$ describes the dynamics of some particle of mass $m$ moving freely on the Riemaniann manifold $(X,g)$. When $m=1,$ this flow transported to $T_X$ by means of (the inverse of) the diffeomorphism (\ref{IsoR}) is the geodesic flow of the Riemannian manifold $(X, g)$ (see for instance \cite{AbrahamMarsden78}, Section 3.7).

The %statistical mechanics
thermodynamics of a kinetic gas model composed  of free particles of mass $m$ on $(X,g)$ is described by the formalism of paragraphs 
\ref{MesH} and \ref{subsMainTh} applied to $\cE := T_X^\vee$ equipped with the Liouville measure $\mu := \omega^d/d!$ and with the Hamiltonian function $H$ defined by (\ref{Hamfree}).  

More generally, one may consider some $C^\infty$ function 
$$V: X \lra \R$$
and introduce the Hamiltonian function
$$H_V: T_X^\vee \lra \R$$
defined by
$$H_V(p):= \frac{1}{2m} \Vert p \Vert^{\vee 2} + V(x)$$
for any point $x$ in $X$ and any $p$ in the fiber $T_{X,x}^\vee := \pi^{-1}(x)$ of the cotangent bundle over $x$.

It describes a particle of mass $m$ moving on the Riemannian manifold $(X,g)$, submitted to the potential $V$, and our formalism applied to 
$$(\cE, \mu, H) := (T_X^\vee, \omega^d/d!, H_V)$$
describes the thermodynamics of a gas of such particles.

\subsubsection{Euclidean lattices and flat tori}\label{Flattori}

let $\Fb$ be some Euclidean lattice, and let $(X,g)$ be the compact Riemannian manifold defined as the flat torus associated to the 
dual Euclidean lattice $\Fb^\vee := (F^\vee, \Vert . \Vert^\vee);$ namely,
$$X := F^\vee_\R/F^\vee$$
and $g$ is the flat  metric on
$$T_X \simeq (F^\vee_\R/F^\vee) \times F^\vee_\R$$
defined by the ``constant" Euclidean norm $\Vert. \Vert^\vee$ on  $F^\vee_\R.$

Then the function $H$ on 
$$T_X^\vee \simeq (F^\vee_\R/F^\vee) \times F_\R$$
is simply the composition 
$$H: T_X \stackrel{{\rm pr}_2}{\lra} F_\R \stackrel{\Vert.\Vert^2/2m}{\lra} \R,$$
where ${\rm pr}_2$ denotes the projection of $(F^\vee_\R/F^\vee) \times F_\R$ onto its second factor. Moreover the Liouville measure $\mu$ on $T_X^\vee$ is nothing but the product of the translation invariant measure on $F^\vee_\R/F^\vee$ deduced from the Euclidean metric $\Vert. \Vert^\vee$ --- its total mass is
$$\covol \Fb^\vee =(\covol \Fb)^{-1}$$
--- and of the normalized Lebesgue measure on the Euclidean space $(E_\R, \Vert.\Vert)$.

Accordingly, in this situation, the triple $(T_X^\vee, \omega^d/d!, H_V)$ introduced in \ref{HRM} coincides --- up to some ``trivial factor" $F^\vee_\R/F^\vee$ --- with the one associated in  \ref{EucGauss} to the Euclidean space $(V, \Vert.\Vert) = (F_\R, \Vert. \Vert).$

In this way, we derive the following expressions for its partition and characteristic functions:
$$Z(\beta) = (\covol \Fb)^{-1} \, (2 \pi m /\beta)^{d/2}$$
and 
\begin{equation}\label{PsibetaMaxw}
\Psi(\beta) :=   (d/2) \log (2 \pi m /\beta) + \dega \Fb.
\end{equation}
Consequently, the expression for the energy $U(\beta)$ is unchanged:
\begin{equation}\label{UbetaMaxw}
U(\beta) = \frac{d}{2\beta}
\end{equation}
and the entropy function is 
$$S(E) = \frac{d}{2} [1 + \log (4 \pi m E/ d)] + \dega \Fb.$$
We recover the classical formulae describing the kinetic theorey of an ideal gas of particles of mass $m$ in the ``box with periodic boundary conditions" described by the flat torus associated to $\Fb^\vee$. 

Observe that when $m= (2\pi)^{-1}$ --- the case relevant for the comparison with the application to Euclidean lattices discussed in Paragraph \ref{AppLatt} \emph{infra} -- the above  expressions for $\Psi$ and $S$ take the following simpler forms:
\begin{equation*}
\Psi(\beta) = (d/2) \log \beta^{-1} +  \dega \Fb 
\end{equation*}
 and 
 \begin{equation*}
S(E) = (d/2)  [1 + \log (2E/ d)] + \dega \Fb.
\end{equation*}

{\small 
\subsubsection{Maxwell's kinetic gas model on a compact Riemannian manifold}

Let us return to the situation of a compact Riemannian manifold $(X,g)$, equipped with some potential function $V$, introduced in \ref{HRM}.

For simplicity, let us assume that $X$ is oriented, and let us denote by $\lambda_g$ the volume form on $X$ associated to the Riemannian metric $g$ (it is a $C^\infty$ form of degree $d$, everywhere positive). The expression (\ref{Gaussp}) for the Gaussian integrals on some Euclidean vector space admits a straightforward ``relative" version, concerning the projection map $\pi: T_X^\vee \lra X,$ namely:
$$\pi_\ast (e^{-\beta \Vert . \Vert^{\vee 2}/ 2m} \,\omega^d/ d!) = (2\pi m / \beta)^{d/2}\, \lambda_g.$$
(Here  $\pi_\ast$ denotes the operation of integration of differential forms along the fibers of $\pi$.) 

This immediately implies that the partition function associated to $(\cE, \mu, H) := (T_X^\vee, \omega^d/d!, H_V)$ is:
$$Z(\beta) := \int_{T_X^\vee} e^{-\beta H_V} \, \omega^d/d!  = (2\pi m / \beta)^{d/2} \, \int_X e^{-\beta V} \, \lambda_g.$$
Therefore its characteristic function is
$$\Psi (\beta) = (d/2) \log (2 \pi m /\beta) + \log \int_X e^{-\beta V} \, \lambda_g$$ and the energy function is given by:
$$U(\beta) = \frac{d}{2\beta} + \frac{\int_X e^{-\beta V} V \, \lambda_g}{\int_X e^{-\beta V} \, \lambda_g}.$$

When the potential $V$ vanishes, we recover the same expression as the ones (\ref{PsibetaMaxw}) and (\ref{UbetaMaxw}) previously derived for flat tori, with $\dega \Fb$ replaced by $\log {\rm vol}(X,g)$  in (\ref{PsibetaMaxw}), where  $${\rm vol}(X,g) := \int_X \lambda_g$$ denotes the volume of the Riemannian manifold $(X,g)$. For a general potential $V$, we have:
$$\frac{d}{2\beta} + \min_X V \leq U(\beta) \leq \frac{d}{2\beta} + \max_X V.$$

}

\subsection{Application to Euclidean lattices: proof of Theorem \ref{Thhonthot}}\label{AppLatt} Let us finally now discuss how Theorem \ref{Thhonthot} may be derived from Theorem \ref{ThMain}.

\subsubsection{}\label{MSHE} Let us consider some Euclidean lattice $\Eb:= (E, \Vert.\Vert)$ of positive rank. To $\Eb$ is attached the measure space with Hamiltonian $((\cE, \cT, \mu), H)$ defined as follows:
\begin{equation}\label{MSHE1}
\cE := E,
\end{equation}
\begin{equation}\label{MSHE2}
\cT := \cP(E),
\end{equation}
\begin{equation}\label{MSHE3}
\mu :=\sum_{v \in E} \delta_v
\end{equation}
--- in other words, $(\cE, \cT, \mu)$ is the set $E$ underlying the Euclidean lattice $\Eb$ equipped withe the counting measure --- and:
\begin{equation}\label{MSHE4}
H := \pi \Vert.\Vert^2.
\end{equation}

The associated partition function is nothing but the theta function of $\Eb$:
\begin{equation}\label{Zth}
Z(\beta) = \sum_{v \in E} e^{-\pi \beta \Vert v \Vert^2} = \theta_\Eb (\beta) \quad \mbox{ for every $\beta \in \Rpa.$}
\end{equation}
Besides, for any $x \in \Rpa,$
$$A_n(\pi x) = \mu^{\otimes n} (\{v \in \cE^n \mid H_n(v) \leq n \pi x \}) = \vert \{(v_1, \ldots, v_n) \in E^{\oplus n} \mid \Vert v_1 \Vert^2 + \ldots + \Vert v_n \Vert^2 \leq n \pi x \} \vert.$$
In other words:
\begin{equation}\label{AnAr}
\log  A_n(\pi x) = \hon (\Eb^{\oplus n}, n \pi x).
\end{equation}

Using the relations (\ref{Zth}) and (\ref{AnAr}), the content of Theorem \ref{ThMain} applied to the measure space with Hamiltonian defined by (\ref{MSHE1})-(\ref{MSHE4}) translates into Theorem \ref{Thhonthot}.

 Indeed, we immediately obtain:
\begin{equation*}
\Psi(\beta) = \log \theta_\Eb (\beta) \quad \mbox{ for every $\beta \in \Rpa$},
\end{equation*}
and
\begin{equation}\label{Shont}
S(\pi x) = \hont(\Eb, x) \quad \mbox{ for every $x \in \Rpa$}.
\end{equation}

\subsubsection{} From Theorem \ref{ThMain}, we also derive some additional properties of the functions $\hont(\Eb, .)$ and $\theta_\Eb$ that may have some interest.  

Firstly:
\begin{equation*}
\lim_{ x \lra 0_+} \hont(\Eb, x) = 0.
\end{equation*}

Moreover, as functions of $x$ and $\beta,$ 
$$\tilde{h}_{\rm Ar}^{0\, \prime}(\Eb, x) := \frac{d \hont(\Eb, x)}{dx}$$
and 
$\theta_\Eb'(\beta)/\theta_\Eb(\beta)$ define real analytic decreasing diffeomorphisms of $\Rpa$ to itself, and for any $(x, \beta) \in \R_+^{\ast 2},$
\begin{equation}\label{xbetaassoc}
\pi \beta = \tilde{h}_{\rm Ar}^{0\, \prime} (\Eb, x) \Longleftrightarrow \pi x = - \theta_\Eb'(\beta)/\theta_\Eb(\beta).
\end{equation}

Finally, for any $(x, \beta) \in \R_+^{\ast 2},$ we have:
\begin{equation}\label{hthLeg}
\hont(\Eb, x) \leq \log \theta_\Eb(\beta) + \pi \beta x,
\end{equation}
and equality holds in (\ref{hthLeg}) if and only $x$ and $\beta$ are related by the equivalent conditions (\ref{xbetaassoc}). 

\subsubsection{} The measure space with Hamiltonian associated to $\Eb$ by the relations (\ref{MSHE1})-(\ref{MSHE4}) may be seen as a quantum version of the one associated to a flat torus in paragraph \ref{Flattori}. 

Indeed, with the notation of this paragraph, the Hamiltonian operator that describes a non-relativistic particle of mass $m$ that freely moves on the flat torus $X := F^\vee_\R/F^\vee$
is $$\mathcal{H} := -(h/2 \pi)^2 \Delta/(2m)$$ acting on $L^2(X)$, where $h$ denotes Planck's constant and $\Delta$ the usual Laplacian. This operator has a discrete spectrum, which may be parametrized by the lattice $F$: the eigenfunction associated to $v \in F$ is the function $e_v$ on $X$ defined by $$e_v([x]) := e^{2 \pi i \langle x, v \rangle} \quad \mbox{for every $x \in F^\vee_\R$};$$
%$$e_v := \left([x] \lmt e^{2 \pi i \langle x, \xi \rangle}\right),$$ 
it satisfies:
$$\mathcal{H} e_v := -\left(\frac{h}{2\pi}\right)^2 \frac{1}{2m} \Delta\,e_v = \frac{h^2}{2 m} \Vert v \Vert^2_{\Fb}.$$

The associated partition function is:
$$Z(\beta) = \sum_{v \in F} e^{- \pi \beta h^2 \Vert v \Vert_{\Fb}/2m} = \theta_{\Fb}( \beta h^2/(2 \pi m)).$$
When $h = 1$ and $m = (2 \pi)^{-1}$, this partition function $Z$ coincides with $\theta_\beta$.

%\cite{Kullback97}

%\cite{MazoOdlyzko90}
%\cite{Schroedinger52}

%\cite{Cramer38}
%\cite{Stroock2011}
%\cite{CerfPetit2011}
%\cite{Georgii03}

%
%
\section{Proof of the Main Theorem}\label{PfMain}

In this section, we give a self-contained proof Theorem \ref{ThMain},  by ``unfolding" the arguments in  \cite{BostTheta}, Appendix A, and in the classical proofs of Cram\'er's theorem. This proof relies on some basic principles of measure and probability theory only. 

In this proof, the assertions in Theorem \ref{ThMain} will \emph{not} be established in the order they have been successively stated. 

Actually, we shall first study the function $\Psi$ and establish Part 2) of Theorem \ref{ThMain}, and then define the function $S$ as the Legendre-Fenchel transform of $-\Psi(-.)$ and establish its Part 3). This first part of the argument, in Subsection \ref{PUS}  will appear rather standard to any mathematician familiar with basic measure theory --- except possibly for the introduction of the probability measures $\nu_\beta$ in paragraph (\ref{nubeta}), which however should be unsurprising to anybody familiar with the first principle of statistical thermodynamics. 

Then, in Subsection \ref{ConvAn}, we shall establish the expression (\ref{Slim}) of $S(E)$ as the limit
\begin{equation}\label{limAnbis} \lim_{n \ra + \infty}(1/n) \log A_n(E),
\end{equation}
 and thus establish the main assertions of Part 1) of Theorem \ref{ThMain}.  This step  constitutes the   key point of the proof Theorem \ref{ThMain}, and  will follow from  applications of Markov's and  Chebyshev's inequalities on measure space $(\cE^n, \cT^{\otimes n})$ defined as the product of $n$-copies of $(\cE, \cT)$ equipped with 
the product mesures $\mu^{\otimes n}$ and $\nu_\beta^{\otimes n}$. These arguments also lead to the variants of the limit formula (\ref{Slim}) stated in Proposition \ref{Svar}.

At this stage, the proof of Theorem \ref{ThMain} will be completed, with the exception of the formula (\ref{limSmin}) for the entropy ``at the zero temperature limit". This formula,  of  a more technical character, will be established in Subsection \ref{zertemp}, which could be skipped at first reading.

Complements and variants to Theorem \ref{ThMain} and its derivation will be presented in the next sections. %Sections \ref{Compl} and~\ref{PDF}. 
Notably, in Subsection \ref{LanfEst}, we shall present a beautifully simple argument, due to Lanford \cite{Lanford73}, for the existence of the limit (\ref{limAnbis}), and in Section \ref{PDF}, we shall discuss  some alternative derivations of the key limit formula (\ref{Slim}) which asserts the equality of this limit with $S(E)$, where $S(E)$ is defined as the Legendre-Fenchel transform (\ref{LegendreS})  of $-\Psi(-.)$. 

Instead of arguments from measure and probability theory,  these derivations will rely on the theory of analytic functions and  on the use of the saddle-point method. They originate in  the work of Poincar\'e (\cite{Poincare12}) and of Darwin and Fowler (\cite{DarwinFowler1922},
\cite{DarwinFowler1922II},
\cite{DarwinFowler1923}).  
 
\bigskip

In this section, we consider a measure space $(\cE, \cT, \mu)$ and some $\cT$-measurable function $H$ from $\cE$ to $\Rp$ as in Section \ref{Thermod}, and we assume that Conditions ${\mathbf T_1}$ and ${\mathbf T_2}$ are satisfied.

\subsection{The functions $\Psi$, $U$ and $S$}\label{PUS}

\subsubsection{Analyticity properties of $Z$ and $\Psi$.}

For any $a$ in $\Rpa$ and any $\beta$ in the half plane  $[a, +\infty) +  i \R$ in $\C,$ we have:
$$\vert e^{- \beta H} \vert = e^{- \Re \beta. H} \leq e^{- a H}.$$
As the function $e^{-a H}$ is $\mu$-integrable on $\cE$, we immediately derive from this estimate:

\begin{proposition}\label{Zan} For any $\beta$ in the open half plane
$$\C_+ := \Rpa + i \R,$$
the integral (\ref{defZ}) which defines $Z(\beta)$ is absolutely convergent. The so-defined function 
$$Z: \C_+ \lra \C$$
is holomorphic and bounded on every half plane $[a, +\infty) + i \R$, where $a>0$. 

Moreover, for any $k \in \N$, the $k$-th derivative of $Z$ is given by the absolutely convergent integral, for every $\beta \in \C_+$:
\begin{equation}\label{Zk}
Z^{(k)}(\beta)= \int_\cE (-H)^k e^{-\beta H} \, d\mu.
\end{equation} \qed
\end{proposition}

This obviously yields the real analyticity of $Z: \Rpa \lra \R.$ Moreover, $Z(\beta)$ is clearly positive for any $\beta \in \Rpa$ (since $\mu(\cE) >0$). Consequently $\Psi = \log Z$ is a well-defined real analytic function on $\Rpa.$

Observe also that, for any $z \in \C_+,$ 
\begin{equation}\label{modZ}
\vert Z(z) \vert = \left\vert \int_\cE e^{-zH} \, d\mu \right\vert \leq \int_\cE \vert e^{-zH} \vert \, d\mu = Z(\Re z).
\end{equation}
Consequently, for every $\beta \in \Rpa,$
$$\Psi(\beta) = \max_{ z \in \beta + i\R} \log \vert Z(z) \vert.$$
According to Hadamard's three-lines theorem (see for instance \cite{Simon2011}, Theorem 12.3),   this representation of $\Psi$ implies its convexity on $\Rpa$. It may also be derived from  arguments of real analysis that we now present. 

\subsubsection{The measures $\nu_\beta$ and the convexity properties of $\Psi$}\label{nubeta}

For every $\beta \in \Rpa,$ we may introduce the probability measure
$$\nu_\beta := Z(\beta)^{-1} e^{-\beta H} \mu$$
on $(\cE, \cT)$.

Clearly, for every $\epsilon \in (0,\beta),$ the function $e^{\epsilon H}$ is $\nu_\beta$-integrable, and \emph{a fortiori} $H$ belongs to $L^p(\cE, \nu_\beta)$ for every $p \in [1, +\infty).$ (However $H$ is not essentially bounded with respect to $\mu$ -- or equivalently to $\nu_\beta$ --- as a consequence of  ${\mathbf T_1}$ and ${\mathbf T_2}$.)

Actually, for any $k\in \N$ and every $\beta \in \Rpa,$ we have:
\begin{equation}\label{Zkbeta}
Z^{(k)} (\beta) = \int_\cE (-H)^k e^{-\beta H} \, d\mu = (-1)^k Z(\beta) \int_\cE H^k\, d\nu_\beta.
\end{equation}

Let us introduce the mean value $m_\beta$ and the variance $\sigma_\beta$ of $H$ with respect to the probability $\nu_\beta$, defined by the relations
$$m_\beta := \int_\cE H \, d\nu_\beta$$
and
$$\sigma_\beta^2 := \int_\cE \vert H - m_\beta \vert^2 \, d\nu_\beta = \int_\cE H^2 \, d\nu_\beta - m_\beta^2.$$

As a straightforward consequence of (\ref{Zkbeta}), we obtain the following formulae: 

\begin{proposition}
For every $\beta \in \Rpa,$ we have:
\begin{equation}\label{mbeta}
m_\beta = Z(\beta)^{-1} \int_\cE H e^{-\beta H} \, d\mu = - Z(\beta)^{-1} Z'(\beta) = -\Psi'(\beta),
\end{equation}
\begin{equation}
\int_\cE H^2 \, d\nu_\beta =  Z(\beta)^{-1} \int_\cE H^2 e^{-\beta H} \, d\mu =  Z(\beta)^{-1} Z''(\beta),
\end{equation}
and
\begin{equation}\label{sbeta}
\sigma_\beta^2 =  Z(\beta)^{-1} Z''(\beta) - (Z(\beta)^{-1} Z'(\beta))^2 = \frac{d \;}{d \beta}\frac{Z'(\beta)}{Z(\beta)} = \Psi''(\beta).
\end{equation}
\qed
\end{proposition}

\begin{corollary}\label{derpsi} For every $\beta$ in $\Rpa,$ $\Psi'(\beta) < 0$ and $\Psi''(\beta) > 0.$
%$$\Psi'(\beta) < 0 \quad \mbox{ and  } \quad  \Psi''(\beta) > 0.$$
\end{corollary}

\begin{proof}
The expression (\ref{mbeta}) (resp. (\ref{sbeta})) of $-\Psi'(\beta)$ (resp., of $\Psi''(\beta)$) as $m_\beta$ (resp., as $\sigma_\beta^2$) shows that it is positive, since $H$ is non-negative and not (almost everywhere) constant.
\end{proof}

\subsubsection{The function $U$}

Let us now consider the real analytic function
$$U:= - \Psi': \Rpa \lra \R.$$
According to (\ref{mbeta}), for every $\beta \in \Rpa,$ we have: $$U(\beta) = m_\beta.$$ As $H \geq \Hm$ $\nu_\beta$-almost everywhere on $\cT,$ it satisfies:
\begin{equation}\label{UHm}
U(\beta) \geq \Hm.
\end{equation}
Besides, according to Corollary \ref{derpsi}, 
\begin{equation}\label{Uprim}
U'(\beta) = -\Psi''(\beta)<0.
\end{equation}

\begin{proposition}\label{limZPU}The limit behaviour of $Z(\beta)$, $\Psi(\beta)$ and $U(\beta)$ when $\beta \in \Rpa$ goes to $0$ and $+\infty$ is given by the following relations:

\begin{equation}\label{Zlim}
\lim_{\beta \ra 0_+} Z(\beta) = + \infty \quad \mbox{and} \quad \lim_{\beta \ra +\infty} Z(\beta) = \mu(H^{-1}(0)) \, (\in \R_+),
\end{equation}
\begin{equation}\label{Psilim}
\lim_{\beta \ra 0_+} \Psi(\beta) = + \infty \quad \mbox{and} \quad \lim_{\beta \ra +\infty} \Psi(\beta) = \log \mu(H^{-1}(0)) \, (\in [-\infty, +\infty)),
\end{equation}
and
\begin{equation}\label{Ulim}
\lim_{\beta \ra 0_+} U(\beta) = + \infty \quad \mbox{and} \quad \lim_{\beta \ra +\infty} U(\beta) = \Hm.
\end{equation}
\end{proposition}

\begin{proof} By monotone convergence (resp., by dominated convergence), as $\beta$ goes to $0$ (resp., to $+\infty$), $Z(\beta)$ goes to 
$$\int_\cE d\mu = \mu(\cE) = +\infty \quad (\mbox{resp., to } \int_{\cE} {\bf 1}_{H^{-1}(0)} d\mu = \mu(H^{-1}(0))).$$ 
This establishes (\ref{Zlim}), and (\ref{Psilim}) immediately follows.

According to (\ref{UHm}) and (\ref{Uprim}), the limit
$l_0:= \lim_{\beta \ra 0_+} U(\beta)$ and  $l_{\infty}:= \lim_{\beta \ra +\infty} U(\beta)$ exist in $(\Hm, +\infty]$ and $[\Hm, +\infty)$ respectively.
%$$l_0:= \lim_{\beta \ra 0_+} U(\beta)  \quad \mbox{(resp. $l_{\infty}:= \lim_{\beta \ra +\infty} U(\beta)$)}$$

If $l_0$ were not $+\infty$, then $U(\beta)$ would stay bounded when $\beta$ goes to $0$, and therefore its primitive $-\Psi(\beta)$ also. This would contradict the first part of (\ref{Psilim}). 

For every $\epsilon \in \Rpa,$ we have:
$$\Psi'(\beta) = -U(\beta) \leq -l_\infty + \epsilon$$
for $\beta$ large enough in $\Rpa.$ Therefore there exists $B_\epsilon$ and $c_\epsilon$ in $\Rpa$ such that
$$\Psi(\beta) \leq -(l_\infty - \epsilon) \beta + c_\epsilon \quad \mbox{for every $\beta \geq B_\epsilon$}.$$
In other words,
$$\int_{\cE} e^{- \beta(H- l_\infty + \epsilon)}\, d\mu \leq e^{c_\epsilon} \quad \mbox{for  $\beta \geq B_\epsilon$}.$$

This immediately implies that
$$\mu(\{ x \in \cE \mid H(x) \leq l_\infty - \epsilon \}) = 0,$$
or equivalently:
$$\Hm \geq l_\infty - \epsilon.$$

As $\epsilon$ is arbitrary in $\Rpa,$ this shows that $l_\infty \leq \Hm,$ and finally that $l_\infty = \Hm.$
\end{proof}

From (\ref{Uprim}) and (\ref{Ulim}), we deduce:
\begin{corollary}\label{Udiffeo} The function $U$ defines a real analytic diffeomorphism
$$ U := \Rpa \lrasim (\Hm, +\infty).$$% \hspace{2cm} \mbox{\qed}$$
\qed
\end{corollary}

\subsubsection{The entropy function $S$}\label{entropy}

For any $E \in (\Hm, +\infty),$ the expression $\Psi(\beta) + \beta E$ defines a strictly convex function of $\beta \in \Rpa$: its derivative $-U(\beta) + E$ is  increasing on $\Rpa$ and vanishes if (and only if) $\beta = U^{-1}(E).$ Therefore $\Psi(\beta) + \beta E$ attains its infimum over $\Rpa$ precisely at $\beta = U^{-1}(E).$

We shall define the entropy function
$$S: (\Hm, +\infty) \lra \R$$
by
$$S(E) := \inf_{\beta \in \Rpa} (\Psi(\beta) + \beta E) = \Psi(U^{-1}(E)) + U^{-1}(E) E.$$
In other words, the function $-S(-.)$ is the Legendre-Fenchel transform of the function $\Psi$, or the functions $-S(-.)$ and $\Psi$ are dual in the sense of Young (see  \ref{remindLegendre} \emph{infra}).

The function $S$, like $\Psi$ and $U^{-1}$, is clearly real analytic. Moreover the elementary properties of the Legendre-Fenchel-Young duality   applied to $\Psi$ and $-S(-.)$ show that $S'' >0$ on $\Rpa$, that $S'$ defines a real analytic diffeomorphism
$$S': (\Hm, + \infty) \lrasim \Rpa$$
inverse of $U$, and that, for any $\beta \in \Rpa,$
\begin{equation}\label{psiinf}
\Psi(\beta) = \sup_{E \in (\Hm, +\infty)} (S(E)- \beta E); 
\end{equation} 
moreover, the infimum in the right hand side of (\ref{psiinf}) is attained at a unique point $E$ in $(\Hm,\infty),$ namely $E = U(\beta).$

In other words, for any $E \in (\Hm, +\infty)$ and any $\beta \in \Rpa,$ the following inequality holds:
\begin{equation}\label{Spsiest}
S(E) \leq \Psi(\beta) + \beta E,
\end{equation}
%moreover the inequality % for any given $E \in (\Hm, +\infty)$ (resp., any $\beta \in \Rpa$), 
%(\ref{Spsiest})
and it  becomes an equality if and only if $E = U(\beta),$ or equivalently $\beta = S'(E)$.

Finally observe that, from
 the trivial lower bound 
$$S(E) \geq \Psi(U^{-1}(E))$$
and the relations
$$\lim_{E \ra + \infty} U^{-1}(E) = 0 \quad \mbox{and} \quad \lim_{\beta \ra 0_+} \Psi(\beta) = +\infty$$
(see Proposition \ref{limZPU} and Corollary \ref{Udiffeo}), immediately follows the relation (\ref{limSinfty}):
$$\lim_{E \ra + \infty} S(E) = +\infty.$$

{\small
\subsubsection{A reminder on Legendre duality}\label{remindLegendre}

For the convenience of the reader, in this paragraph we establish the basic facts concerning the Legendre--Fenchel--Young duality of convex smooth functions of one variable used in \ref{entropy}. 
They are well known (see for instance \cite{Arnold89}, \S 14), but usually not formulated in the precise form used here\footnote{Legendre duality actually holds in a much more general setting, and we refer the reader to \cite{Hormander94}, Section 2.2 for a more general discussion of Legendre duality, concerning convex functions on finite dimensional vector spaces, with no smoothness assumptions, and  to \cite{Simon2011}, Chapter 5, for its extension to convex functions on locally convex topological vector space.}.

Let $I$ be a non empty interval in $\R$ and let $f: I \lra \R$ be a function of class $C^2$ which is strictly convex, namely which satisfies
$$f''(x) >0 \quad \mbox{for every $x \in I$}.$$
The inverse function theorem applied to $f'$ shows that $J:= f(I)$ is a non empty open interval in $\R$ and that  
$f'$ defines a $C^1$ diffeomorphism
$$f': I \lrasim J.$$
Moreover, for any $p \in J,$ the function
$$F(x,p) := px - f(x)$$
of $x \in I$ is concave and attains its  supremum at a unique point of $I$, namely $f'^{-1}(p).$ (Indeed $\partial F(x,)/\partial x = p-f'(x)$ and $\partial^2 F(x,p)/\partial  x^2 = -f"(x) <0.$)

The \emph{Legendre-Fenchel transform} or \emph{Young dual} of $f$ is the function 
$$g: J \lra \R$$
defined by
\begin{equation}\label{defg}
g(p) := \max_{x \in I} F(x, p) = F(f'^{-1}(p), p).
\end{equation}

\begin{proposition}\label{Legendredual} The function $g$ is strictly convex of class $C^2$. Its derivative defines a $C^1$ diffeomorphism inverse of $f'^{_1}$:
$$g' = f'^{-1} : J \lrasim I.$$
Moreover, for any $x \in I,$
\begin{equation}\label{Ldual}
f(x) = \max_{p \in J} G(x,p) = G(x, g'^{-1}(x))
\end{equation}
where $G(x,p) := px -g(p).$
\end{proposition}

The equality (\ref{Ldual}) precisely asserts that $f$ coincides with the Legendre--Fenchel  transform of its Legendre--Fenchel transform $g$. 
 In other words, the Legendre-Fenchel transformation is involutive.
 
 The symmetry  between the two functions $f$ and $g$ may also be expressed by the fact that, for any $(x,p) \in I \times J,$
\begin{equation}\label{fgest}
px \leq f(x) + g(p)
\end{equation}
and that the inequality (\ref{fgest}) becomes an equality precisely when $p= f'(x)$, or equivalently when $x = g'(p).$  (The inequality (\ref{fgest}) is  sometimes called the  \emph{inequality of Young}; see for instance \cite{HardyLittlewoodPolya52}, \S 4.8.) 

\begin{proof}[Proof of Proposition \ref{Legendredual}] By definition for any $x$ in $I,$
$$g(f'(x)) = f'(x)x -f(x).$$
As $f'$ is a $C^1$ diffeomorphism from $I$ onto $J,$ this shows that $g$ is of class $C^1$ on $J$ and that, for every $x \in I,$
$$\frac{d}{dx} g(f'(x)) = f'(x) x -f(x).$$
In other words,
$$g'(f'(x)) = f''(x).x$$
and therefore, as $f" >0,$
$$g'(f'(x))=x.$$

This shows that
$$g' \circ f' = {\rm Id}_I.$$
Consequently $g'(J) =I$ and $g'$ establishes a $C^1$ diffeomorphism from $J$ to $I$, inverse to $f'$. In particular, like $f'$, the function $g'$ is $C^1$ with a positive derivative, and $g$ is therefore of class $C^2$ and strictly convex.

For any $(x,y) \in I \times J,$ we have:
$$g(f'(y)) = F(y, f'(y))) = f'(y) y -f(y)$$
and 
$$G(x, f'(y)) = f'(y) x - g(f'(y)) = f(y) + f'(y) (x-y).$$ 
This is the ordinate at the point of abscissa $x$ on the line tangent to the graph of $f$ at the point $(y, f(y))$. As $f$ is strictly convex, this tangent line lies below  this graph, and we have
$$G(x,f'(y)) \leq f(x)$$
with equality if and only if $x=y$.

By letting $p := f'(y)$, this shows that, for any $(x,p) \in I \times J,$
$$G(x,p) \leq f(x),$$
with equality if and only if $p =f'(x).$ 

This establishes (\ref{fgest}) and completes the proof.
\end{proof}

Observe finally that, if $f$ and $g$ are two strictly convex $C^2$ functions that are Young dual as in Proposition \ref{Legendredual}, then $f$ is of class $C^k$ for $k>2$ (resp. of class $C^\infty$, resp. real analytic) if and only $g$ is. This directly follows from the expressions (\ref{defg}) and (\ref{Ldual}) for $g$ and $f$ in terms of each other. 

}

\subsection{The convergence of $(1/n) \log A_n(E)$}\label{ConvAn}
\subsubsection{The Markov and Chebyshev inequalities and the weak law of   large numbers}

 Let us start with a reminder of some basic results in probability theory, that we will formulate in a mesure theoretic language adapted to the derivation of Theorem \ref{ThMain}.
 
Let us  consider a $\sigma$-finite measure on $\cE$, $$\nu : \cT \lra [0,+\infty]$$ and some $\cT$-measurable function $$f: \cE \lra \R.$$ 

Markov's inequality is the observation that, when $f$ is non-negative, then, for any $\epsilon \in \Rpa,$
\begin{equation}\label{MI}
\epsilon\,  \mu(f^{-1}([\epsilon, +\infty))) = \int_\cE \epsilon\,{\bf 1}_{f^{-1}([\epsilon, +\infty))} \, d\nu \leq \int_\cE f \, d\nu.
\end{equation}

Let us now know assume that $\nu$ is a probability measure and that $f$ is square integrable, and therefore integrable, with respect to $\nu$, and let us introduce its ``mean value"
$$m:= \int_E f \, d \nu$$ % \quad \mbox{and} \quad  
and its ``variance" 
$$\sigma := \Vert f - m \Vert^2_{L^2(\cE, \nu)}.$$  
In other words,
\begin{equation}\label{sigmadef}
\sigma^2 = \int_\cE \vert f - m \vert^2 \, d\nu = \int_\cE f^2 \, d\nu - m^2. 
\end{equation}

The Chebyshev inequality is derived by applying Markov inequality (\ref{MI}) to the function $\vert f - m \vert^2$. It asserts that, for every $\epsilon \in \Rpa,$
$$\epsilon^2 \, \nu(\{x \in \cE \mid \vert f(x) - m \vert \geq \epsilon \}) \leq \int_\cE \vert f - m \vert^2 \, d\nu = \sigma^2.$$

For any integer $n\geq 1,$ we may consider the $n$-functions 
$$\phi_i: \cE^n \lra \R, \quad 1 \leq i \leq n,$$
defined by
$$\phi_i(x_1, \ldots,x_n) := f(x_i) - m.$$

They are clearly square integrable on $(\cE^n, \cT^{\otimes n}, \nu^{\otimes n})$ and satisfy:
$$\Vert \phi_i\Vert^2_{L^2} := \int_{\cE^n} \vert \phi_i \vert^2 \, d\nu^{\otimes n} = \int_\cE \vert f - m \vert^2 \, d\nu = \sigma^2.$$

Moreover, since the function $f-m$ satisfies
$$\int_\cE (f-m) \, d\nu =0,$$
or in other words, is orthogonal to the function ${\bf 1}_\cE$ in $L^2(\cE, \nu)$, the functions $\phi_1, \ldots, \phi_n$ are pairwise orthogonal in $L^2(\cE^n, \nu^{\otimes n})$. This implies that
 $$
 \Vert (\phi_1 + \ldots \phi_n)/n \Vert^2_{L^2}
 =
 (\Vert \phi_1\Vert^2_{L^2}
 + \ldots +
 \Vert \phi_n\Vert^2_{L^2})/n^2 = \sigma^2/n.$$
 
 This observation implies that, if we define the $L^2$-function $\tilde{f}_n$ on $(\cE^n, \cT^{\otimes n}, \nu^{\otimes n})$ by the equality
 $$\tilde{f}_n(x_1, \ldots,x_n) := (f(x_1) + \ldots + f(x_n))/n,$$
 the mean value of which is clearly
 $$\tilde{m}_n := \int_{\cE^n} \tilde{f}_n \, d\nu^{\otimes n} = m,$$
 the variance of $\tilde{f}_n$ is given by:
 $$\tilde{\sigma}^2_n := \int_{\cE^n} \vert \tilde{f}_n - \tilde{m}_n\vert^n \, d\nu^{\otimes n} 
 =   \Vert (\phi_1 + \ldots \phi_n)/n \Vert^2_{L^2} = \sigma^2/n.$$
 
 Therefore 
Chebyshev inequality applied to  the function $\tilde{f}_n$ 
establishes the following proposition, which constitutes a form of the weak law of large numbers: 

%on $(\cE^n, \cT^{\otimes n}, \nu^{\otimes n})$ defined by 
% $$\tilde{f}_n(X_1, \ldots,x_n) := (f(x_1) + \ldots + f(x_n))/n$$
% --- which clearly satisfies
% $$\tilde{m}_n := \int_{\cE^n} \tilde{f}_n \, d\nu^{\otimes n} = m$$
% and
% $$\tilde{\sigma}^2_n := \int_{\cE^n} \vert \tilde{f}_n - \tilde{m}_n\vert^n \, d\nu^{\otimes n} 
% =   \Vert (\phi_1 + \ldots \phi_n)/n \Vert^2_{L^2(\cE^n, \nu^{\otimes n})} = \sigma^2/n$$
% --- establishes the following 

\begin{proposition}\label{Csum} With the above notation, for any integer $n \geq 1$ and any $\epsilon \in \Rpa,$
we have:
\begin{equation}\label{eq:Csum}
\epsilon^2 \, \nu^{\otimes n}(\{(x_1,\ldots,x_n) \in \cE^n \mid \vert (f(x_1) + \ldots f(x_n))/n -  m\vert \geq \epsilon \}) \leq  \sigma^2/n.
\end{equation}
\end{proposition}
\qed 

\subsubsection{Bounding $A_n(E)$ from above}\label{BoundAbove}

The obvious relation 
$$e^{-\beta H_n(x_1, \ldots, x_n)} = e^{-\beta H(x_1)} \ldots e^{-\beta H(x_n)}$$
and the very definition of the measure $\mu^{\otimes n}$ show that, for every integer $n\geq 1,$
\begin{equation}\label{Zn}
Z(\beta)^n = \left(\int_\cE e^{-\beta H} \, d\mu\right)^n = \int_{\cE^n} e^{-\beta H_n} \, d\mu^{\otimes n}.
\end{equation}

\begin{proposition}\label{AZn} For any integer $n \geq 1$, and any $E$ and $\beta$ in $\Rpa$, we have:
\begin{equation}\label{Anup}
A_n(E) \leq e^{n \beta E} \, Z(\beta)^n. 
\end{equation}
\end{proposition}

\begin{proof}
This is Markov inequality (\ref{MI}) applied to the function $f := e^{-\beta H_n}$ on the measure space $(\cE^n, \cT^{\otimes n}, \mu^{\otimes n})$ and to $\epsilon := e^{-n\beta E}.$ 
\end{proof}

By taking the logarithm of (\ref{Anup}), we obtain:
\begin{equation}\label{logAnup}
\frac{1}{n} \log A_n(E) \leq \beta E + \Psi(\beta).
\end{equation}
(We define $\log 0$ to be $-\infty$.)

When $E > \Hm,$ the infimum over $\beta \in \Rpa$ of the right-and side of (\ref{logAnup}) is, by definition, $S(E),$ and the inequality (\ref{logAnup}) may be rephrased as follows:

\begin{proposition}\label{Anabove} For any $E$ in $(\Hm, +\infty)$ and any integer $n \geq 1,$
\begin{equation}\label{logAnupS}
\frac{1}{n} \log A_n(E) \leq S(E).
\end{equation}
\qed
\end{proposition}

\subsubsection{Bounding $A_n(E)$ from below}\label{BoundBelow}

%\subsubsection{Applying the weak law of large numbers to $\nu_\beta$}

For every integer $n \geq 1$, and every $(E, \epsilon)$ in $\R\times \Rpa,$ we may consider the 
$\cT^{\otimes n}$-measurable subset
$$\cS_n(E, \epsilon) := \{ x \in \cE^n \mid \vert H_n(x) - n E \vert < n \epsilon \}$$
of $\cE^n$, which describes the $n$-particle states of the system under study which average energy in the interval $(E-\epsilon, E + \epsilon).$ 

We may also introduce its measure
$$\Sigma_{n}(E,\epsilon):= \mu^{\otimes n}(\cS_n(E, \epsilon)).$$
The obvious inclusion $$\cS_n(E, \epsilon) \subset H_n^{-1}((-\infty, E + \epsilon])$$
yields the estimate:
\begin{equation}\label{SigmaAtriv}
\Sigma_{n}(E,\epsilon) \leq A_n(E+ \epsilon).
\end{equation}

We shall actually derive a lower bound on $\Sigma_{n}(E,\epsilon)$, which will immediately yield some lower bounds on $A_n(E)$. It will be a consequence of the ``weak law of large number" (\ref{eq:Csum}) applied to the function $f:= H$ and to the probability measure $\nu_\beta.$

Indeed, using the expressions (\ref{mbeta}) and (\ref{sbeta}) for the mean value $m$ and the variance $\sigma$ in this special case, Proposition   \ref{Csum} then takes the form:

\begin{lemma}\label{nubWLN} For every integer $n\geq 1$, and every $\beta$ and $\epsilon$ in $\Rpa,$ we have:
\begin{equation}\label{nusigmaup}
\nu_\beta^{\otimes n}(\cE^n \setminus \cS_n( U(\beta), \epsilon))
\leq \epsilon^{-2}\, \Psi''(\beta)/n.
\end{equation} \qed
\end{lemma}

From the upper bound (\ref{nusigmaup}) on the measure of $\cE^n \setminus \cS_n(E,\epsilon)$ with respect to $\nu_\beta^{\otimes n}$, we may derive some lower bound on the measure $\Sigma_{n}(E,\epsilon)$ of $\cS_n(E,\epsilon)$  with respect to $\mu^{\otimes n}$:

\begin{proposition}\label{prop:keyest} For every integer $n\geq 1$, and every $\beta$ and $\epsilon$ in $\Rpa,$ we have:
\begin{equation}\label{keyest}
\Sigma_{n}(U(\beta),\epsilon) \geq e^{n(S(U(\beta)) - \epsilon \beta)} \, (1 - \epsilon^{-2} \,\Psi''(\beta)/n).
\end{equation}
\end{proposition}

\begin{proof}%[Proof of Proposition \ref{prop:keyest}]
From the very definition of $\nu_\beta$, we get:
\begin{equation}\label{un}
\Sigma_{n}(U(\beta),\epsilon):= \mu^{\otimes n}(\cS_n(U(\beta), \epsilon))
= Z(\beta)^n \, \int_{\cS_n(U(\beta), \epsilon)} e^{\beta H_n} \, d\nu_\beta^{\otimes n}.
\end{equation}

Moreover the lower bound
$H_n > n (U(\beta) - \epsilon)$
holds over $\cS_{n}(E,\epsilon)$. 
Consequently:
\begin{equation}\label{deux}
\int_{\cS(U(\beta), \epsilon)} e^{\beta H_n} \, d\nu_\beta^{\otimes n} 
\geq e^{n\beta(U(\beta) - \epsilon)} \, \nu_{\beta}^{\otimes n}(\cS_{n}(U(\beta),\epsilon)).
\end{equation}

Besides, we have:
\begin{equation}\label{trois}
Z(\beta)^n = e^{n \Psi(\beta)},
\end{equation}
and, according to Lemma \ref{nubWLN}, 
\begin{equation}\label{quatre}
\nu_{\beta}^{\otimes n}(\cS_{n}(U(\beta),\epsilon)) \geq 1 -\epsilon^{-2}\, \Psi''(\beta)/n.
\end{equation}

The estimate (\ref{keyest}) follows from (\ref{un})-(\ref{quatre}) and from the relation
$S(U(\beta)) = \beta U(\beta) + \Psi(\beta).$
\end{proof}

The estimate (\ref{keyest}) is non-trivial  only when $n > \epsilon^{-2}\, \Psi''(\beta).$ When this holds, it may be written: 
\begin{equation}\label{keyestbis}
\frac{1}{n} \log \Sigma_{n}(U(\beta),\epsilon) \geq S(U(\beta)) - \epsilon \beta + (1/n) \log(1 - \epsilon^{-2}\, \Psi''(\beta)/n).
\end{equation}

This clearly implies:

\begin{corollary}\label{liminfsigma} For any $\beta$ and $\epsilon$ in $\Rpa,$
$$\liminf_{n \ra +\infty} \frac{1}{n} \log \Sigma_{n}(U(\beta),\epsilon) \geq S(\beta) - \epsilon \beta.$$
\qed
\end{corollary}

We may now complete the proof of Theorem \ref{ThMain}.

Together with the trivial estimate (\ref{SigmaAtriv}), Corollary \ref{liminfsigma} shows that, for any $E \in (\Hm, +\infty)$ and any $\epsilon \in \Rpa,$
%By taking the limit when $n$ goes to $+\infty$, we obtain that, for every $\beta$ and $\epsilon$ in $\Rpa,$
%$$ \liminf_{ n \ra + \infty} \frac{1}{n} \log A_n(U(\beta) + \epsilon) \geq S(U(\beta)) - \epsilon \beta,$$ 
%
% for every $E\in (\Hm, +\infty)$ and any $\epsilon \in \Rpa,$
$$ \liminf_{ n \ra + \infty} \frac{1}{n} \log A_n(E + \epsilon) \geq S(E) - \epsilon S'(E).$$ 
(We have performed the change of variable $E = U(\beta),$ or equivalently $\beta = S'(E).$)
Equivalently, %that
for every $E\in (\Hm, +\infty)$ and any $\epsilon \in (0, E-\Hm),$
$$ \liminf _{ n \ra + \infty} \frac{1}{n} \log A_n(E) \geq S(E -\epsilon) - \epsilon S'(E - \epsilon).$$

By taking the limit when $\epsilon$ goes to $0_+,$ we get:
\begin{equation}\label{lwAn}
\liminf _{ n \ra + \infty} \frac{1}{n} \log A_n(E) \geq S(E).
\end{equation}
Together with Proposition \ref{Anabove}, this proves that
$$\lim_{ n \ra + \infty} \frac{1}{n} \log A_n(E) = \sup_{n \geq 1} \frac{1}{n} \log A_n(E) = S(E).$$

{\small \subsubsection{Proof of Proposition \ref{Svar}}\label{PrSvar}

To establish the lower bound (\ref{lwAn}), we have used Proposition \ref{prop:keyest} through its Corollary \ref{liminfsigma}. By using the full strength of the estimate (\ref{keyest}) established in this Proposition, it is possible to derive stronger results. To illustrate this point, we now  explain how to use it to establish  Proposition \ref{Svar}.

With the notation of  Proposition \ref{Svar}, we may clearly assume that, for every $n\geq n_0,$ the interval $I_n$ is bounded, contained in $(\Hm, +\infty)$, and has a non-empty interior (or equivalently, $l_n >0$), and we may define
$$b_n:= \sup I_n$$
and
$$a_n := b_n -\min (l_n, n^{-1/3}).$$

Then, for every $n \geq n_0,$
$$(a_n, b_n) \subset I_n.$$
Moreover,
$$\lim_{n \ra + \infty} b_n = E,$$
and the positive real numbers
$$\epsilon_n := (b_n - a_n)/2$$
satisfy
$$\lim_{n\ra +\infty} \epsilon_n = 0$$
and 
\begin{equation}\label{nepsPsi}
\liminf_{n \ra +\infty} n\epsilon_n^2 > \Psi''(\beta).
\end{equation}

For every $n \geq n_0$, let us introduce:
$$\tilde{A}_n := \mu^{\otimes n} (\{ x \in \cE^n \mid H_n(x) \in n I_n\}).$$
To establish Proposition \ref{Svar}, we have to prove that
\begin{equation}\label{limAnt}
\lim_{n \ra +\infty} \frac{1}{n} \log \tilde{A}_n = S(E).
\end{equation}

The inclusion 
$$\{ x \in \cE^n \mid H_n(x) \in n I_n\} \subseteq \{ x \in \cE^n \mid H_n(x) \leq n b_n\}$$
yields the upper bound:
$$\tilde{A}_n \leq A_n(b_n).$$
Together with the upper bound (\ref{logAnupS}) on $A_n,$ this implies:
$$\frac{1}{n} \log \tilde{A}_n \leq \frac{1}{n} \log A_n(b_n) \leq S(b_n),$$
and therefore:
\begin{equation}\label{liminfAt}
\limsup_{n \ra + \infty} \frac{1}{n} \log \tilde{A}_n \leq \lim_{n \ra +\infty} S(b_n) = S(E).
\end{equation}

For every $n \geq n_0,$ we may also introduce 
$$\beta_n := S'((a_n + b_n)/2) = U^{-1}((a_n + b_n)/2).$$
Clearly, when $n$ goes to $+\infty,$ $(a_n + b_n)/2$ converges to $E$, and $\beta_n$ to $\beta = S'(E).$ Therefore the estimate (\ref{nepsPsi}) implies the existence of $\eta \in (0,1)$ such that, for any large enough integer $n$,
$$\epsilon_n^{-2}\,  \Psi''(\beta_n)/n \leq \eta.$$

The inclusion
$$\cS((a_n+ b_n)/2, \epsilon_n) =\{ x \in \cE^n \mid H_n(x) \in (a_n, b_n) \} \subseteq \{ x \in \cE^n \mid H_n(x) \in n I_n\}$$
yields the lower bound on $\tilde{A}_n$:
$$\Sigma_n(U(\beta_n), \epsilon_n): = \mu^{\otimes n} (\cS_n((a_n+ b_n)/2, \epsilon_n)) \leq \tilde{A}_n.$$
Besides, the lower bound on $\Sigma_n$ established in Proposition \ref{prop:keyest}, written in the form (\ref{keyestbis}), 
shows that, when $n$ is large enough:
$$\frac{1}{n} \log \Sigma_n(U(\beta_n), \epsilon_n) \geq S(U(\beta_n)) - \epsilon_n \beta_n + (1/n) \log (1-\eta).$$

The last two estimates immediately imply that
$$
\liminf_{n \ra + \infty} \frac{1}{n} \log \tilde{A}_n \geq \lim_{n \ra +\infty} [S(b_n) - \epsilon_n \beta_n + (1/n) \log (1-\eta)] = S(E).$$

Together with (\ref{liminfAt}), this establishes (\ref{limAnt}). \qed 
%{\small
\subsection{The zero temperature limit}\label{zertemp} At this stage, all assertions in Theorem \ref{ThMain} have been established, but for the expression (\ref{limSmin}) for the limit of $S(E)$ when $E$ decreases to $\Hm$. 

We will establish it in this subsection, which turns out to be of a more technical character and logically independent of the proof of convergence of  $(1/n) \log A_n(E)$ to $S(E)$ in the previous section, %It is of a more technical character, and logically not required for the proof of the other assertions in Theorem \ref{ThMain},
and could therefore be skipped at first reading. 

In \cite{BostTheta}, Appendix A, the expression (\ref{limSmin}) is obtained as a consequence of the convexity and semi-continuity of  $S(E)$ as a function of $E\in \R$ with values in $[-\infty, +\infty)$. Here we will derive it from a closer study of  the asymptotic behavior of the function $Z(\beta)$ --- defined as the   Laplace transform (\ref{defZ}) of the measure $\mu$ ---  and of the associated functions $\Psi(\beta)$, and $U(\beta)$,   when $\beta$ goes to $+\infty$. 

Physically, this corresponds to the limit where the temperature $\beta^{-1}$ goes to zero, and the results of this paragraph may be seen as a mathematical interpretation of the third law of thermodynamics, which governs the behavior of the entropy and the heat capacity in this limit (see for instance \cite{Huang87}, Sections 1.7 and 8.4).

\subsubsection{Asymptotics of $Z$ and its derivatives at zero temperature}

Our study will rely on the following asymptotic relations satisfied by the derivatives $Z^{(k)}(\beta)$ of the partition function when $\beta$ goes to $+\infty$.

\begin{proposition}\label{betainf}
When $\beta$ goes to $+\infty,$
\begin{equation}\label{betainf1}
(-1)^k Z^{(k)}(\beta) = H^k_{\rm min} \, e^{-\beta \Hm} \mu(\cE_m) + o(e^{-\beta \Hm})% \quad \mbox{for every $k \in \N$}
\end{equation}
for every $k \in \N$, and
\begin{equation}\label{betainf2}
\beta^k\left(\frac{d}{d\beta} + \Hm\right)^k Z(\beta) := \beta^k \, \sum_{i=0}^k \binom{k}{i} H_{\rm min}^i \, Z^{(k-i)}(\beta) = 
o(e^{-\beta \Hm}) %\quad \mbox{for integer $k\in \N\setminus\{0\}$.}
\end{equation}
for every $k\in \N\setminus\{0\}$.
\end{proposition}

\begin{proof} To establish (\ref{betainf1}), observe that, for any $\beta \in \Rpa$ and any $k \in \N,$
\begin{equation*}%\label{betainfpr1}
(-1)^k Z^{(k)}(\beta) \, e^{\beta \Hm} = \int_{\cE} H^k e^{-\beta (H-\Hm)} \, d\mu 
\end{equation*}
and that, according to the dominated convergence theorem,
$$\lim_{\beta \ra +\infty} \int_{\cE} H^k e^{-\beta (H-\Hm)} \, d\mu = \int_{\cE} H^k {\bf 1}_\cE \, d \mu = H^k_{\rm min}\, \mu(\cEm).
$$

To establish (\ref{betainf2}), we write, for any $\beta \in \Rpa$ and any integer $k \geq 1$:
\begin{equation}\label{betainfpr}
(-1)^k \beta^k\left(\frac{d}{d\beta} + \Hm\right)^k Z(\beta) = \int_{\cE} \beta^k (H-\Hm)^k e^{-\beta (H-\Hm)} \, d\mu.
\end{equation}
Observe that, as a function of $t\in \Rp$, $t^k e^{-t}$ increases on $[0,k]$ and decreases on $[k,+\infty);$ in particular, it is bounded from above by $k^k e^{-k}$. For any $E \in (\Hm, +\infty),$ we let
$$\cE_{\leq E} := H^{-1} ([O,E]) \quad \mbox{and} \quad \cE_{>E} := H^{-1}((E, +\infty)).$$

The non-negative function $\beta^k (H-\Hm)^k\, e^{-\beta(H-\Hm)}$ vanishes on $\cEm,$ is bounded from above by $k^k e^{-k}$ and, over $\cE_{>E}$, decreases as a function of $\beta$ when $\beta \geq (E-\Hm)^{-1}.$ This shows that:
$$\int_{\cE_{\leq E}} \beta^k (H-\Hm)^k\, e^{-\beta(H-\Hm)}\, d\mu \leq k^k e^{-k} \, \mu(\cE_{\leq E} \setminus \cEm) = k^k e^{-k} \, \mu(H^{-1}((0,E])$$
and, by dominated convergence again,
$$\lim_{\beta \ra + \infty} \int_{\cE_{>E}} \beta^k (H-\Hm)^k\, e^{-\beta(H-\Hm)}\, d\mu =0.$$

This shows that
$$\limsup_{\beta \ra + \infty} \int_{\cE} \beta^k (H-\Hm)^k\, e^{-\beta(H-\Hm)}\, d\mu
\leq  k^k e^{-k} \, \mu(H^{-1}((0,E])).$$

As $E$ is arbitrary in $(\Hm, +\infty)$ and $\lim_{E \ra H_{{\rm min}, +}} \mu(H^{-1}((0,E]) )=0,$
This shows that
$$\lim_{\beta \ra + \infty} \int_{\cE} \beta^k (H-\Hm)^k\, e^{-\beta(H-\Hm)}\, d\mu =0.$$

Together with (\ref{betainfpr}), this establishes (\ref{betainf2}).
\end{proof}

\subsubsection{The asymptotics of $S$, $U$, $U'$ at zero temperature and the third law of thermodynamics}

When $k=0,$ the equality (\ref{betainf1}) shows that, if  $\mu(\cEm) =0,$ then
$$Z(\beta) = o(e^{-\beta \Hm}) \quad \mbox{ when $\beta \ra + \infty$},$$
or equivalently,  by taking logarithms:
\begin{equation}\label{psimin}
\lim_{\beta \ra +\infty} (\Psi(\beta) + \beta \Hm) = -\infty.
\end{equation}

According to (\ref{Spsiest}), for every $\beta \in \Rpa,$
$$\limsup_{E \ra \Hm} S(E) \leq \Psi(\beta) + \beta \Hm.$$
Together with (\ref{psimin}), this shows that
\begin{equation}\label{entinf}
\lim_{H \ra \Hm_+}  S(E) = -\infty \quad \mbox{ when}  \quad \mu(\cEm) =0.
\end{equation}

Let us now assume that $\mu(\cEm) >0.$ Then (\ref{betainf1}) with $k =0$ shows that, when $\beta$ goes to $+\infty,$
$$Z(\beta) = \mu(\cEm) e^{-\beta \Hm} + o( e^{-\beta \Hm}) \sim  \mu(\cEm) e^{-\beta \Hm},$$ or equivalently,
\begin{equation}\label{psizerot}
\Psi(\beta) := \log Z(\beta) = - \Hm \beta + \log \mu(\cEm) + o(1).
\end{equation}

Accordingly, the relations (\ref{betainf1}) and (\ref{betainf2}) may be written:
\begin{equation}\label{betainf1bis}
(-1)^k Z^{(k)}(\beta) = H^k_{\rm min} \, e^{-\beta \Hm} \mu(\cE_m) + o(Z(\beta))
\end{equation}
and
\begin{equation}\label{betainf2bis}
\beta^k \, \sum_{i=0}^k \binom{k}{i} H_{\rm min}^i \, Z^{(k-i)}(\beta) = o(Z(\beta)).
\end{equation}

The asymptotic relations (\ref{betainf2bis}) may be reformulated in a more convenient form, namely: 

\begin{corollary}\label{thirdlaw} If $\mu(\cEm) >0,$ then
\begin{equation}\label{Ubetazero}
\lim_{\beta \ra +\infty} \beta (U(\beta) -\Hm) =0
\end{equation}
and, for any integer $k \geq 1,$
\begin{equation}\label{Uderzero}
\lim_{\beta \ra +\infty} \beta^{k+1} U^{(k)}(\beta) =0.
\end{equation}

\begin{proof} When $k=1$, the relation (\ref{betainf2bis}) reads
$$\beta (Z'(\beta) + \Hm \,Z(\beta)) = o(Z(\beta)),$$
and may be written as (\ref{Ubetazero}), since $U(\beta) = -Z'(\beta)/Z(\beta).$

From this last relation also follows, by a straightforward induction on the integer $n \geq 1,$ the existence of some polynomial $P_n$ in $\Z[X_0,\cdots, X_{n-2}]$ such that
$$\frac{Z^{(n)}(\beta)}{Z(\beta)} = - U^{(n-1)}(\beta) + P_k(U(\beta),\cdots, U^{(n-2)}(\beta)).$$
When $k =1,$ $P_1 = 0.$ Moreover $P_k(X_0, \cdots, X_{n-2})$ is homogeneous of weight $n$ when each indeterminate $X_i$ is given the weight $i+1.$ 

The relations (\ref{Uderzero}) now follow from (\ref{betainf2bis}) by induction on $k\geq 1$, by using the relation 
  $$U^{(n-1)}(\beta) = - \frac{Z^{(n)}(\beta)}{Z(\beta)} + P_k(U(\beta),\cdots, U^{(n-2)}(\beta))$$ with $n = k+1$. We leave the details to the reader.
\end{proof}
\end{corollary}

Corollary \ref{thirdlaw} may be understood as a mathematical expression of the third law of thermodynamics, which notably asserts the existence of a finite limit of entropy at zero temperature.

Indeed, combined with the expression
$$S(U(\beta)) = \Psi(\beta) + \beta U(\beta)$$
for the entropy function $S$ at the energy $U(\beta)$, the asymptotics (\ref{psizerot}) and (\ref{Ubetazero}) of $\Psi(\beta)$ and $U(\beta)$ at zero temperature show that, when $\mu(\cEm) >0$,
\begin{equation}\label{entfin}
\lim_{E\ra H_{{\rm min}, +}} S(E) =\lim_{\beta \ra + \infty} S(U(\beta)) = \log \mu(\cEm)
\end{equation}
and is therefore finite.

As shown by (\ref{entinf}), the relation (\ref{entfin}) still holds when $\mu(\cEm) =0$ (with the convention $\log 0 = -\infty$). In our mathematical approach, the validity of the third law of thermodynamics --- phrased as the existence of a finite limit of $S(U(\beta))$ when $\beta$ goes to $+\infty$ --- is therefore equivalent to the positivity of $\mu(\cEm)$. It notably forbids ``classical mechanical systems" for which $\mu(\cEm) =0$, like the ones discussed in Section \ref{GaussMaxwell} above. This gives a mathematical interpretation of the well-known fact that the third law of thermodynamics reflects the quantum nature of the physical world.

To interpret the relation (\ref{Uderzero}), observe that the derivative of $U(\beta)$ with respect to the temperature $\beta^{-1}$,
\begin{equation}\label{heatcap}
c(\beta^{-1}) := \frac{d U(\beta)}{d(\beta^{-1})} = - \beta^2 U'(\beta) = \beta^2 \Psi''(\beta)
\end{equation}
represents the heat capacity of the system under study. Accordingly, (\ref{Uderzero}) for $k=1$ asserts that the heat capacity goes to zero with the temperature, a well known consequence of the third law of thermodynamics (see for instance \cite{Huang87}, Section 1.7). 

More generally, the relation (\ref{Uderzero}) for $k\geq 1$ arbitrary is easily seen to be equivalent to
$$c^{(k)}(T) = o(T^{-k}) \quad \mbox{ when $T\lra 0_+$}.$$  
}

\section{Complements}\label{Compl}

In this section, we present some complements to Theorem \ref{ThMain} and its proof. 

In Subsections \ref{MainR} and \ref{CR},  we begin by some remarks on its various possible formulations and on the relations between some of the estimates involved in its proof and various classical estimates in probability and analytic number theory. 
Then, in Subsection  \ref{LanfEst}, we present Lanford's direct approach to the asymptotic behavior of the measures $A_n(E)$ investigated in Theorem \ref{ThMain}, based on elementary subadditivity estimates. 
Finally, in \ref{Prod}, we discuss a mathematical interpretation of the second law of thermodynamics in our formalism and its application to Euclidean lattices.

\subsection{The main theorem when $(\cT, H) = (\Rp, {\rm Id}_{\Rp})$}\label{MainR} 

\subsubsection{}\label{SpecR}
In the special case where $(\cE, \cT)$ is $(\Rp, \cB)$, the non-negative real numbers equipped with the $\sigma$-algebra of Borel subsets, and where $H$ is the identity function:
$$H = {\rm Id}_{\Rpa} : \Rp \lra \Rp,$$
Theorem \ref{ThMain} boils down to some result concerning positive Radon measures on $\Rp$ with finite Laplace transforms and there powers under convolution product.

Indeed, let us consider some non-negative Borel measure on $\Rp.$ It is a Radon measure (that is, $\mu(K) < +\infty$ for every compact subset of $\Rp$) if and only if its distribution function:
$$N(E) := \mu([0,E]) < +\infty$$
for every $E \in \Rp.$

Then the function $N: \Rp \lra \Rp$ is non-decreasing and right-continuous, and the measure $\mu$ is the Stieljes measure associated to the function $N$ (extended by $0$ on $\R_-^\ast$), that is with the distributional derivative  of the distribution on $\R$ associated to the locally bounded function $N$.\footnote{Actually this construction establishes a bijection between non-negative Radon measures $\mu$ on $\Rpa$ and non-decreasing right-continuous functions $N: \Rpa \lra \Rpa,$ and one usually writes:
$\int_{\Rp} f \, d\mu = \int_{\Rp} f(x) \, dN(x).$}

When moreover $H = {\rm Id}_{\Rp},$  the partition function $Z$ becomes the Laplace transform of $\mu$, $\Psi$ its logarithm, and $U$ its logarithmic derivative. Namely, for every $\beta \in \Rpa,$ we have :
$$\Psi(\beta) = \log \int_{\Rp} e^{- \beta x} \, d\mu(x),$$
and
$$U(\beta) = \frac{\int_{\Rp} x e^{- \beta x} \, d\mu(x)}{\int_{\Rp} e^{- \beta x} \, d\mu(x).}$$

Moreover, for any integer $n \geq 1,$ we may consider the $n$-th power $$\mu^{\ast n}:= \mu \ast \cdots \ast \mu 
\quad \mbox{ ($n$ times)}
$$ of $\mu$ under the convolution product. It is the Borel measure on $\R$ defined by the equality:
$$\mu^{\ast n} (B) := \mu^{\otimes n} \{(x_1, \ldots,x_n) \in \R^n \mid x_1 + \ldots + x_n \in B \}$$
for any Borel subset $B$ of $\R$. It is easily seen to be a Radon measure supported by $\Rp$. Moreover, for every $E \in \Rp,$
$$A_n(E) = \mu^{\ast n}([0, nE])$$
and therefore
$$S(E) = \lim_{n \ra + \infty}  \frac{1}{n} \log \mu^{\ast n}([0, nE]).$$

\subsubsection{}\label{RedSepcR} Let us return to Theorem \ref{ThMain}, in its general formulation. We may introduce the Borel measure $\tilde{\mu}$ on $\Rp$ defined as the image of the measure $\mu$ by the measurable function $H$:
$$\tilde{\mu} := H_\ast \mu : B \longmapsto \mu (H^{-1}(B)).$$ It is the straightforward that the functions $Z(\beta),$ $\Psi(\beta)$ and $A_n(E)$ attached to the measure space $(\Rp, \cB, \tilde{\mu})$ equipped with the function $\tilde{H} := {\rm Id}_{\Rp}$ coincides with the ones attached to $(\cE,\cT, \mu)$ equipped with $H$.

In particular, the validity of Theorem \ref{ThMain} in the special case discussed in \ref{SpecR} above implies its general validity. However this reduction does not lead to any actual simplification in the derivation of Theorem \ref{ThMain} presented in Section  \ref{PfMain}. One might even argue that  the measure theoretic arguments in Paragraphs \ref{BoundAbove} and \ref{BoundBelow} are actually clearer when presented in the  general setting dealt with in Section \ref{PfMain}.

\subsection{Chernoff's bounds and Rankin's method}\label{CR}

In paragraph \ref{BoundAbove}, the first step in the proof of the convergence of $\log A_n(E) /n$ to
\begin{equation}\label{Sdefagain}S(E) := \inf_{\beta \in \Rpa} (\Psi(\beta) + \beta E)
\end{equation}
has been to establish the upper bound
$$\frac{1}{n} \log A_n(E) \leq \inf_{\beta \in \Rpa} (\Psi(\beta) + \beta E)$$
for any integer $n \geq 1$ and any $E > \Hm$. When $n =1,$ this upper bound reads:
\begin{equation}\label{A1}
\log \mu(H^{-1}([\Hm, E])) =: \log A_1(E) \leq \inf_{\beta \in \Rpa} (\Psi(\beta) + \beta E).
\end{equation}
(This is the content of Proposition  \ref{AZn} when $n=1$, itself a straightforward consequence of Markov's inequality applied to the function $e^{-\beta H}$.)

Inequalities of this type, which provides an upper bound for ``tails probability" in terms of the ``logarithmic moment generating function" $\Psi$, are classically known as \emph{Chernoff's bounds}, by reference to Chernoff's seminal article \cite{Chernoff52}, which constitutes, with the earlier article by Cram\'er \cite{Cramer38}, the starting point of the theory of large deviations. In \cite{Chernoff52}, Chernoff establishes a basic theorem of large deviations, on which Theorem \ref{ThMain} is modeled, by considering in substance a framework similar to the one in Subsection \ref{MesH}, but were $\mu$ is a probability measure. Chernoff's theorem extends the earlier results in \cite{Cramer38}, established under more specific assumptions on the measure $H_\ast \mu$, and the arguments in paragraph \ref{BoundAbove} are direct adaptations of the ones in \cite{Chernoff52}.

In spite of the simplicity of their derivation, Chernoff's bounds like (\ref{A1}) turn out to provide surprisingly sharp estimates\footnote{The equality 
$\lim_{n \ra + \infty} \frac{1}{n} \log A_n(E) \leq \inf_{\beta \in \Rpa} (\Psi(\beta) + \beta E)$ somewhat explains this sharpness. See also \cite{Odlyzko92} for related ``Tauberian estimates".} for tail probabilities, and have led to important inequalities, that play a key role  in probability theory and its application.  We refer the reader to  
\cite{BoucheronLugosiMassart2013}, Chapter 2, for a presentation of such inequalities, from the perspective of recent developments on concentration inequalities.

An avatar of Chernoff's bounds also appears in analytic number theory under the name of \emph{Rankin's trick}. Let 
$$f(s) := \sum_{n = 1}^{+\infty} \frac{a_n}{n^s}$$
be a Dirichlet series with non-negative coefficients, which admits 0 as abscissa of convergence. One is interested in bounding the partial sums
$\sum_{1 \leq n \leq x} a_n$
from above, as a function of $x$ in $[1, +\infty)$. To achieve this, one observes that, for any $\eta \in \Rpa,$
$$\sum_{1 \leq n \leq x} a_n \leq x^\eta \,  \sum_{n=1}^{+\infty} \frac{a_n}{n^\eta} = x^\eta f(\eta).$$ 
One often obtain a sharp upper bound on $\sum_{1 \leq n \leq x} a_n$ by choosing $\eta \in \Rpa$ that minimizes $x^\eta f(\eta).$ 

Such estimates notably appear in Rankin's article \cite{Rankin36} (see proof of Lemma II). Similar arguments had actually been used earlier by Hardy and Ramanujan (see \cite{HardyRamanujan17}, Section 4.1). They constitue nothing but the special case of Chernoff's bound (\ref{A1}) when 
$$\cE = \N_{>0}, \quad \mu:= \sum_{n=1}^{+\infty} a_n \delta_n, \quad \mbox{and} \quad H(n) = \log n.$$

%\cite{HardyLittlewood14}

\subsection{Lanford's estimates}\label{LanfEst}

It turns out that the existence of the limit (\ref{Slim}):  
$$\lim_{ n \ra + \infty} \frac{1}{n} \log A_n(E)
$$
when $E$ belongs to $(\Hm, + \infty)$, together with its concavity as a function of $E$, may be directly established, independently of the more sophisticated arguments\footnote{which, of course, prove more, namely the equality of this limit with the function $S$ defined by the Legendre transform of $\Psi$.} in paragraph \ref{BoundBelow}.

In this paragraph, we briefly discuss this direct approach, which originates in Lanford's work \cite{Lanford73} on the rigorous derivation of ``thermodynamic limits" in statistical mechanics. We refer the reader to the original article \cite{Lanford73} for developments of this approach, which emphasize the role of convexity in the thermodynamic formalism. One should also consult the long introduction \footnote{entitled \emph{Convexity and the notion of equilibrium state in thermodynamics and statistical mechanics}.} of \cite{Israel79} by Wightman for an enlightening discussion of this circle of ideas in a historical perspective.

We are going to present a simple proof of  the following fragment of Theorem \ref{ThMain}:
        
\begin{proposition}\label{LanfProp} Let us consider a measure space with Hamiltonian $((\cE, \cT, \mu)), H)$, as in Subsection \ref{MesH}, and let us assume that the measure $\mu$ is non zero --- \emph{or equivalently, that $\Hm$ is finite} --- and that Condition ${\mathbf T_2}$ is satisfied.

Then, for any $E \in (\Hm, + \infty)$, the limit $\lim_{ n \ra + \infty}  \log A_n(E)/n$ exists in $\R$, and also equals $\sup_{n \geq 1} \log A_n(E)/n$. Moreover, it defines a continuous, non-decreasing, and concave function of $E \in (\Hm, + \infty)$.  
\end{proposition}

Lanford's arguments to derive such a statement  rely on the following subadditivity estimates:

\begin{lemma}\label{LanfIneq}
For any $(E_1, E_2)$ in $[\Hm, + \infty)^2$
and any two positive integers $n_1$ and $n_2,$ 
\begin{equation}\label{AAA}
A_{n_1}(E_1) . A_{n_2}(E_2) \leq A_{n_1 + n_2}\left(\frac{n_1 E_1 + n_2 E_2} {n_1 + n_2}\right). 
\end{equation}
\end{lemma}

\begin{proof} The folllowing inclusion of subsets of $\cE^{n_1 + n_2}$ is a straightforward consequence of their definitions:
$$H_{n_1}^ {-1}((-\infty, n_1 E_1]) \times H_{n_2}^{-1}((-\infty, n_2 E_2]) \subseteq H_{n_1 + n_2}^{-1}((-\infty, n_1 E_1 + n_2 E_2]).$$
These subsets are $\cT^{\otimes (n_1 + n_2)}$-measurable, and by applying the measure $\mu^{\otimes(n_1 + n_2)}$ to this inclusion, we get (\ref{AAA}).
\end{proof}

\begin{proof}[Proof of Proposition \ref{LanfProp}]  
Recall that, according to a well-known observation that goes back to Fekete \cite{Fekete23}, superadditive sequences of real numbers have a simple asymptotic behaviour:

\begin{lemma}\label{lemFekete} Let $(a_n)_{n \geq 1}$ be a sequence of real numbers that is superadditive \emph{(namely, that satisfies
$a_{n_1 +n_2} \geq a_{n_1} + a_{n_2}$ for any two positive integers $n_1$ and $n_2$.)}

 Then the sequence $(a_n/n)_{n \in \N_{\geq 1}}$ admits a limit in $(-\infty, +\infty]$. Moreover:
\begin{equation*}
\lim_{n \rightarrow +\infty} a_n/n = \sup_{n \geq 1} a_n/n.
\end{equation*}
\end{lemma}
\qed

For any $E$ in $(\Hm, +\infty),$ we define define a sequence $(a_n)_{n \geq 1}$ of real numbers by letting:
$$a_n := \log A_n(E).$$
Indeed, the estimates (\ref{AAA}) with $E_1 = E_2$ yields the lower bound $A_n(E) \geq A_1(E)^n$, and this is positive by the very definition of $\Hm$. 

These  estimates also implies that the  sequence $(a_n)_{n \geq 1}$ is superadditive. Moreover, the upper bound (\ref{logAnup}) --- which, as explained in paragraph \ref{BoundAbove}, easily follows from Condition   ${\mathbf T_2}$, once it is expressed as the finiteness ${\mathbf T'_2}$ of the partition function $Z(\beta)$ for every $\beta \in \Rpa$ --- shows that this sequence is bounded from above.  

According to Lemma \ref{lemFekete}, this already establishes the required convergence and finiteness:
$$\lim_{ n \ra + \infty}  \log A_n(E)/n = \sup_{n \geq 1} \log A_n(E)/n \in \R.$$

As a function of $E \in (\Hm, +\infty),$ this limit
$$s(E) := \lim_{ n \ra + \infty}  \log A_n(E)/n$$
is  non-decreasing, like $\log A_n(E)$ for every $n \geq 1.$ For any $E \in (\Hm, +\infty),$ we may define:
$$s(E)_-:= \lim_{\tilde{E} \ra E_-} s(\tilde{E}) \quad \mbox{ and } \quad s(E)_+:= \lim_{\tilde{E} \ra E_+} s(\tilde{E}).$$
Clearly, we have:
\begin{equation}\label{spm}
s(E)_- \leq s(E)_+,
\end{equation}
and the function $s$ is continuous at the point $E$ if and only if equality holds in (\ref{spm}).

Besides, for any $E_1$ and $E_2$ in $(\Hm, + \infty)$, Lanford's estimates (\ref{AAA}) may be written:
%\begin{multline} 
$$\frac{n_1}{n_1 + n_2} \frac{\log A_{n_1}(E_1)}{n_1} 
+
\frac{n_2}{n_1 + n_2} \frac{\log A_{n_2}(E_2)}{n_2} 
\leq \frac{\log A_{n_1 + n_2}((n_1 E_1 + n_2 E_2)/(n_1 + n_2))}{n_1 +n_2}.
$$
%\end{multline}
This implies that, for any $E_1$ and  $E_2$ in $(\Hm, + \infty)$ and any $\alpha_1$ and $\alpha_2$ in $\Q \cap [0,1]$ such that $\alpha_1 + \alpha_2 =1,$ the following inequality holds:
\begin{equation}\label{alphasE}
\alpha_1\, s(E_1) + \alpha_2\, s(E_2) \leq s( \alpha_1 E_1 + \alpha_2 E_2).
\end{equation}

For any $E$ in $(\Hm, +\infty)$ and any $\eta \in \Q \cap (0,1),$ one may easily construct sequences $(E_{1,k})$  and $(E_{2,k})$ in $(\Hm, +\infty)$ such that
$$\lim_{k \ra + \infty} E_{1,k} = \lim_{k \ra + \infty} E_{2,k} =E,$$
and such that $(E_{1,k})$ and $(\eta E_{1,k} + (1-\eta) E_{2,k})$ are increasing and $(E_{2,k})$ is decreasing. Applying (\ref{alphasE}) to $E_1= E_{1,k},$ $E_2 = E_{2,k},$ $\alpha_1 = \eta,$ and $\alpha_2 = 1-\eta,$ and  letting $k$ go to infinity, we obtain:
\begin{equation*}
\eta s(E)_- + (1-\eta) s(E)_+ \leq s(E)_-.
\end{equation*}
By taking the limit where $\eta$ goes to zero, we finally obtain:
$$s(E)_+ \leq s(E)_-.$$
This establishes the continuity of $s$.

Using this continuity, we immediately derive that the estimates (\ref{alphasE}) still holds for any $\alpha_1$ and $\alpha_2$ in $[0,1]$ such that $\alpha_1 + \alpha_2 =1.$ This establishes the concavity of $s$.
\end{proof}

\subsection{Products and thermal equilibrium}\label{Prod}

The formalism developed in Sections \ref{Thermod} and \ref{PfMain} --- that attaches functions $\Psi$ and  $S$  to a measure space  $(\cE, \cT,\mu)$ and to a non-negative function $H$ on $\cE$ satisfying {\bf SE} --- satisfies a simple but remarkable compatibility with finite products, that we want to discuss briefly.

\subsubsection{Products of measures spaces with a Hamiltonian}
Assume that, for any element $i$ in some non-empty finite set $I$, we are given a measure space $(\cE_i, \cT_i, \mu_i)$ and a measurable function $H_i: \cE_i \lra \R_+$ as in paragraph \ref{MesH} above.

Then we may form the product measure space $(\cE, \cT, \mu)$ defined by the set
$\cE := \prod_{i \in I} \cE_i$ equipped with the $\sigma$-algebra $\cT := \bigotimes_{i \in I} \cT_i$ and the product measure $\mu := \bigotimes_{i \in I} \mu_i$. 

We may also define a measurable function
$$H: \cE \lra \R_+$$
by the formula
$$H := \sum_{i \in I} {\rm pr}_i^\ast H_i,$$
where ${\rm pr}_i: \cE \lra \cE_i$ denotes the projection on the $i$-th factor.

Let us assume that, for every $i\in I$,  $(\cE_i, \cT_i, \mu_i)$ and $H_i$ satisfy the condition ${\bf T_2}$, or equivalently that the functions $e^{-\beta H_i}$ is $\mu_i$-integrable for every $\beta \in\R^\ast_+.$ 

Then $(\cE, \cT, \mu)$ and $H$ are easily seen to satisfy ${\bf T_2}$ also, as a consequence of Fubini's Theorem.
Actually Fubini's Theorem shows that the function $Z: \Rpa \lra \Rpa$ and $\Psi: \R^\ast_+ \lra \R$ attached to the above data, defined as in Paragraph \ref{MesH} by the formulae 
$$Z(\beta) = \int_\cE e^{-\beta H} d\mu \quad \mbox{and} \quad  \Psi(\beta) := \log Z(\beta)$$
and the ``partial functions" $Z_i$ and $\Psi_i$, $i \in I,$ attached to the measured space $(\cE_i, \cT_i, \mu_i)$ equipped with the function $H_i$ by the similar formulae 
$$Z_(\beta) :=  \int_{\cE_i} e^{-\beta H_i} d\mu_i \quad \mbox{and} \quad \Psi_i(\beta) := \log Z_i(\beta) $$
satisfy the  relations:
\begin{equation}\label{addPsi}
Z(\beta) = \prod_{i \in I} Z_i(\beta) \quad \mbox{and} \quad \Psi(\beta) = \sum_{i \in I} \Psi_i(\beta).
\end{equation}

In particular, the functions $U$ and $U_i,$ $i \in I$, defined by (\ref{UH}), satisfy the additivity realtion:
\begin{equation}\label{UExt}
U(\beta) = \sum_{i \in I} U_i(\beta).
\end{equation}

\subsubsection{The entropy function associated to a product and  the second law of thermodynamics}

From now on, let us also assume that,  for every $i\in I$, Condition $\bf T_1$ holds, namely that $\mu_i(\cE_i) = +\infty$. Then  $\mu(\cE) = +\infty$  --- in other words, $(\cE, \cT, \mu)$ also satisfies $\bf T_1$ --- and we may apply Theorem \ref{ThMain} to the data $(\cE_i, \cT_i, \mu_i, H_i)$, $i \in I,$ and $(\cE, \cT, \mu, H)$.

Notably, if $H_{i,{\rm min}}$ (resp. $\Hm$) denotes the essential infimum of the the function $H_i$ on the measure space $(\cE_i, \cT_i, \mu_i)$ (resp., of $H$ on $\cE,\cT, \mu)$), we may define some concave functions
$$S_i: \, (H_{i,{\rm min}}, +\infty) \lra \R, \;\;\; \mbox{ for $i \in I$,}$$
and 
$$S: \, (\Hm, +\infty) \lra \R.$$

Observe also that, as a straightforward consequence of the definitions, we have:
$$\Hm = \sum_{i \in I} H_{i,{\rm min}}.$$

The expression (\ref{addPsi}) of $\Psi$ as sum of the $\Psi_i$'s translates into the following description of the entropy function $S$ in terms of the $S_i$'s:

\begin{proposition}\label{Prop:secondlaw} 1) For each $i \in I,$ let $E_i$ be a real number in $(H_{i,{\rm min}}, + \infty [.$

Then the following inequality is satisfied: 
 \begin{equation}\label{Secondlaw}
\sum_{i \in I} S_i(E_i) \leq S(\sum_{i \in I} E_i).
\end{equation}
 
 Moreover equality holds in (\ref{Secondlaw}) if and only if the positive real numbers $S'(E_i),$ $i \in I,$ are all equal. When this holds, if $\beta$ denotes their common value, we also have: 
 $$\beta = S'(\sum_{i \in I} E_i).$$ 
 
 2) Conversely, for any $E \in (\Hm, +\infty),$ there exists a unique family $(E_i)_{i\in I} \in \prod_{i \in I} (H_{i,{\rm min}}, + \infty )$ such that $$E= \sum_{i \in I} E_i \;\; {\mbox  and }\;\; S(E) = \sum_{i \in I} S_i(E_i).$$
 Indeed, if $\beta = S'(E),$ it is given by $$(E_i)_{i \in I} = (U_i(\beta))_{i \in I},$$
where $U_i = -\Psi'_i.$ \end{proposition}

\begin{proof} Let $(E_i)_{i\in I}$ be an element of  $\prod_{i \in I} (H_{i,{\rm min}}, + \infty )$.   According to Theorem \ref{ThMain}, 3), we have, for every $i \in I$: 
\begin{equation}
S(E_i) = \inf _{\beta > 0} (\beta E_i + \Psi_i(\beta)).
\end{equation}
Moreover, the infimum is attained for a unique  $\beta$ in $\R^\ast_+$, namely $S'(E_i)$.

Similarly, for $E := \sum_{i \in I} E_i,$
\begin{equation}
S(E) = \inf _{\beta > 0} (\beta E + \Psi(\beta)),
\end{equation}
and the infimum is attained for a unique positive $\beta$, namely $S'(E)$.

Besides, the additivity relation (\ref{addPsi}) shows that, for every $\beta$ in $\R^\ast_+,$
$$\beta E + \Psi(\beta) = \sum_{i \in I} \left(\beta E_i + \Psi_i(\beta)\right).$$

Part 1) of the proposition  
directly follows from these observations.
Part 2) follows from Part 1) and from the relation $\Psi'= \sum_{i \in I} \Psi'_i.$
 \end{proof}
 
 Proposition \ref{Prop:secondlaw} notably asserts that, for any $E$ in  $(\Hm, +\infty),$
 $$S(E) = \max \left\{\sum_{i \in I} S(E_i);  (E_i)_{i \in I} \in \prod_{i \in I} (H_{i,{\rm min}}, + \infty ), \sum_{i\in I} E_i =E\right\}.$$
 
 In other words, the function $S$ is the ``tropical convolution" of the functions $(S_i)_{i \in I}$. (Recall that, in tropical mathematics, products are replaced by sums, and sums and integrals by maxima and suprema.)

 The above results admit the following physical interpretation, in line with the discussion in Subsection \ref{Phys}. 
 
 The product $((\cE, \cT, \mu),H)$ of the measure space with Hamiltonian represents an elementary system composed of basic elementary systems (indiced by $I$). The relation (\ref{UExt}) expresses the fact that the energy is an extensive quantity. Proposition \ref{Prop:secondlaw} shows that the entropy $S(E)$ of the composite system may be computed as the sum of the entropies $S_i(E_i)$ of its subsystems for the (unique) values of the energies $(E_i)_{i \in I}$ of its subsystems which add up to $E$, and maximizes the sum of these partail entropies, or equivalently that gives each of the subsystems the same temperature as the total system. 
 
 In this way,  Proposition \ref{Prop:secondlaw} appears as a mathematical interpretation of the second law of thermodynamics.
 
% It describes how the energy $E$ is distributed among the subsystems $((\cE_i, \cT_i, \mu_i), H_i)$ that constitutes the system $((\cE, \cT, \mu),H)$: namely, in the unique way that gives the same temperature to each of these subsystems.  

 \subsubsection{Application to Euclidean lattices}\label{SPLatt}
 
 In paragraph \ref{MSHE}, in order to derive Theorem \ref{Thhonthot} from Theorem \ref{ThMain}, we have associated a measure space with Hamiltonian to any Euclidean lattice, defined  by the relations (\ref{MSHE1})-(\ref{MSHE4}).
 
 It directly follow from its definition that this construction is  compatible with direct sums of Euclidean lattices: for any two Euclidean lattices $\Eb_1$ and $\Eb_2,$ the measure space with Hamiltonian associated to $\Eb := \Eb_1 \oplus \Eb_2$ may be identified with the product of the measure spaces with Hamiltonian associated to $\Eb_1$ and $\Eb_2$.
 
Taking into account the relation (\ref{Shont}) between the invariant $\hont$ attached to Euclidean lattices and the entropy function of the associated measure spaces with Hamiltonian,  Proposition \ref{Prop:secondlaw} (with $I= \{1,2\}$) applied to this product decomposition immediately establishes 
 Corollary \ref{SecondLawLattice}. Using (\ref{xbetaassoc}), it actually shows that the maximum in the right-hand side of  (\ref{equ:secondlaawlattice}) is achieved at unique pair $(x_1, x_2)$, namely
 when
 $$x_1 = - \pi^{-1} \, \theta'_{\Eb_1}(\beta)/\theta_{\Eb_1}(\beta) \quad \mbox{and} \quad x_2 = - \pi^{-1} \,  \theta'_{\Eb_2}(\beta)/\theta_{\Eb_2}(\beta),$$
 where $\beta \in \Rpa$ is defined by the equality:
 $$x = - \pi^{-1}  \, \theta'_{\Eb}(\beta)/\theta_{\Eb}(\beta).$$
 
\section{The approaches of Poincar\'e and of Darwin-Fowler}\label{PDF}

In this section, we consider a measure space equipped with some Hamiltonian $((\cE, \cT, \mu),H)$ as in Section \ref{Thermod},  and we use the notation introduced in  \ref{MesH} and \ref{subsMainTh}. Our aim will be to give, under   suitable assumptions on the measure $H_\ast \mu$, some asymptotic expression   for 
$$ A_n(E) := \mu^{\otimes n} (\{x \in \cE^n \mid H_n(x) \leq nE\})$$
when $n$ goes to infinity. These expressions will be refined versions of the limit formula
\begin{equation}\label{AnSbis}
\lim_{ n \ra + \infty} \frac{1}{n} \log A_n(E) = S(E),
\end{equation}
valid for every $E \in (\Hm, +\infty)$, established by probabilistic arguments in Section 6.2.

Our derivation of these asymptotic expressions will rely on some arguments involving Fourier and Laplace transforms in the complex domain. The measure $A_n(E)$ will be expressed as a weighted integral along a suitable complex path of the function 
$$\left[Z(s) e^{Es}\right]^n$$
of the complex variable $s$ in the right half-plane defined by $\Re s >0$.
 The  asymptotic expression of $A_n(E)$ when $n$ goes to infinity will be obtained as an application of Laplace's method to this integral. 
 
 This derivation may be seen as an application of the saddle-point method (see for instance \cite{Copson65}, Chapters 7 and 8) and is a modern rendering of arguments in the articles \cite{Poincare12} by Poincar\'e and
\cite{DarwinFowler1922},
\cite{DarwinFowler1922II}, and 
\cite{DarwinFowler1923}
by Darwin and Fowler, devoted to the statistical mechanics of classical and quantum systems.

We will rely on the analyticity and convexity properties of the functions $Z$ and $\Psi$ and on the construction of the entropy function $S$ presented in Subsection \ref{PUS}, but not on the results in Subsection~\ref{ConvAn}. Actually, from the asymptotic expressions for $A_n(E)$ established in \ref{AsSaddle} \emph{infra} under some additional assumptions on the measure $H_\ast \mu,$, one may recover the validity of the limit formula (\ref{AnSbis}) under the general assumptions of Theorem \ref{ThMain} --- which constituted the main result of Subsection  \ref{ConvAn} --- by some simple approximation arguments that we present in Subsection \ref{Approx}.

Needless to say, to comply with the change of standards of %precision and 
rigor during the last century, we have been led to formulate the asymptotic results in this section with more precision than in the original articles by Poincar\'e and Darwin and Fowler.\footnote{Notably by the introduction of Condition $\mathbf{L}^2_\epsilon$ in \ref{CondPDF} and \ref{AsL2epsilon} \emph{infra}. Observe that Theorem \ref{ThMainPoinc} requires such an additional assumption on the measure $\mu$ to be valid, as shown by a comparison with Theorem \ref{ThMainDF}.} The informal character of Poincar\'e's arguments\footnote{Poincaré sketches  an argument, based on the use of Laplace transforms, to derive an asymptotic formula of the kind of the one established in Theorem  \ref{ThMainPoinc}, but is concerned with applications of his results in situations where the measure $H_\ast \mu$ may have a discrete support, or even satisfy Condition $\mathbf{DF}$, where this derivation actually fails.} appears to have  led to divergent appreciations of   its significance (compare for instance the discussions by Planck in \cite{Planck1914}, Appendix II, and  \cite{Planck1921}, and the comments by Fowler in \cite{Fowler36}, Section 6.7). It is however remarkable that now routine analytic techniques are enough to transform the arguments in \cite{Poincare12} and \cite{DarwinFowler1922}, which are either informal or of limited scope, into rigorous derivations of the general limit formula (\ref{AnSbis}). 
\subsection{Preliminaries}

\subsubsection{Laplace transforms of measures on $\Rpa$}\label{LaplaceTr}

Let $\mu$ be a complex valued Radon measure on $\R$, supported by $\R_+$. For any $\gamma \in \R,$ we shall say that \emph{the measure $\mu$ satisfies the condition $\Sigma_\gamma$} when
\begin{equation}\label{betabd}
\int_\R e^{-\gamma E} d\vert \mu \vert (E)  < + \infty.
\end{equation}
When this holds, for any $s$ in the right half-plane
$$\C_{\geq \gamma} := \{ s \in \C \mid \Re s \geq \beta\},$$
we may consider the integral
$$\cL \mu (s) := \int_\R e^{-sE} d\mu(E).$$

The function 
$$\cL \mu : \C_{\geq \gamma} \lra \C$$
so-defined --- the \emph{Laplace transform} of $\mu$ --- is continuous on $\C_{\geq \gamma}$ and holomorphic on its interior $\C_{> \gamma}$. Its restriction ton any vertical line $\beta + i \R$ is, up to a normalization, the Fourier transform of the measure of finite mass $e^{-\beta E} d\mu(E)$, and accordingly, uniquely determines $\mu.$  Moreover, for any $s \in \C_{\geq \gamma},$
\begin{equation*}
\vert \cL \mu(s) \vert \leq \cL \vert \mu \vert (\gamma) := \int_\R e^{-\gamma E} d\vert \mu \vert (E).
\end{equation*}

For any two Radon measures $\mu_1$ and $\mu_2$ on $\R_+,$ we may consider their convolution product $\mu_1 \ast \mu_2$, namely the Radon measure on $\R_+$ defined by 
$$\mu_1 \ast \mu_2 (E) := (\mu_1 \boxtimes \mu_2) \left(\Sigma^{-1} (E)\right)$$
for any bounded Borel subset of $\R_+$, where 
$$\Sigma : \R_+ \times \R_+ \lra \R_+$$
denotes the sum map. From basic measure theory, it follows that, if $\mu_1$ and $\mu_2$ both satisfy the integrability condition $\Sigma_\gamma$, then  $\mu_1 \ast \mu_2$ satisfies it also, and that, for any $s \in \C_{\geq \gamma},$
\begin{equation}\label{LaplMult}
\cL( \mu_1 \ast \mu_2)(s) = \cL \mu_1(s) . \cL \mu_2(s).
\end{equation}

\begin{proposition}\label{LaplaceL2} 
Let $\mu$ be a Radon measure on $\R$, supported by $\R_+,$ which satisfies $\Sigma_\gamma$ for some $\gamma \in \R$.

For any $\beta \in [\gamma, +\infty),$ the following two conditions are equivalent:

(i) the measure $e^{-\beta E} \, d\mu(E)$ on $\R_+$ is defined by some $L^2$-function on $\R_+.$

(ii) $\int_{- \infty}^{+ \infty} \vert \cL (\beta + i \xi) \vert^2 \, d\xi < + \infty.$

When they are satisfied and when $\beta >0,$ then, for every $E \in \R$,
\begin{equation}\label{muLbeta}
\mu((-\infty, E]) = 
 \frac{1}{2\pi} \int_\R \cL \mu(\beta + i\xi)\,  e^{E(\beta + i\xi)}\, \frac{d\xi}{\beta + i\xi}.
\end{equation}
\end{proposition}

Condition (i) precisely means that the measure $\mu$ is absolutely continuous with respect to the Lebesgue measure $\lambda$ on $\R$ and that, if $f$ denotes the Radon-Nikodym derivative\footnote{Namely, a non-negative Borel function supported by $\R_+$ such that
$\mu = f \lambda$, or more exactly, its class modulo equality $\lambda$-almost everywhere.}
 $d\mu/ d\lambda$
--- the function $(E \longmapsto e^{-\beta E} \, f(E))$ is square integrable on $\R_+$.

%\begin{equation}\label{mue}
%\mu([0,E]) = \int_0^E f(x) \, dx.
%\end{equation}

Observe also that, when (ii) holds,  the integral in the right-hand side of (\ref{muLbeta}) is absolutely convergent, since $e^{E(\beta + i\xi)}/ (\beta + i\xi)$ is, like $\cL \mu(\beta + i\xi)$, a function of $\xi$ in $L^2(\R).$ This integral may be seen as an integral along the ``infinite vertical path" in the complex plane defined by the map $(\R \lra \C, t \longmapsto \beta + it)$. Accordingly, the equality (\ref{muLbeta}) may be written more suggestively as:
 \begin{equation}\label{FintBis}
\mu([0,E]) = \frac{1}{2\pi i} \int_{\beta + i\R} \cL \mu(s)\,  e^{Es}\, s^{-1} ds.%/s.%\frac{ds}{s}.
\end{equation}

\begin{proof}[Proof of Proposition \ref{LaplaceL2}] For any $\beta \in [\gamma, + \infty),$ the Radon measure $\mu_\beta$ defined as 
$$d\mu_\beta (x) = e^{-\beta x} \, d\mu(x)$$
has a finite mass, and its Fourier transform $\cF \mu_\beta$ is a continuous function, defined for every $\xi \in \R$ as:
%\begin{equation}\label{Fmubeta}
%\begin{split}
%\cF \mu_\beta (\xi) & := \int_\R e^{-i x \xi} \, d\mu_\gamma(x) \\
%& = \int_{\R_+} e^{-(\beta + i  \xi) x} \, d\mu(x) \\
%& = \cL \mu( \beta + i \xi).
%\end{split}
%\end{equation}
\begin{equation}\label{Fmubeta}
\cF \mu_\beta (\xi)  := \int_\R e^{-i x \xi} \, d\mu_\gamma(x)  = \int_{\R_+} e^{-(\beta + i  \xi) x} \, d\mu(x)  = \cL \mu( \beta + i \xi).
\end{equation}

By Parseval's Theorem, the Radon measure $\mu_\beta$ belongs to $L^2(\R)$ if and only if $\cF \mu_\beta$ belongs to $L^2$. This proves the equivalence of (i) and (ii). 

Moreover, when $\beta >0,$ for every $E \in \R$, we may consider the function $\phi_{E, \beta}$ in $L^1(\R) \cap L^\infty (\R)$ defined by
$$\phi_{E, \beta}(x) := e^{\beta x} \mathbf{1}_{(-\infty, E]}(x).$$
Its Fourier transform $\cF \phi_{E, \beta}$ is easily computed; namely, for every $\xi \in \R,$ we have:
\begin{equation}\label{Fphi}
\cF \phi_{E, \beta}(\xi) := \int_\R e^{-i x\xi} \phi_{E,\beta} (x) \, dx = \int_{- \infty}^E e^{-i (\xi + i \beta) x} \, dx = (\beta - i \xi)^{-1} e^{E (\beta - i \xi)}.
\end{equation} 

The functions $\phi_{E, \beta}$ and  $\cF \phi_{E, \beta}$ belong to $L^2(\R)$, and Parseval's formula applied to $\phi_{E,\beta}$ and $\mu_\beta$ shows that
$$\int_\R \overline{\phi_{E, \beta}}(x) \, d\mu_\beta(x) = \frac{1}{2 \pi} \int_\R \overline{\cF \phi_{E, \beta}}(\xi)\, \cF\mu_\beta (\xi) \, d\xi.$$
This establishes (\ref{muLbeta}). Indeed, according to the very definitions of $\phi_{E,\beta}$ and $\mu_\beta$, we have:
$$\int_\R \phi_{E,\beta}(x) \, d\mu_\beta(x) = \mu ((-\infty, E]),$$
and (\ref{Fmubeta}) and (\ref{Fphi}) imply:
 $$\frac{1}{2 \pi} \int_\R \overline{\cF \phi_{E, \beta}}(\xi)\, \cF\mu_\beta (\xi) \, d\xi = \frac{1}{2 \pi} \int_\R  (\beta + i \xi)^{-1} e^{E(\beta + i \xi)} \, \cL(\beta + i \xi) \, d\xi.
 $$
\end{proof}

\subsubsection{Asymptotics of complex integrals by Laplace's method}\label{AsymLapl}

Let $I$ be some interval in $\R$ that contains $0$ in its interior, and let $g$ and $F$ be two complex valued Borel functions on $I$.

Let assume that they satisfy the following condition:

$\mathbf{L_1 :}$ \emph{The function $F$ is bounded on $I$ and there exists $N_0 \in \N$ such that $g.F^{N_0}$ is integrable on $I$.}

Then, for any integer $N \geq N_0,$ the function $g.F^N$ is integrable on $I$ and we may consider its integral:
$$I_N := \int_I g(t) \, F(t)^N \, dt.$$

Let us introduce some further conditions on $g$ and $F$, that will allow us to use Laplace's method to obtain the asymptotic behavior of these integrals when $N$ goes to infinity.

$\mathbf{L_2 :}$ \emph{The functions $g$ and $F$ are continuous and do not vanish at $0$. Moreover, there exists $\alpha \in \Rpa$ such that
\begin{equation*}
F(t) = F(0) (1- \alpha t^2) + o(t^2) \quad \mbox{ when $t \lra 0.$}
\end{equation*}
}

$\mathbf{L_3 :}$ \emph{For any $\eta \in \Rpa,$}
\begin{equation*}
M_\eta := \sup_{t\in I \setminus(-\eta, \eta)} \vert F(t) \vert < \vert F(0) \vert.
\end{equation*}

For instance, when $I = \R$, the function $F$ is bounded on $\R$ and satisfies $\mathbf{L_3}$ when it is continuous and satisfies
\begin{equation}\label{limH}
\lim_{\vert t \vert \lra + \infty} F(t) = 0
\end{equation}
and 
\begin{equation}\label{HtHzero}
\vert F(t) \vert < \vert F(0) \vert \quad \mbox{ for any $t\in \Rpa.$} 
\end{equation}

\begin{proposition}\label{PropLap} With the above notation, when Conditions $\mathbf{L_1}$, $\mathbf{L_2}$ and $\mathbf{L_3}$ are satisfied, we have:
\begin{equation}\label{INeq}
I_N \sim g(0) F(0)^N   \sqrt{\frac{2\pi}{\alpha N}} \quad \mbox{ when $N \lra + \infty.$} 
\end{equation}
\end{proposition}

\begin{proof} For any $\eta \in \Rpa$ and any integer $N \geq N_0,$ we may define:
$$I_N(\eta) := \int_{I \cap [-\eta,\eta]} g(t) \, F(t)^N \, dt.$$
We have:
$$\vert I_N(\eta) - I_n \vert \leq \int_{I \setminus [-\eta,\eta]} \vert g(t)\vert \, \vert F(t)\vert^N \, dt \leq \int_I \vert g(t)\vert \, \vert F(t)\vert^N \, dt \leq 
M_\eta^{N - N_0} \, \int_I \vert g(t)\vert \, \vert F(t)\vert^{N_0} \, dt,$$
and therefore:
\begin{equation}\label{CompIN}
\vert I_N(\eta) - I_N \vert = O(M_\eta^ N) \quad \mbox{ when $N \lra + \infty.$}
\end{equation}

Besides, for $t \in I$ close enough to $0$, we may write:
$$F(0)^{-1} \, F(t) = e^{-\alpha t^2 + \epsilon(t) t^2},$$
with
$$\lim_{ t \lra 0} \epsilon (t) = 0.$$
Let us choose $\eta \in \Rpa$ small enough, so that $I$ contains the interval $[-\eta, \eta]$ and
$$\tilde{\alpha} := \sup_{\vert t \vert \leq \eta} \Re \epsilon (t) < \alpha \quad \mbox{ and } \quad M:=\sup_{\vert t \vert \leq \eta} \vert g(t) \vert < + \infty.$$
Then, for any integer $N \geq N_0,$ we have:
\begin{equation*}
\begin{split}
g(0)^{-1} F(0)^{-N} I_N(\eta) & = \int_{-\eta}^\eta g(0)^{-1} g(t) \, [F(0)^{-1} F(t)]^N \, dt   \\
& =  \int_{-\eta}^\eta g(0)^{-1} g(t) \, e^{-(\alpha - \epsilon(t))Nt^2} \, dt \\
&= \sqrt{N}^{-1} \int_{-\sqrt{N}\eta}^{\sqrt{N}\eta} g(0)^{-1} g(u/\sqrt{N}) \, e^{-(\alpha - \epsilon(u/\sqrt{N}))u^2} \, du. 
\end{split}
\end{equation*}
(We have performed the change of variables $u = \sqrt{N} t$.) When $N$ goes to infinity, the last integral converges to 
$$\int_{-\infty}^{+\infty} e^{-\alpha u^2} \, du = \sqrt{2\pi/ \alpha}.$$
Indeed, for any fixed $u \in \R$, its integrand converges to $e^{-\alpha u^2}$ when $N$ goes to $+\infty,$ and its absolute value is bounded from above by $\vert g(0)\vert^{-1} M e^{-(\alpha -\tilde{\alpha}) u^2}$, which is integrable over $\R$.
This proves that
\begin{equation}\label{INeta}
I_N(\eta) \sim g(0) F(0)^N   \sqrt{\frac{2\pi}{\alpha N}} \quad \mbox{ when $N \lra + \infty.$} 
\end{equation}

Since $M_\eta < \vert F(0) \vert,$ the asymptotic equivalent (\ref{INeq}) for $I_N$ follows from (\ref{CompIN}) and (\ref{INeta}). 
\end{proof}

\subsection{Asymptotics of $A_n(E)$ by the saddle-point method}\label{AsSaddle}

We return to the notation introduced at the beginning of this section. Namely, we consider a measure space equipped with some Hamiltonian $((\cE, \cT, \mu),H)$ as in Section \ref{Thermod},  and we freely use the notation introduced in  \ref{MesH} and \ref{subsMainTh}.

\subsubsection{The conditions ${\mathbf L^2_\epsilon}$ and $\mathbf{DF}$}\label{CondPDF}

Let us introduce the following additional conditions on the measure $\mu$ and on the functions $H$: 

${\mathbf L^2_\epsilon}: $ \emph{For any $\beta \in \Rpa$, the measure 
$$H_\ast (e^{-\beta H} \, \mu) = e^{-\beta {\rm Id_{\R_+}}}\, H_\ast \mu$$
on $\R_+$ is defined by some $L^2$ function;}

\noindent and: 

$\mathbf{DF} :$ \emph{There exists $\eta$ in $\Rpa$ such that, $\mu$-almost everywhere on $\cE$, the function $H$ takes its values in $\N \eta$, }or equivalently \emph{such that  the measure $H_\ast \mu$ is supported by }$\N \eta$.

We are going to derive some asymptotic representation of $A_n(E)$ when $n$ goes to infinity when, besides Conditions $\mathbf{T_1}$ and $\mathbf{T_2}$,  one of these conditions holds.

Clearly the conditions ${\mathbf L^2_\epsilon}$ and $\mathbf{DF}$ are never simultaneously satisfied, unless $\mu =0.$

Let us also indicate that the pairs  $((\cE, \cT, \mu),H)$ that arise from classical mechanics, as discussed in the introduction of Section \ref{Thermod} and in paragraph \ref{HRM}, often satisfy 
Condition ${\mathbf L^2_\epsilon}$. This is related to the following observation, that we leave as an exercise for the reader: \emph{if $\cE$ is a $C^\infty$ manifold of pure dimension $n$, if $\mu$ is defined by some $C^\infty$ density on this manifold, and if the function $H: M \lra \R_+$ is $C^\infty$ and proper, then the measure $H_\ast \mu$ is locally $L^2$ when $n \geq 2$ and $H$ is a Morse function.} Clearly such pairs never satisfy Condition  $\mathbf{DF}$, except in trivial cases.

\subsubsection{The approach of Poincaré}\label{AsL2epsilon}

%Let us introduce the following regularity condition on the measure $H_\ast \mu$:

\begin{proposition}\label{Poinca1} Let us assume that Condition ${\mathbf L^2_\epsilon}$ is satisfied. 

Then  Condition $\mathbf{T_2}$ holds, and for any $\beta \in \Rpa$, the function $(\xi \mapsto Z(\beta + i \xi))$ belongs to $\mathcal{C}_0(\R) \cap L^2(\R)$. Moreover, for any $E \in \R$ and any integer $n \geq 1,$
\begin{equation}\label{PoincaFor}
A_n(E) =   \frac{1}{2\pi} \int_\R \left[Z(\beta + i\xi)\,  e^{E(\beta + i\xi)}\right]^n\, \frac{d\xi}{\beta + i\xi} =: \frac{1}{2\pi i} \int_{\beta + i \R} \left[ Z(s)\, e^{Es}\right]^n \, s^{-1} ds.
\end{equation}
\end{proposition}

%\begin{equation}
%A_n(E) = \mu^{\ast n} ([0, nE]) = \frac{1}{2\pi i} \int_{\beta + i \R} Z(s)^N \, e^{NEs} \, s^{-1} ds =  \frac{1}{2\pi} \int_\R Z(\beta + it)^N\,  e^{NE(\beta + it)}\, \frac{dt}{\beta + it}
%\end{equation}

Observe that, since the function $(\xi \mapsto Z(\beta + i \xi))$ is both $L^\infty$ and $L^2$, the function 
$$\left(\xi \longmapsto \left[Z(\beta + i\xi)\,  e^{E(\beta + i\xi)}\right]^n\right)$$ is $L^2$ for any positive integer $n$. The integrals in the right-hand side of (\ref{PoincaFor}) are therefore absolutely convergent.

\begin{proof} When Condition ${\mathbf L^2_\epsilon}$ is satisfied, then, for any $\beta \in \Rpa,$ the measure $e^{-\beta {\rm Id_{\R_+}}}\, H_\ast \mu$ on $\R_+$ is defined by some $L^1$ function on $\Rpa$. (Indeed, the equality
$$e^{-\beta {\rm Id_{\R_+}}}\, H_\ast \mu = e^{-(\beta/2) {\rm Id_{\R_+}}}. e^{-(\beta/2) {\rm Id_{\R_+}}}\, H_\ast \mu$$
shows that it is the product of two $L^2$-functions on $\Rpa$.) This implies that $e^{- \beta H} \mu$ has a finite mass for every $\beta \in \Rpa$, that is, that Condition $\mathbf{T_2}$ holds.
 
 By the very definitions of the image measure $H_\ast \mu$ and of the partition function $Z$ (see (\ref{defZ} and Proposition \ref{Zan}), 
 we have:
 $$A_1(E) := \mu( H^{-1}((-\infty, E]) = H_\ast \mu((-\infty, E]) \quad \mbox{for every $E \in \R$},$$
 and:
 $$Z(s) = \cL (H_\ast \mu)(s) \quad \mbox{for every $s \in \C_{>0}$}.$$
 
 Using (\ref{Fmubeta}) with $H_\ast \mu$ instead of $\mu$, this shows that the function $(\xi \mapsto Z(\beta + i \xi))$ is the Fourier transform of the measure 
 $e^{-\beta {\rm Id_{\R_+}}}\, H_\ast \mu$ on $\R_+$, which is defined by some function both in $L^2$ and $L^1$, and therefore belongs to $\mathcal{C}_0(\R) \cap L^2(\R)$. Moreover, when $n =1,$ the equality (\ref{PoincaFor}) follows from Proposition \ref{LaplaceL2} applied to the measure $H_\ast \mu$ on $\R_+$. 
 
 Observe  that Proposition \ref{LaplaceL2} also shows that, conversely, when $\mathbf T_2$ holds and $Z$ is $L^2$ on the vertical line $\beta + i \R,$ then the measure $ e^{-\beta {\rm Id_{\R_+}}}\, H_\ast \mu$ on $\R_+$ is defined by some $L^2$-function.  
 
 Let us now consider an arbitrary positive integer $n$, and let us introduce the pair
 \begin{equation}\label{pairn}
 ((\cE^N, \cT^{\otimes n}, \mu^{\otimes n}), H_n).
 \end{equation}
 It is nothing but the product, in the sense of Subsection \ref{Prod}, of $n$-copies of the given measure space with Hamiltonian $((\cE, \cT, \mu), H)$. 
 
 As observed in $\emph{loc. cit.}$, it follows from Fubini theorem  that it still satisfies Condition $\mathbf T_2$  and its partition function is $Z^n$. This function is $L^2$ on the vertical line $\beta + i \R$ for any $\beta \in \Rpa$, and the above observation shows that the pair (\ref{pairn}) also satisfies Condition  ${\mathbf L^2_\epsilon}$ and that we may apply to it the equality (\ref{PoincaFor}) with $n=1$.
 
 This shows that, for any $E \in \R,$
 $$\mu^{\otimes n} (H_n^{-1}((-\infty, nE]))=
 \frac{1}{2\pi i} \int_{\beta + i \R} Z(s)^n\, e^{nEs} \, s^{-1} ds$$
 and establishes (\ref{PoincaFor}) in general.
\end{proof}

For any given $E \in (\Hm, +\infty)$, we may derive an asymptotic expression for $A_n(E)$ when $n$ goes to $+\infty$ from the integral formulae (\ref{PoincaFor}), by choosing $$\beta := S'(E)$$ and then applying Laplace's method. In this way, we shall establish:

\begin{theorem}\label{ThMainPoinc} Let us assume that conditions ${\mathbf T_1}$ and ${\mathbf L^2_\epsilon}$ are satisfied. Then, for any $E \in (\Hm, +\infty),$ we have:
\begin{equation}\label{AnAsPoinc}
A_n(E) \sim \left[\pi \beta^2 \Psi''(\beta)n\right]^{-1/2} \, e^{n S(E)},
\end{equation}
with $\beta := S'(E),$  when the integer $n$ goes to infinity.
\end{theorem}

\begin{proof} For any $E\in \R$ and $\beta \in \Rpa,$ the right-hand side of the integral expression (\ref{PoincaFor}) for $A_n(E)$ may be written as
\begin{equation}\label{AnInt}\frac{1}{2\pi} \int_\R \left[Z(\beta + i\xi)\,  e^{E(\beta + i\xi)}\right]^n\, \frac{d\xi}{\beta + i\xi} 
= \int_I g_\beta(t) \, F_\beta(t)^n \, dt,
\end{equation}
where
$$I := \R, \quad
g_\beta(t) := [2\pi (\beta + it)]^{-1},
\quad \mbox{and} \quad  F_\beta(t) := Z(\beta + it) e^{E(\beta + it)}.$$

The functions $g_\beta$ and $F_\beta$ satisfy the conditions $\bf L_1$ and $\bf L_3$ on the functions $g$ and $F$ introduced in pour discussion of Laplace's method in paragraph \ref{AsymLapl}. 

Indeed, $F_\beta$ is continuous on $I := \R$ and satisfies (\ref{limH}), as shown in Proposition \ref{Poinca1}; it also satisfies (\ref{HtHzero}), as a consequence of the estimates (\ref{modZ}) on $\vert Z \vert$ and of the fact that $e^{-it H}$ is not $\mu$-almost everywhere constant for any $t \in \Rpa$, as a straightforward consequence of ${\mathbf L^2_\epsilon}$. This implies that $F_\beta$ is bounded and satisfies $\bf L_3$. Moreover, as observed after Proposition \ref{Poinca1}, the function $g_\beta F_\beta$ is integrable, and this establishes $\bf L_1$.

Let us now assume that $E > \Hm,$ and let us choose $$\beta := S'(E),$$
where the function $S$ has been introduced in paragraph \ref{entropy}. By the very definition of $S$, the function $S'$ is the compositional inverse of the function $U := - \Psi' = - Z'/Z$, and $\beta$ is the unique zero in $\Rpa$ of the derivative
$$\frac{d}{ds} [\Psi(s) + E s] = \Psi'(s) + E.$$ 
Moreover,
$$S(\beta) = \Psi(\beta) + E \beta.$$

Accordingly, when $s \in \C$ goes to $\beta$, we may write:
%$$Z(s)\, e^{Es} = e^{\Psi(s) + Es} = e^{\Psi(\beta) + E \beta +  \Psi''(\beta) (s-\beta)^2/2 + o((s-\beta)^2)} = Z(\beta)\, e^{E\beta} e^{\Psi''(\beta) (s-\beta)^2/2 + o((s-\beta)^2)}.$$
$$Z(s)\, e^{Es} = e^{\Psi(s) + Es} = e^{\Psi(\beta) + E \beta +  \Psi''(\beta) (s-\beta)^2/2 + o((s-\beta)^2)} =  e^{S(\beta) + \Psi''(\beta) (s-\beta)^2/2 + o((s-\beta)^2)}.$$
(The analytic function $\Psi:= \log Z$ is well defined on some open neighbourhood of $\Rpa$ in $\C_{>0}$.)   

This immediately shows that $F_\beta$ satisfies Condition $\bf L_2$  with
$$F_\beta(0) = e^{S(\beta)} \quad \mbox{ and } \quad \alpha = (1/2) \Psi''(\beta).$$
Moreover,
$$g_\beta(0) = (2\pi\beta)^{-1}.$$ 

We may therefore apply Proposition \ref{PropLap} to the integrals (\ref{AnInt}). The asymptotic expression (\ref{INeq}) for these integrals given by Laplace's method is the announced expression  (\ref{AnAsPoinc}).
\end{proof}

\subsubsection{The approach of Darwin-Fowler}\label{AsDF}

Let us assume in this paragraph that Condition $\bf{T}_2$ is satisfied. 

The following proposition follows from the fact that a Radon measure on $\R_+$ which satisfies Condition $\Sigma_0$ (see paragraph \ref{LaplaceTr}) is uniquely determined by its Laplace transform on the half-plane $\C_{>0}$.

\begin{proposition}\label{Dfeq}
For any $\eta \in \Rpa,$ the following three conditions are equivalent:

$\mathbf{DF}^1_\eta :$ For $\mu$-almost every $x \in \cE,$ $H(x)$ belongs to $\N \eta$;

$\mathbf{DF}^2_\eta :$ the measure $H_\ast \mu$ is supported by $\N \eta$;

$\mathbf{DF}^3_\eta :$ the partition function $Z: \C_{>0} \lra \C$ is $2\pi i/\eta$-periodic.  \qed

\end{proposition}

Condition $\mathbf{DF}$ is equivalent to the existence of $\eta$ in $\Rpa$ such that these conditions $\mathbf{DF}_\eta^{1-3}$ are satisfied.

Actually,  the obviously equivalent conditions $\mathbf{DF}^1_\eta$ and $\mathbf{DF}^2_\eta$ are satisfied if and only if we may write the measure $H_\ast \mu$ as:
\begin{equation}\label{Hast}
H_\ast \mu = \sum_{k \in \N} h_k \, \delta_{k \eta}
\end{equation}
for some sequence $(h_k)_{k \in \N}$ in $\Rpa$. When this holds,  we have:
$$Z(s)= \sum_{k \in \N} h_k \, e^{-k \eta s} \quad \mbox{for any $s\in \C_{>0}$}.$$
Then the series with non-negative coefficients
\begin{equation}\label{defh}
f(X) := \sum_{k \in \N} h_k X^k
\end{equation}
has radius of convergence at least 1 (this is  a reformulation of Condition  $\bf{T}_2$), and we have:
\begin{equation}
Z(s) = f(e^{-\eta s}) \quad \mbox{ for any $s \in \C_{>0}$.}
\end{equation}
This makes clear the validity of Condition $\mathbf{DF}^3_\eta$.

Observe also that
$$\mu(\cE) = \sum_{k \in \N} h_k = \lim_{q \ra 1_-} f(q).$$
Condition $\bf{T}_1$ holds if and only if this limit is $+ \infty$.

These observations show that, when conditions $\bf{T}_1$, $\bf{T}_2$, and $\bf{DF}$ are satisfied, there exists a smallest $\eta \in \Rpa$ such that the conditions $\bf{DF}^{1-3}_\eta$ are satisfied. we shall denote it by $\eta_H$. 
Then the analytic function $f$ on $D(0,1)$ such that
\begin{equation}\label{deff}
Z(s) = f(e^{-\eta_H s}) \quad \mbox{ for any $s \in \C_{>0}$.}
\end{equation}
is defined by the series (\ref{defh})  where the $(h_k)_{k \in \N}$ is defined by the relation (\ref{Hast}) with $\eta = \eta_H.$
Moreover, by the very definition of $\eta_H,$ for every integer $n >1$, we have:
$$\{ k \in \N \mid h_k \neq 0 \} \nsubseteq n \N.$$

This immediately implies:
\begin{lemma}\label{lemZmax} For every 
$s \in \C_{>0},$ the inequality (\ref{modZ}) 
$$\vert Z(s) \vert \leq Z({\rm Re} s)$$
is an equality if and only if
$s$ belongs to $\Re s + (2\pi i/\eta_H)\, \Z,$
or equivalently if and only if the element
$$q := e^{- \eta_H s}$$
of the pointed unit disc $D(0, 1)\setminus \{0\}$ belongs to the interval $(0,1).$ \qed 
\end{lemma}

For any $r \in (0,1)$, we shall denote by $C(r)$ the closed path in the complex plane
$$([0,1] \lra \C, \, t \longmapsto r e^{2 \pi i t}).$$

\begin{proposition}\label{PropDF} Let us assume that Conditions $\bf{T}_1$, $\bf{T}_2$, and $\bf{DF}$ hold. Let $r \in (0,1)$ and let $$\beta := \eta_H^{-1} \, \log r^{-1}.$$

Then, for every $E \in \R$ and any integer $n \geq 1$ such that $n E \in \Z \eta_H,$ we have:
\begin{align}
A_n(E) & = \frac{1}{2\pi i} \int_{C(r)} (1-q)^{-1}  q^{- n E/\eta_H} \, f(q)^n \, q^{-1}\, dq \label{DF1} \\
& = \frac{1}{2\pi i} \int_{\beta-\pi i/\eta_H}^{\beta +\pi i/\eta_H} \frac{\eta_H}{1- e^{-\eta_H s}} \left[Z(s) e^{Es}\right]^n \, ds \label{DF2} \\
& =\frac{1}{2\pi} \int_{-\pi /\eta_H}^{\pi /\eta_H} \frac{\eta_H}{1- e^{-\eta_H (\beta + it)}} \left[Z(\beta + it ) e^{E(\beta + it)}\right]^n  \, dt. \label{DF3}
\end{align}
\end{proposition}

\begin{proof} Let us write, as above:
$$H_\ast \mu = \sum_{k \in \N} h_k \delta_{k \eta_H}.$$
Then the measure
$$H_{n \ast}(\mu^{\otimes n}) = (H_\ast \mu)^{\ast n}:= (H_\ast \mu) \ast \ldots \ast (H_\ast \mu) \quad \mbox{ ($n$-times)}$$
may be written
$$H_{n \ast}(\mu^{\otimes n}) = \sum_{k \in \N} h_k^{[n]}\,  \delta_{k \eta_H},$$
where the sequence $(h_k^{[n]})_{k \in \N}$ satisfies:
$$\sum_{k \in \N} h^{[n]}_k X^k = \left(\sum_{k \in \N} h_k X^k\right)^n = f(X)^n.$$

When $\tilde{n} := nE/\eta_H$ is an integer, we have:
$$A_n(E) = H_{n \ast}(\mu^{\otimes n})([0, nE]) = \sum_{0 \leq k \leq \tilde{n}} h_k^{[n]} = \Res_0 [(1-X)^{-1} X^{-(\tilde{n} +1)} f(X)^n],$$
where we denote by $\Res_0$ the residue at $0$. By the residue formula, this coincides with the right-hand side of (\ref{DF1}). 

We deduce (\ref{DF2}) and (\ref{DF3}) from (\ref{DF1}) by the changes of variables
$$q = e^{- \eta_H s} \quad \mbox{ and } \quad s =\beta + it.$$
\end{proof}

By applying Laplace's method to the integral formulae established in Proposition \ref{PropDF}, it is possible to derive an asymptotic expression for $A_n(E)$ similar to the one in Theorem \ref{ThMainPoinc}:
\begin{theorem}\label{ThMainDF} 
 Let us assume that Conditions $\bf{T}_1$, $\bf{T}_2$, and $\bf{DF}$ hold. Let us consider
 $$E \in \Q\, \eta_H \cap (\Hm, +\infty),$$
 and let us define\footnote{If $E =  \eta_H a/b$ for some integer $a$ and $b$ prime together, then $\mathcal{N}(E) = \vert b \vert \Z_{>0}.$}
 $$\mathcal{N}(E) := \{ n \in \Z_{>0} \mid n E \in \Z_{>0} \eta_H \}.$$
 Then, when the integer $n \in \mathcal{N}(E)$ goes to infinity, we have:
\begin{equation}\label{AnAsDF}
A_n(E) \sim \frac{\eta_H \beta}{1- e^{-\eta_H \beta}} \left[\pi \beta^2 \Psi''(\beta)n\right]^{-1/2} \, e^{n S(E)},
\end{equation}
with $\beta := S'(E)$.
\end{theorem}

%The asymptotic expression (\ref{AnAsDF}) clearly implies that, for any $E \in \Q\, \eta_H \cap (\Hm, +\infty),$
%%$$\lim_{n \in \cN(E), n \ra +\infty} $$
%$$\lim_{\stackrel{n \in \cN(E)}{ n \ra +\infty}} \frac{1}{n} \log A_n(E) = S(E).$$
%
%Combined with Proposition \ref{LanfProp}, which was established by a simple argument relying on Lanford's estimates, this implies, for any $E \in (\Hm, +\infty),$ the validity of the limit formula (\ref{Slim}):  
%$$\lim_{ n \ra + \infty} \frac{1}{n} \log A_n(E) = S(E).
%$$

\begin{proof}%[Proof of Proposition \ref{ThMainDF}] 
For any $n$ in  $\cN(E)$ and any $\beta$ in $\Rpa$, according to  (\ref{DF3}),$A_n(E)$ admits the expression (\ref{DF3}):
\begin{equation}\label{DFLap}
A_n(E)  =\frac{1}{2\pi} \int_{-\pi /\eta_H}^{\pi /\eta_H} \frac{\eta_H}{1- e^{-\eta_H (\beta + it)}} \left[Z(\beta + it ) e^{E(\beta + it)}\right]^n  \, dt = \int_I g_\beta(t) F_\beta(t)^n \, dt,
\end{equation}
where: 
$$I := [-\pi/\eta_h, \pi/\eta_H], \quad g_\beta(t) := \frac{\eta_H}{1-e^{-\eta_H (\beta + i t)}}, \quad \mbox{and} \quad  F_\beta(t) := Z(\beta + it) e^{E(\beta + it)}.$$

The functions $g_\beta$ and $F_\beta$ clearly satisfy the condition $\mathbf{L}_1$ on the functions $g$ and $F$ in our discussion of Laplace's method in paragraph \ref{AsymLapl}. According to Lemma \ref{lemZmax}, the function 
$F_\beta$ also satisfies Condition $\mathbf{L}_3$. Moreover, if we choose $\beta := S'(E)$, $F_\beta$ also satisfies Condition $\mathbf{L}_2,$ as already shown in the proof of Theorem \ref{ThMainPoinc}.

The asymptotic expression (\ref{AnAsPoinc}) therefore follows from Proposition \ref{PropLap} applied to the integrals (\ref{DFLap}).
\end{proof}

Besides the original articles (\cite{DarwinFowler1922},
\cite{DarwinFowler1922II}, \cite{DarwinFowler1923}) the results of Darwin-Fowler are presented in the reference text by Fowler (\cite{Fowler36}, Chapter 2), in the beautiful introductory notes by Schr\"odinger (\cite{Schroedinger52}, Chapter 4), and in the textbook of Huang (\cite{Huang87}, Section 9.1).

\subsection{Some approximation arguments}\label{Approx}
 The asymptotic equivalents (\ref{AnAsPoinc}) and (\ref{AnAsDF}) for $A_n(E)$, established under the additional assumptions ${\mathbf L^2_\epsilon}$ and $\mathbf{DF}$ on the pair $((\cE, \cT, \mu), H)$ in Theorems \ref{ThMainPoinc} and \ref{ThMainDF}, both  imply the limit formula 
 \begin{equation}\label{keylimagain}
 \lim_{ n \ra + \infty} \frac{1}{n} \log A_n(E) = S(E),
 \end{equation} %(\ref{Slim})
  which was the key point in our derivation of Theorem \ref{ThMain}.
  
  Indeed, when ${\mathbf L^2_\epsilon}$ holds, (\ref{AnAsPoinc}) clearly implies this limit formula. When $\mathbf{DF}$ holds, it follows from the asymptotic expression (\ref{AnAsDF}) combined with the existence in $(-\infty, +\infty]$ of the limit 
  $\lim_{ n \ra + \infty} (1/n) \log A_n(E)$, established in Subsection \ref{LanfEst} by Lanford's method, and the fact that this limit is a non-decreasing function of $E$.
  
  In this subsection, we present some simple approximation arguments that will allow us to derive (\ref{keylimagain}) in the general context of Theorem \ref{ThMain} from its special cases implied by Theorem \ref{ThMainPoinc} or Theorem \ref{ThMainDF}. (This is clear when 
  
  As already observed in \ref{RedSepcR}, to prove the validity of Theorem \ref{ThMain}, one immediately reduces to the case where $(\cE, \cT)$ is $(\R_+, \cB)$ --- where $\cB$ denotes the $\sigma$-algebra of Borel subsets of $\R_+$ --- and where the Hamiltonian function $H$ is $\rm{Id}_{\R_+}$. 
  
From now one, we place ourselves in this framework; namely, we consider a positive Radon measure $\mu$ on $\R_+$ which satifisfies the conditions
$${\mathbf T_1}: \quad  \mu(\R_+) = + \infty,$$ and $${\mathbf T_2}: \quad Z(\beta) := \cL \mu(\beta) :=\int_{\R_+} e^{- \beta E}\, d\mu(E) < + \infty \quad \mbox{ for any $\beta \in \Rpa.$} $$

\subsubsection{Approximation by convolution} Let us choose $\chi \in C_c^\infty(\R)$ such that
$$\chi(\R) \subset \R_+, \quad \supp \chi \subset [0,1], \mbox{ and }  \int_\R \chi(x) \, dx =1.$$
For any $\eta \in \Rpa,$ we let
$$\chi_\eta := \eta^{-1} \,\chi(\eta^{-1} .)$$
and
$$\mu_\eta := \mu \ast \chi_\eta.$$

The validity of ${\mathbf T_1}$ and ${\mathbf T_2}$ for $\mu$ immediately implies their validity for $\mu_\eta$.

To the pair $((\R_+, \cB, \mu), {\rm Id}_{\R_+})$ are associated the non-negative real number $\Hm$, and the functions $Z$, $\Psi$ and $S$,  and the sequence of functions $(A_n)_{n \geq 1}$. 

Similarly, for any $\eta \in \Rpa,$ to $((\R_+, \cB, \mu), {\rm Id}_{\R_+})$ are associated the essential minimum
$\Hmeta$ of ${\rm Id}_{\R_+}$ with respect to $\mu_\delta$ (or equivalently, the minimum of $\supp \mu_\eta$), and functions
$Z_\eta$, $\Psi_\eta$  and $S_\eta$,  and the sequence of functions $(A_{n, \eta})_{n \geq 1}$
 defined as follows: for any $s \in \C_{>0}$,
$$Z_\eta(s) := \int_{\R_+} e^{-sE} \, d\mu_\eta (E);$$ 
 for any $\beta \in \Rpa$, 
$$\Psi_\eta(\beta) := \log Z_\eta(\beta);$$ 
for any $E \in (\Hmeta, +\infty),$
$$S_\eta (E) := \sup_{\beta \in \Rpa} (\Psi_\eta(\beta) + \beta E);$$
and for any positive integer $n$ and any $E\in \R_+,$
$$A_{n, \eta}(E) := \mu_\eta^{\otimes n} (\{(x_1, \ldots, x_n) \in \R_+^n \mid x_1+ \ldots + x_n \leq nE \}) = \mu_\eta^{\ast n} ([0, nE]).$$

\begin{lemma}\label{ApproxHA}
 For any $\eta \in \Rpa$ and any $E \in \R_+,$ we have:
\begin{equation}\label{InHeta}
\Hm \leq \Hmeta \leq \Hm + \eta
\end{equation}
and
\begin{equation}\label{InAeta}
A_{n, \eta} (E)\leq A_n(E) \leq A_{n, \eta} (E + \eta).
\end{equation}
\end{lemma}

\begin{proof} The estimates (\ref{InHeta}) are straightforward. The estimates (\ref{InAeta}) follow from the identities:
$$A_{n, \eta} (E) = \mu_\eta^{\ast n}([0, nE]) = \int_{x,y \geq 0, x+y \leq nE} d\mu^{\ast n} (x). \chi^{\ast n}_\eta(y) \, dy$$
$$A_n(E) = \mu^{\ast n}([0, nE]) = \int_{0 \leq x \leq nE} d\mu^{\ast n} (x)$$
and
$$A_{n, \eta} (E+ \eta) = \mu_\eta^{\ast n}([0, nE+ n\eta]) = \int_{x,y \geq 0, x+y \leq nE+ n\eta} d\mu^{\ast n} (x). \chi^{\ast n}_\eta(y) \, dy,$$
and from the fact that $\chi^{\ast n}_\eta$ is non-negative, supported by $[0, n \delta]$, and of integral $1$. 
\end{proof}

\begin{lemma}\label{ApproxPsi}
 For any $(\eta, \beta) \in \R_+^{\ast 2},$
 \begin{equation}\label{app1}
 \Psi(\beta) - \eta \beta \leq \Psi_\eta(\beta) \leq \Psi(\beta).
 \end{equation}
\end{lemma}

\begin{proof} From the multiplicativity property (\ref{LaplMult}) of the Laplace transform, we obtain that, for any $s \in \C_{>0}$ :
$$Z_\eta(s) = \cL \mu_\eta(s) = \cL \mu(s) \, \cL \chi(s) = Z(s) \,  \int_0^{+\infty} e^{-sE} \chi_\eta(E) \, dE.$$
Therefore, for any $\beta \in \Rpa$:
\begin{equation}\label{app2}\Psi_\eta(\beta) = \Psi(\beta) + \log \int_0^{+\infty} e^{-sE} \chi_\eta(E) \, dE.
 \end{equation}

Besides, as $\chi_\eta$ is non-negative, supported by $[0, \eta]$ and of integral $1$, we have:
\begin{equation}\label{app3}
e^{-\beta \delta} \leq \int_0^{+\infty} e^{-sE} \chi_\eta(E) \, dE \leq 1.
 \end{equation}
 
 The estimates (\ref{app1}) follow from (\ref{app2}) and (\ref{app3}).
\end{proof}

From Lemma \ref{ApproxHA} and \ref{ApproxPsi}, one easily derives:

\begin{lemma}\label{ApproxFin}
If, for any $\eta \in \Rpa$ and any $E \in (\Hmeta, +\infty)$,
\begin{equation}\label{limSeta}
\lim_{ n \ra + \infty} \frac{1}{n} \log A_{n,\eta}(E) = S_\eta(E),
\end{equation}
then, for any $E \in (\Hm, +\infty),$
\begin{equation}\label{limS}
\lim_{ n \ra + \infty} \frac{1}{n} \log A_n(E) = S(E).
\end{equation}
\end{lemma}

\begin{proof} From Lemma \ref{ApproxPsi}, we immediately derive:
\begin{equation}\label{InSeta2}
S(E-\eta) \leq S_\eta(E)  \quad \mbox{ for any $\eta >0$ and any $E> \Hm + \eta$,}
\end{equation}
and
\begin{equation}\label{InSeta1}
S_\eta(E) \leq S(E) \quad \mbox{ for any $\eta >0$ and any $E> \Hmeta$.}
\end{equation}

Let us now consider $E \in (\Hm, +\infty)$. For any $\eta \in (E -\Hm),$ from  (\ref{InAeta})  and (\ref{limSeta}), we get:
$$S_\eta(E) = \lim_{ n \ra + \infty} \frac{1}{n} \log A_{n,\eta}(E) \leq \liminf_{ n \ra + \infty} \frac{1}{n} \log A_n(E)$$
and
$$\limsup_{ n \ra + \infty} \frac{1}{n} \log A_n(E) \leq \lim_{ n \ra + \infty} \frac{1}{n} \log A_{n,\eta}(E+\eta) = S_\eta(E+ \eta).$$
Together with (\ref{InSeta2}) and (\ref{InSeta1}), this shows:
$$S(E-\eta)  \leq \liminf_{ n \ra + \infty} \frac{1}{n} \log A_n(E) \leq \limsup_{ n \ra + \infty} \frac{1}{n} \log A_n(E) \leq S (E+ \eta).$$

Since the function $S$ is continuous on $(\Hm, +\infty)$, this establishes (\ref{limS}) by taking the limit when $\eta$ goes to zero.
\end{proof} 

Lemma \ref{ApproxFin} allows one to derive the validity of the limit formula (\ref{limS}) from its validity when furthermore Condition ${\mathbf L^2_\epsilon}$ holds. Indeed, for any $\eta \in \Rpa,$ the pair $((\R_+, \cB, \mu_\eta), \rm{Id}_{\R_+})$ satisfies this condition, which ensures the validity of (\ref{limSeta}); in other words, the measure $\mu_\eta$ is absolutely continuous with respect to the Lebesgue measure $\lambda$ and may be written 
$$\mu_\eta = f_\eta \, \lambda$$
where, for every $\epsilon \in \Rpa,$ the function $(x \longmapsto e^{- \epsilon x} f_\eta(x))$ is in $L^2(\R_+)$. 

Actually, the density $f_\eta$ is a $C^\infty$ function on $\R_+$, and we have, for any $x \in \Rpa$:
$$f_\eta(x) = \int_\R \chi_\eta (x-t) \, d\mu(t) = \eta^{-1} \int_{x-\eta}^x \chi(\eta^{-1} (x-t)) \, d\mu(t) \leq \eta \Vert \chi\Vert_{L^\infty} \, \mu([0, x]).$$
This shows that, for any $\epsilon \in \Rpa,$
$$f_\eta(x) = O(e^{\epsilon x}) \quad \mbox{ when $x \lra + \infty$.}$$

\subsubsection{Approximation by discretization}

For any $\eta \in \Rpa,$ we may also introduce the Radon measure
$$\mu_\eta := \sum_{k \in \Z_{>0}} \mu([(k-1)\eta, k\eta)) \, \delta_{k \eta}.$$

As in the previous paragraph, to the measure $\mu_\eta$ are associated its essential minimum $H_{\rm{min}, \eta}$, and the functions $Z_\eta,$ $\Psi_\eta,$ $S_\eta,$ and $A_{n,\eta}$.
Then Lemma   \ref{ApproxHA} and \ref{ApproxPsi} remain valid  (we leave the details as an exercice for the interested reader), and consequently Lemma \ref{ApproxFin} also.

By construction, $\mu_\eta$ is supported by $\N \eta$, and therefore  the pair $((\R_+, \cB, \mu_\eta), \rm{Id}_{\R_+})$ satisfies Condition $\bf{DF}$. Thus the fact that  Lemma \ref{ApproxFin} holds in the present context 
shows that the validity of the limit formula (\ref{limS}) follows from its validity when furthermore Condition  $\bf{DF}$ holds.

Let us finally remark that, when we consider a Euclidean lattice $\Eb$ defined by some integral quadratic form, the associated measure space with Hamiltonian $((\cE, \cT, \mu), H)$ associated to $\Eb$ by the construction in Subsection \ref{AppLatt} never satisfies Condition ${\mathbf L^2_\epsilon}$. It satisfies Condition $\mathbf{DF}$ if and only if some positive real multiple of  $\Vert.\Vert^2$ is an integral quadratic form on the free $\Z$-module $E$. Accordingly, in this case, the validity of Theorem \ref{Thhonthot} directly follows from the asymptotics \emph{à la} Darwin-Fowler established in paragraph \ref{AsDF} (see \cite{MazoOdlyzko90}, Section 3, for a related discussion).

%Besides, the pair $((E, \sum_{v \in E} \delta_v), \Vert.\Vert^2)$ attached to some Euclidean lattice $\Eb := (E,\Vert.\Vert)$ as in  \ref{AppLatt} never satifies Condition ${\mathbf L^2_\epsilon}$. It satisfies Condition $\mathbf{DF}$ if and only if some positive real multiple of  $\Vert.\Vert^2$ is an integral quadratic form on the free $\Z$-module $E$.

%allows one to derive the validity of the limit formula (\ref{limS}) from its validity when furthermore Condition  $\bf{DF}$ holds. As explained in paragraph \ref{LanfEst}, the latter is a consequence of Proposition \ref{LanfProp} and Theorem \ref{ThMainDF}.

%\bibliographystyle{alpha}

%\bibliography{REFThermoLattices2019}

\end{document}